\documentclass[10pt,a4paper,twoside,reqno]{amsart}

\usepackage[a4paper,
left=1in,right=1in,top=1in,bottom=1.2in,%
footskip=.25in]{geometry}

\usepackage{bm}
\usepackage{upgreek}
\usepackage{latexsym, amsmath, amstext, amssymb, amsfonts, amscd, bm, array, multirow, amsbsy, mathrsfs}
\usepackage{amsthm, amsbsy}
\usepackage{t1enc}
\usepackage[mathscr]{eucal}
\usepackage{indentfirst}
\usepackage{pb-diagram}
\usepackage{graphicx}
\usepackage{fancyhdr}
\usepackage{fancybox}
\usepackage{enumerate}
\usepackage{color}
\usepackage{tikz-cd}
\usepackage[all]{xy}
\usepackage{hyperref}
\usepackage{tikz}
\usepackage{xparse}
\usepackage{bbm}
\usepackage[shortlabels]{enumitem}
\hypersetup{colorlinks=false,pdfborderstyle={/S/U/W 0}}
\usetikzlibrary{matrix}

\newcommand{\arxiv}[1]{{\tt
		\href{http://www.arXiv.org/abs/#1}{arXiv:#1}}}

\usepackage{url}
\usepackage[sort&compress,comma]{natbib}
\bibpunct{[}{]}{,}{n}{}{,}
\theoremstyle{plain}
\newtheorem{thm}{Theorem}[section]

\newtheorem{prop}[thm]{Proposition}

\newtheorem{lemma}[thm]{Lemma}
\newtheorem{cor}[thm]{Corollary}

\theoremstyle{definition}
\newtheorem{definition}[thm]{Definition}

\theoremstyle{remark}
\newtheorem{remark}[thm]{Remark}

\newtheorem*{ack}{Acknowledgements}

\def\ad{\mathrm{ad}}

\newcommand{\bDelta}{{\boldsymbol{\Delta}}}
\newcommand{\bXi}{\boldsymbol{\Xi}}
\newcommand{\im}{\mathrm{Im}}

\newcommand{\End}{\mathrm{End}}
\newcommand{\Aut}{\mathrm{Aut}}

\newcommand{\Mat}{\mathrm{Mat}}

\newcommand{\eqdef}{\stackrel{{\rm def.}}{=}}

\newcommand{\frt}{\mathfrak{t}}

\DeclareFontFamily{U}{rsf}{}
\DeclareFontShape{U}{rsf}{m}{n}{<5> <6> rsfs5 <7> <8> <9> rsfs7 <10-> rsfs10}{}
\DeclareMathAlphabet\Scr{U}{rsf}{m}{n}

\def\grad{\mathrm{grad}}
\def\curl{\mathrm{curl}}
\def\div{\mathrm{div}}

\def\Z{\mathbb{Z}}
\def\C{\mathbb{C}}
\def\R{\mathbb{R}}

\def\H{\mathbb{H}}

\def\rk{{\rm rk}}

\def\GL{\mathrm{GL}}

\def\dd{\mathrm{d}}

\def\vol{\mathrm{vol}}

\def\l\Xi{\overrightarrow{\Xi}}
\def\r\Xi{\overleftarrow{\Xi}}
\def\tf{{\tilde f}}

\def\l{\partial^l}
\def\Stab{\mathrm{Stab}}

\def\Ad{\mathrm{Ad}}

\def\ad{\mathrm{ad}}

\def\Diff{\mathrm{Diff}}
\def\Conf{\mathrm{Conf}}
\def\Sol{\mathrm{Sol}}
\def\Aff{\mathrm{Aff}}

\def\Met{\mathrm{Met}}
\def\hol{\mathrm{hol}}
\def\htheta{\hat{\theta}}
\def\hpsi{\hat{ u}}

\def\pot{\mathrm{pot}}
\def\EMC{\mathrm{EMC}}
\def\tf{\mathrm{tf}}

\newcommand{\be}{\begin{equation*}}
\newcommand{\ee}{\end{equation*}}
\newcommand{\ben}{\begin{equation}}
\newcommand{\een}{\end{equation}}
\newcommand{\beqa}{\begin{eqnarray*}}
	\newcommand{\eeqa}{\end{eqnarray*}}
\newcommand{\beqan}{\begin{eqnarray}}
\newcommand{\eeqan}{\end{eqnarray}}

\newcommand{\nn}{\nonumber}

\newcommand{\twopartdef}[4]
{
	\left\{
	\begin{array}{ll}
		#1 & \mbox{ if } #2 \\
		#3 & \mbox{ if } #4
	\end{array}
	\right .
}

\newcommand{\id}{\mathrm{id}}

\def\cR{{\mathcal R}}
\def\hcR{{\hat \cR}}
\def\hcI{{\hat \cI}}
\def\hcN{{\hat \cN}}
\def\hF{{\hat F}}
\def\cC{{\mathcal C}}

\def\Spin{\mathrm{Spin}}

\def\Spin{\mathrm{Spin}}

\def\i{\mathbf{i}}

\def\SO{\mathrm{SO}}

\def\U{\mathrm{U}}
\def\Hol{\mathrm{Hol}}
\def\cD{\mathcal{D}}
\def\cJ{\mathcal{J}}

\def\cA{\mathcal{A}}
\def\cE{\mathcal{E}}
\def\cI{\mathcal{I}}

\def\cN{\mathcal{N}}

\def\cT{\mathcal{T}}
\def\cF{\mathcal{F}}
\def\cC{\mathcal{C}}

\def\Sp{\mathrm{Sp}}
\def\G_2{\mathrm{G_2}}

\def\cL{\mathcal{L}}
\def\cS{\mathcal{S}}

\def\cV{\mathcal{V}}

\def\P{\mathbb{P}}

\newcommand{\Hom}{{\rm Hom}}

\newcommand{\Iso}{{\rm Iso}}

\def\Aut{\mathrm{Aut}}
\def\Ob{\mathrm{Ob}}

\def\Fr{\mathrm{Fr}}

\def\Re{\mathrm{Re}}
\def\Im{\mathrm{Im}}
\def\im{\mathrm{im}}

\def\G{\mathrm{G}}

\def\T{\mathbb{T}}
\def\R{\mathbb{R}}

\def\rS{\mathrm{S}}

\def\Prin{\mathrm{Prin}}

\def\hP{\hat{P}}

\def\A{\mathbb{A}}

\def\cL{\mathcal{L}}

\def\dd{\mathrm{d}}

\def\mC{\mathrm{C}}
\def\cl{\mathrm{cl}}
\def\bT{\mathbf{T}}
\def\Sieg{\mathrm{Sieg}}
\def\Dual{\mathrm{Dual}}
\def\curv{\mathrm{curv}}
\def\rT{\mathrm{T}}
\def\rS{\mathrm{S}}
\def\Symp{\mathrm{Symp}}

\def\A{\mathbb{A}}
\def\Div{\mathrm{Div}}
\def\homega{\hat{\omega}}

\def\per{\mathrm{per}}

\def\aff{\mathfrak{aff}}

\def\bgamma{\boldsymbol{\gamma}}
\def\Sg{\mathrm{Sg}}

\def\cConf{\mathfrak{Conf}}
\def\cSol{\mathfrak{Sol}}
\def\cU{\mathcal{U}}
\def\Conn{\mathrm{Conn}}
\def\fg{\mathfrak{g}}
\def\fa{\mathfrak{a}}
\def\cl{{\mathrm{cl}}}
\def\ex{{\mathrm{ex}}}
\def\fl{{\mathrm{flat}}}
\def\Tor{\mathrm{Tor}}

\def\lcm{\mathrm{lcm}}
\def\gcd{\mathrm{gcd}}

\def\fT{\mathbb{T}}
\def\fA{\mathbb{A}}
\DeclareSymbolFont{bbold}{U}{bbold}{m}{n}
\DeclareSymbolFontAlphabet{\mathbbold}{bbold}
\def\bfA{\mathbbold{A}}
\def\fc{\mathfrak{c}}
\def\T{\mathfrak{T}}

\def\hM{\hat{M}}
\def\hDelta{\hat{\Delta}}
\def\hcS{\hat{\cS}}
\def\hcD{{\hat{\cD}}}

\def\bc{\mathbf{c}}

\def\hP{{\hat P}}
\def\hcA{{\hat \cA}}

\def\hP{{\hat P}}

\def\hcJ{\hat{\cJ}}

\def\fJ{\mathfrak{J}}

\def\SympRep{\mathrm{SympRep}}
\def\fR{\mathfrak{R}}

\def\bpi{\boldsymbol{\pi}}

\def\bA{\mathbf{A}}

\def\hc{{\hat c}}
\def\hpi{\hat{\pi}}
\def\tf{\mathrm{tf}}
\def\Dyons{\mathrm{Dyons}}
\def\dex{{\dd_{\cD}\!\mbox{-}\ex}}
\def\dcl{{\dd_{\cD}\!\mbox{-}\cl}}

\def\bP{\mathbf{P}}
\def\rB{\mathrm{B}}

\def\Per{\mathrm{Per}}
\def\disc{\mathrm{disc}}

\def\brho{\bar{\rho}}

\usepackage{enumitem}
\setlist[itemize]{leftmargin=*}

\newcolumntype{P}[1]{>{\centering\arraybackslash}p{#1}}

\begin{document}

\title[The duality covariant geometry of abelian gauge theory]{The
duality covariant geometry and DSZ quantization of abelian gauge theory}

\author[C. Lazaroiu]{C. Lazaroiu} \address{Department of Theoretical
Physics, Horia Hulubei National Institute for Physics and Nuclear
Engineering, Bucharest-Magurele, Romania}
\email{lcalin@theory.nipne.ro}

\author[C. S. Shahbazi]{C. S. Shahbazi} \address{Department of
Mathematics, University of Hamburg, Germany}
\email{carlos.shahbazi@uni-hamburg.de}

\thanks{2010 MSC. Primary: 53C80. Secondary: 53C07.}
\keywords{Abelian gauge theory, Electromagnetic duality, symplectic vector 
	bundles, character varieties, symplectic lattices}

\begin{abstract}
We develop the Dirac-Schwinger-Zwanziger (DSZ) quantization of
classical abelian gauge theories with general duality structure on
oriented and connected Lorentzian four-manifolds $(M,g)$ of arbitrary
topology, obtaining as a result the duality-covariant geometric formulation 
of such theories through connections on principal bundles. We implement 
the DSZ condition by restricting the field strengths of the theory 
to those which define classes originating in the degree-two cohomology 
of a local system valued in the groupoid of integral symplectic spaces. 
We prove that such field strengths are curvatures of connections $\cA$ 
defined on principal bundles $P$ whose structure group $G$ is the 
disconnected {\em non-abelian} group of automorphisms of an integral 
affine symplectic torus. The connected component of the identity of 
$G$ is a torus group, while its group of connected components is a
modified Siegel modular group which coincides with the group of local
duality transformations of the theory. This formulation includes
electromagnetic and magnetoelectric gauge potentials on an equal
footing and describes the equations of motion through a first-order
{\em polarized self-duality condition} for the curvature of $\cA$. 
The condition involves a combination of the Hodge operator of
$(M,g)$ with a taming of the duality structure determined by $P$,
whose choice encodes the self-couplings of the theory. This
description is reminiscent of the theory of four-dimensional euclidean
instantons, even though we consider a two-derivative theory in Lorentzian 
signature. We use this formulation to characterize the hierarchy of duality 
groups of abelian gauge theory, providing a gauge-theoretic description of 
the electromagnetic duality group as the discrete remnant of the gauge 
group of $P$. We also perform the time-like reduction of the polarized 
self-duality condition to a Riemannian three-manifold, obtaining a new type 
of Bogomolny equation which is modified by the given taming and duality 
structure induced by $P$. We give explicit examples of such solutions, 
which we call \emph{polarized dyons}.
\end{abstract}

\maketitle

\setcounter{tocdepth}{1} 
\tableofcontents


\section*{Introduction}


Abelian gauge theory on Lorentzian four-manifolds is a natural
extension of Maxwell electrodynamics, which locally describes a finite
number of abelian gauge fields interacting through couplings which are
allowed to vary over space-time. Such theories occur frequently in
high energy physics. For instance, the low energy limit of a
non-abelian gauge theory coupled to scalar fields which can be
maximally higgsed contains an abelian gauge theory sector; this occurs
in particular on the Coulomb branch of supersymmetric non-abelian gauge
theories. Moreover, the universal bosonic sector of four-dimensional
ungauged supergravity involves a fixed number of abelian gauge fields
interacting with each other through couplings which can vary over
space-time due to their dependence on the scalars of the theory. This
sector of four-dimensional supergravity can be described by pulling
back an abelian gauge theory defined on the target space of the sigma
model of the scalar fields \cite{gesm}.

The {\em local} behavior of abelian gauge theories (including their
supersymmetric extensions) was studied intensively in the physics
literature, where the subject has achieved the level of textbook
material \cite{Cecotti}. Despite intense activity, a global geometric
formulation of such theories on arbitrary spacetimes is still missing.
At the level of field strengths, such a description was given in
\cite{gesm} (see \cite{gesmproc} for a summary) in the wider context
of the geometric description of the universal sector of classical
supergravity in four dimensions. As explained there and recalled in
Section \ref{sec:classical}, the global formulation requires the
specification of a ``duality structure'', defined as a flat symplectic
vector bundle $\Delta=(\cS,\omega,\cD)$ on the spacetime manifold
$M$. The even rank $2n$ of $\cS$ equals the number of field strengths,
where {\em both} electromagnetic and magnetoelectric fields are
included. When the spacetime is not simply connected, such a bundle
need not be trivial and it `twists' the local formulation in such a
way that the combination of all electromagnetic and magnetoelectric
field strengths can be described globally by a $\dd_\cD$-closed
two-form $\cV$ defined on $M$ and valued in $\cS$, where $\dd_\cD$ is
the differential induced by the flat connection $\cD$ of the duality
structure. A classical electromagnetic field strength {\em
configuration} is a $\dd_\cD$-closed two-form $\cV$.  The classical
equations of motion are encoded by the condition that $\cV$ be
self-dual with respect to a `polarized Hodge operator' obtained by
tensoring the Hodge operator $\ast_g$ defined by the Lorentzian
spacetime metric $g$ with a (generally non-flat) taming $\cJ$ of the
symplectic bundle $(\cS,\omega)$, an object which encodes all
couplings and theta angles in a fully geometric manner. A classical
field strength {\em solution} is a field strength configuration which
obeys the polarized self-duality condition. This provides a global
formulation of the theory on oriented Lorentzian four manifolds of
arbitrary topology, which is manifestly covariant with respect to
electromagnetic duality. In this formulation, classical duality
transformations are described by (based) flat symplectic automorphisms
of $\Delta$. Such theories admits global solutions which correspond to
`classical electromagnetic U-folds' -- a notion which had been used
previously in the physics literature without being given a mathematically
clear definition.

While the treatment found in the physics literature discusses {\em
local} gauge potentials and {\em local} gauge transformations (which
are described using differential forms defined locally on spacetime),
a fully general and manifestly duality-covariant geometric formulation
of such theories in terms of connections on an appropriate principal
bundle has not yet been given. Such a formulation is required by the
Aharonov-Bohm effect \cite{Aharonov:1959fk} and by
Dirac-Schwinger-Zwanziger (DSZ) ``quantization''
\cite{Dirac:1931kp,Schwinger:1966nj,Zwanziger:1968rs}, which force the
field strengths to obey an \emph{integrality condition} implied by the
requirement of a consistent coupling between classical gauge fields and
quantum charge carriers. Imposing this condition restricts the set of
allowed field strengths, defining a so-called {\em prequantum  abelian
gauge theory} As pointed out in \cite{gesm}, the general formulation
of this condition involves the choice of a {\em Dirac system} $\cL$
for $\Delta$, defined as a $\cD$-flat fiber sub-bundle of $\cS$ whose
fibers are full symplectic lattices inside the fibers of $\cS$. Every
Dirac system has a {\em type} $\frt=(t_1,\ldots, t_n)$, where
$t_1,\ldots, t_n$ are positive integers such that $t_1|t_2|\ldots
|t_n$.  A duality structure is called {\em semiclassical} if it admits
a Dirac system.  The choice of a Dirac system $\cL$ refines a
semiclassical duality structure $\Delta$ to an {\em integral duality
structure} $\bDelta=(\cS,\omega,\cD,\cL)$ and reduces the group of
duality transformations to a discrete group, which generalizes the
arithmetic duality group known from the local formulation found in the
physics literature.  In the global setting, the Dirac system replaces
the ``Dirac lattice'' of the local approach and makes the DSZ
condition is rather subtle since it requires the use of cohomology
with local coefficients (see \cite{Hatcher, Spanier, Whitehead,
Steenrod}). In particular, the bundle-valued two-form $\cV$ (which
describes all electromagnetic and magnetoelectric field strengths
simultaneously) {\em cannot} in general be the curvature of a
connection defined on a principal torus bundle. Indeed, the vector
bundle $\cS$ is generally non-trivial, while the adjoint bundle of any
principal torus bundle is trivial.

In the present paper, we `solve' the DSZ integrality condition
determined by a Dirac system $\cL$ by giving the geometric formulation
of prequantum abelian gauge theory in terms of connections defined
on an adequate principal bundle $P$. More precisely, we show that the
combined field strength $\cV$ is the curvature of a connection defined
on a principal bundle with {\em non-abelian} and {\em disconnected}
structure group $G$, whose connected component of the identity is a
torus group and whose group of connected components is a modified
Siegel modular group. This shows that the manifestly duality-covariant
formulation of such gauge theories is not truly abelian, since it
involves a non-abelian structure group. Instead, the structure group
$G$ is {\em weakly abelian}, in the sense that only its {\em Lie
algebra} is abelian. As a consequence, the gauge group of $P$ has a
``discrete remnant'' which encodes equivariance of the theory under
electromagnetic duality transformations. Using this framework, we show 
that the electromagnetic duality transformations of the theory are given 
by (non-infinitesimal) gauge transformations, a fact that provides a 
geometric interpretation for the former. Principal connections $\cA$ on 
$P$ describe the combined electromagnetic and magnetoelectric {\em gauge 
potentials} of the theory. The polarized self-duality condition becomes 
a first-order differential equation for $\cA$ which is reminiscent of 
the instanton equations, though the signature of our spacetime is Lorentzian 
and the theory is of second order. These results provide a global geometric 
formulation of prequantum abelian gauge theory as a theory of principal 
connections, in a manner that is manifestly covariant under electromagnetic 
duality.

To extract the principal bundle description, we proceed as follows. We
first show that integral duality structures of rank $2n$ and type
$\frt$ are associated to {\em Siegel systems} of rank $2n$, which we
define to be local systems $Z$ of free abelian groups of rank $2n$
whose monodromy is contained in the {\em modified Siegel modular}
group $\Gamma=\Sp_\frt(2n,\Z)$ of type $\frt$. The latter is defined
as the group of automorphisms of a full symplectic lattice of rank
$2n$ and type $\frt$ and is a subgroup of $\Sp(2n,\R)$ which contains
the usual Siegel modular group $\Sp(2n,\Z)$, to which it reduces when
$\frt$ coincides with the {\em principal type}
$\delta=(1,\ldots,1)$. The vector bundle of the duality structure is
given by $\cS=Z\otimes_\Z \R$ while the flat connection $\cD$ is
induced by the monodromy connection of $Z$. A classical field strength
configuration $\cV$ is called {\em integral} if it satisfies the DSZ
quantization condition, which states that the cohomology class
$[\cV]_{\cD}\in H^2_\cD(M,\cS)$ of $\cV$ with respect to $\dd_\cD$
lies in the image of the local coefficient cohomology group $H^2(M,Z)$
through the natural morphism $H^2(M,Z)\rightarrow H^2_\cD(M,\cS)$,
where $ H^2_\cD(M,\cS)$ is the the second cohomology of the complex
$(\Omega^\ast(M,\cS),\dd_\cD)$.  We then show that $\cV$ is integral
if and only if it coincides with the adjoint curvature $\cV_\cA$ of a
connection $\cA$ defined on a principal bundle $P$ (called a {\em
Siegel bundle}) whose structure group is the group $G=\Aff_\frt$ of
automorphisms of an {\em integral affine symplectic torus} and whose
adjoint bundle identifies with $\cS$. The group $\Aff_\frt$ is a
semidirect product $A\rtimes \Gamma$ of the torus group
$A=\R^{2n}/\Z^{2n}$ with the modified Siegel modular group
$\Gamma=\Sp_\frt(2n,\Z)$ and hence has a countable group of components
which is isomorphic with $\Gamma$. The $\dd_\cD$-cohomology class of
$\cV_\cA$ coincides with the image in $H^2_\cD(M,\cS)$ of the {\em
twisted Chern class} $c(P)\in H^2(M,Z)$ of $P$. The Siegel system $Z$
(and hence the duality structure $\Delta$) is uniquely determined by
the Siegel bundle $P$ and Siegel bundles are determined up to
isomorphism by the pair $(Z,c)$. The classifying space of such
principal bundles is a {\em twisted} Eilenberg MacLane space
\cite{Gitler} of type $2$, namely a $K(\Z^{2n},2)$-fibration over
$K(\Sp_\frt(2n,\Z),1)$ whose $\kappa$-invariant is trivial and whose
monodromy is induced by the fundamental action of $\Sp_\frt(2n,\Z)$ on
$\Z^{2n}$.

There exists a large group of classical `pseudo-duality'
transformations which identifies the spaces of classical field
strength configurations and solutions of \emph{different} abelian
gauge theories. Locally, such transformations involve matrices $T\in
\Sp(2n,\R)$ acting on the field strengths and the DSZ integrality
condition with respect to a Dirac lattice of type $\frt$ restricts $T$
to lie in the arithmetic group $\Sp_\frt(2n,\Z)$.  As already pointed
out in \cite{gesm}, the global theory of such duality transformations
is much richer. We develop this theory from scratch, defining a
hierarchy of duality groups and providing short exact sequences to
compute them. We exploit this geometric framework to show that the 
electromagnetic duality transformations of abelian gauge theory 
correspond to gauge transformations of Siegel bundles, elucidating 
the geometric origin of electromagnetic duality. In particular, we emphasize 
the role played by the \emph{type} $\frt$ of the underlying Dirac system and 
by the monodromy representation of the Siegel system $Z$ in the correct 
definition and computation of the discrete duality group of the prequantum 
theory. The fact that a symplectic lattice need not be of principal type 
is well-known in the theory of symplectic tori as well as in  that of 
Abelian varieties, where non-principal types correspond to non-principal 
polarizations. The physical implications of non-principal types have been  
systematically explored only recently in the context of supersymmetric field
theories, see for instance \cite{Argyres:2015ffa,Argyres:2015gha,Argyres:2016yzz,Argyres:2019yyb,Argyres:2020wmq,Caorsi:2018zsq,Caorsi:2019vex} and references therein.

The construction of the present paper produces a class of geometric
gauge models which is amenable to the methods of mathematical gauge
theory \cite{Donaldson}. In particular, it allows for the study of
moduli spaces of solutions (which, as shown in \cite{gesm}, can be
viewed as `electromagnetic U-folds') using techniques borrowed from
the theory of instantons. In this spirit, we perform the time-like
reduction of the polarized self-duality equations, obtaining a novel
system of Bogomolny-like equations.  Solutions of these equations
define {\em polarized abelian dyons}, of which we describe a few
examples.

The paper is organized as follows. Section \ref{sec:classical} recalls
the description of classical abelian gauge theories with arbitrary
duality structure in terms of combined electromagnetic and
magnetoelectric field strengths, following \cite{gesm}. In the same
section, we describe the hierarchy of duality groups of such theories
and give a few short exact sequences to characterize them. Section
\ref{sec:DQsymplecticabelian} discusses the DSZ integrality condition
for general duality structures, relating the notion of Dirac system
(which appears in its formulation) to various equivalent objects.
Section \ref{sec:associatedbundle} discusses Siegel bundles and
connections showing in particular that a Siegel
induces a canonical integral duality structure.  In Section
\ref{sec:DSZquantization}, we give the formulation of
``prequantum'' abelian gauge theory (defined as classical abelian
gauge theory supplemented by the DSZ integrality condition) as a
theory of principal connections on a Siegel bundle. Section
\ref{sec:TimelikeReduction} discusses the time-like dimensional
reduction of the polarized self-duality equations, which leads to the
notion of polarized abelian dyon. In the same section, we construct examples
of polarized abelian dyons on the punctured affine 3-space. Appendix A recalls
the duality-covariant formulation of abelian gauge theory on
contractible Lorentzian four-manifolds, starting from the local
treatment found in the physics literature. Appendix B discusses
integral symplectic spaces and integral symplectic tori, introducing
certain notions used in the main text.

\subsection{Notations and conventions}

All manifolds and fiber bundles considered in the paper are smooth,
Hausdorff and paracompact.  The manifold denoted by $M$ is assumed
to be connected. In our convention, a Lorentzian four-manifold
$(M,g)$ has ``mostly plus'' signature $(3,1)$. If $E$ is a fiber
bundle defined on a manifold $M$, we denote by $\cC^\infty(M,E)$ the set
of globally-defined smooth sections of $E$ and by $\cC^\infty(E)$ the
sheaf of smooth sections of $E$.  

\begin{ack}
We thank V. Cort\'es for useful comments and discussions. C.S.S would like to thank P. C. Argyres, A. Bourget, M. Martone and R. Szabo for useful comments and suggestions. Part of the work of C. I. L. on this project was supported by grant IBS-R003-S1. The work of C.S.S. is supported by the Germany Excellence Strategy \emph{Quantum Universe} - 390833306.
\end{ack}


\section{Classical abelian gauge theory}
\label{sec:classical}


In this section we introduce the configuration space and equations of
motion defining {\em classical abelian gauge theory} on an oriented
Lorentzian four-manifold $(M,g)$ and discuss its global dualities and
symmetries. Appendix \ref{app:local} gives the description of this
theory for the special case when $M$ is contractible, which recovers
the local treatment found in the physics literature. The global
formulation presented in this section was proposed in a wider context
in reference \cite{gesm}, to which we refer the reader for certain
details. The definition of classical abelian gauge theory on $(M,g)$
is given in terms of field strengths and relies on the choice of a
\emph{duality structure} (defined as a flat symplectic vector bundle
$\Delta=(\cS,\omega,\cD)$ on $M$) equipped with a {\em taming} $\cJ$
of $(\cS,\omega)$, which encodes the gauge-kinetic functions (coupling
constants and theta angles) in a globally-correct and
frame-independent manner. The equations of motion of the theory are
encoded by the {\em $\cJ$-polarized self-duality condition} for the
combined electromagnetic and magnetoelectric {\em field strengths}, which are
modeled mathematically by a $\cD$-flat two-form valued in the
underlying vector bundle $\cS$ of the duality structure.


\subsection{Duality structures}


\noindent Let $M$ be a connected smooth manifold and $(\cS,\omega)$ be
a symplectic vector bundle defined on $M$ with symplectic structure $\omega$. Since 
$\Sp(2n,\R)$ and $\GL(n,\C)$ are homotopy equivalent to their common 
maximal compact subgroup $\U(n)$, the classification of symplectic, 
complex and Hermitian vector bundles defined on $M$ are equivalent. 
In particular, any complex vector bundle  admits a Hermitian pairing 
and any symplectic vector bundle admits a complex structure which is 
compatible with its symplectic pairing and  makes it into a Hermitian 
vector bundle. Thus a real vector bundle of even rank admits a symplectic 
pairing if and only if it admits a complex structure. The classifying 
spaces $\rB\Sp(2n,\R)$ and $\rB\U(n)$ are homotopy equivalent, hence 
the fundamental characteristic classes of a symplectic vector bundle
$(\cS,\omega)$ are Chern classes, which we denote by $c_k(\cS,\omega)$.

\begin{remark}
Suppose that $\dim M=4$ and let $\cS$ be an oriented real vector
bundle of rank $2n$ defined on $M$, thus $w_1(\cS)=0$. If $\cS$ admits
a complex structure $\cJ$ inducing its orientation, then its third
Stiefel-Whitney class $w_3(\cS)$ must vanish and its even
Stiefel-Whitney classes $w_2(\cS)$ and $w_4(\cS)$ must coincide with
the $\mathrm{mod}~2$ reduction of the Chern classes $c_1(\cS,\cJ)$ and
$c_2(\cS,\cJ)$ of the complex rank $n$ vector bundle defined
by $\cS$ and $\cJ$. In particular, the third {\em integral}
Stiefel-Whitney class $W_3(\cS)\in H^3(M,\Z)$ must vanish, i.e. $\cS$
must admit a $\Spin^c$ structure (notice that $W_5(\cS)$ vanishes for
dimension reasons). These conditions are not always sufficient. To
state the necessary and sufficient conditions (see \cite{Massey}),
we distinguish the cases:
\begin{itemize}
\item $n=1$, i.e. $\rk\cS=2$. Then $\cS$ always admits a complex
structure (equivalently, a symplectic pairing) which induces its
orientation (since $\SO(2)=\U(1)$).
\item $n=2$, i.e. $\rk\cS=4$. In this case, $\cS$ admits a complex
structure (equivalently, a symplectic pairing) which induces its
orientation if and only if it satisfies $W_3(\cS)=0$ and
$\fc^4(\cS)=0$, where $\fc^4(\cS)\in H^4(M,\Z)$ is an integral
obstruction class described in \cite[Theorem II]{Massey}.
\item $n\geq 3$, i.e. $\rk\cS\geq 6$. Then $\cS$ admits a complex
structure (equivalently, a symplectic pairing $\omega$) which induces
its orientation if and only if $W_3(\cS)=0$.
\end{itemize}
Notice that an oriented real vector bundle $\cS$ of rank four on a 
four-manifold $M$ is determined up to isomorphism by its first Pontryaghin 
class $p_1(\cS)\in H^4(M,\Z)$, its second Stiefel-Whitney class $w_2(\cS)\in
H^2(M,\Z_2)$ and its Euler class $e(\cS)\in
H^4(M,\Z)$. 
\end{remark}

\begin{definition}
A {\em duality structure} $\Delta \eqdef (\cS,\omega,\cD)$ on $M$ is a
flat symplectic vector bundle $(\cS,\omega)$ over $M$ equipped with a
flat connection $\cD:\cC^\infty(M,\cS)\rightarrow\Omega^1(M,\cS)$ which preserves 
$\omega$. The {\em rank} of $\Delta$ is the rank of the vector bundle $\cS$, 
which is necessarily even.
\end{definition}

\noindent Notice that the Chern classes $c_1(\cS,\omega)$ and
$c_2(\cS,\omega)$ of the underlying symplectic vector bundle of a
duality structure must be torsion classes.

\begin{definition}
A {\em based isomorphism of duality structures} from
$\Delta=(\cS,\omega,\cD)$ to $\Delta^{\prime}=(\cS^{\prime},\omega^{\prime}
,\cD^{\prime})$ is a based isomorphism of vector bundles 
$f:\cS\xrightarrow{\sim} \cS'$ which satisfies the conditions 
$\omega'\circ (f\otimes f)=\omega$ and 
$\cD'\circ f=(\id_{T^\ast M}\otimes f)\circ \cD$.
\end{definition}

\noindent Here and below, we let based morphisms of vector bundles act
on sections in the natural manner. We denote by $\Dual(M)$ be the groupoid 
of duality structures defined on $M$ and based isomorphisms of such.

The group $\Aut_b(\cS,\omega)$ of based automorphisms of a symplectic
vector bundle $(\cS,\omega)$ is called its group of {\em symplectic
gauge transformations}. Such transformations $\varphi$ act on the set
of linear connections $\cD$ defined on $\cS$ through:
\be
\cD\rightarrow (\id_{T^\ast M}\otimes \varphi)\circ \cD\circ \varphi^{-1}~~
\ee
and preserve the set of flat symplectic connections. The group
$\Aut_b(\Delta)$ of based automorphism of a duality structure
$\Delta=(\cS,\omega,\cD)$ coincides with the stabilizer of $\cD$ in
$\Aut_b(\cS,\omega)$. For any such duality structure, we have:
\be
c_1(\cS,\omega)=\delta(\hc_1(\cD))\, ,\quad c_2(\cS,\omega)=\delta(\hc_2(\cD))~~,
\ee
where $\hc_1(\cD)\in H^1(M,\U(1))$ and $\hc_2(\cD)\in H^3(M,\U(1))$
are the Cheeger-Chern-Simons invariants of the flat connection $\cD$ and
$\delta:H^i(M,\C/\Z)\rightarrow H^{i+1}(M,\Z)$ are the Bockstein
morphisms in the long exact sequence:
\be
\ldots \rightarrow H^i(M,\Z)\rightarrow H^i(M,\R)\stackrel{\exp_\ast}{\rightarrow}
H^i(M,\U(1))\stackrel{\delta}{\rightarrow}H^{i+1}(M,\Z)\rightarrow
H^{i+1}(M,\R)\rightarrow \ldots
\ee
induced by the exponential sequence:
\be
0\rightarrow \Z \rightarrow \R \rightarrow \U(1)\rightarrow 0~~.
\ee
The Cheeger-Chern-Simons invariants depend only on the gauge
equivalence class of $\cD$. 


\subsection{The twisted de Rham complex of a duality structure}


Given a duality structure $\Delta=(\cS,\omega,\cD)$ on a connected manifold $M$, let
$\dd_\cD:\Omega(M,\cS)\rightarrow \Omega(M,\cS)$ be the exterior
covariant derivative twisted by $\cD$ (notice that
$\dd_\cD\vert_{\Omega^0(M,\cS)}=\cD$). This defines a cochain complex:
\ben
\label{TwistedDeRhamComplex}
0\to\cC^\infty(M,\cS)\xrightarrow{\cD} \Omega^1(M,\cS) \xrightarrow{\dd_\cD}
\Omega^2(M,\cS) \xrightarrow{\dd_\cD} \Omega^3(M,\cS)
\xrightarrow{\dd_\cD} \Omega^4(M,\cS)\to 0\, ,
\een
whose total cohomology group (viewed as a $\Z$-graded abelian group)
we denote by $H_\cD^\ast(M,\cS)$.

\begin{definition}
The vector spaces $H_\cD^k(M,\cS)$ (where $k=0,\ldots, 4$) are called
the {\em twisted de Rham cohomology spaces} of the duality structure
$\Delta$.
\end{definition}

\noindent Let $\Omega^k_\fl(\cS)$ be the locally-constant sheaf of
$\cD$-flat $\cS$-valued $k$-forms. Then a straightforward
modification of the Poincar\'e lemma shows that the complex of sheaves:
\be
0\to\cC^\infty_\fl(\cS)\hookrightarrow \Omega^0(\cS)\xrightarrow{\cD}
\Omega^1(\cS) \xrightarrow{\dd_\cD} \Omega^2(\cS)
\xrightarrow{\dd_\cD} \Omega^3(\cS) \xrightarrow{\dd_\cD}
\Omega^4(\cS)\to 0
\ee
is exact and hence provides a resolution of the locally-constant sheaf
$\cC^\infty_\fl(\cS)$. Since $M$ is paracompact, each of the sheaves
$\Omega^k(\cS)$ is acyclic. Thus the sheaf
cohomology of $\Omega^0_\fl(\cS)$ can be computed as the cohomology of
the complex \eqref{TwistedDeRhamComplex}. This gives a natural
isomorphism of graded vector spaces:
\be
H^\ast_\cD(M,\cS)\simeq H^\ast(M,\cC^\infty_\fl(\cS))~~,
\ee
where the right hand side denotes sheaf cohomology.


\subsection{Flat systems of symplectic vector spaces}


Let $\Pi_1(M)$ be the fundamental groupoid of a connected manifold
$M$, whose objects are the points of $M$ and whose set of morphisms
$\Pi_1(m,m')$ from $m$ to $m'$ is the set of homotopy classes of
piecewise-smooth curves starting at $m$ and ending at $m'$. Let
$\Symp$ be the groupoid of finite-dimensional symplectic vector spaces
(see Appendix \ref{app:symp}). The functor category $[\Pi_1(M),\Symp]$
is a groupoid since all its morphisms (which are natural
transformations) are invertible.

\begin{definition}
A {\em flat system of symplectic vector spaces} (or {\em
$\Symp$-valued local system}) on $M$ is a functor
$F:\Pi_1(M)\rightarrow \Symp$, i.e. an object of the groupoid
$[\Pi_1(M),\Symp]$. An isomorphism of such systems is an isomorphism
in this groupoid. A flat system $F$ of symplectic vector spaces has
rank\footnote{The rank is constant on $M$ since $M$ is connected.}
$2n$ if $\dim F(m)=2n$ for all $m\in M$.
\end{definition}

\noindent Recall that a {\em symplectic representation} of a group $G$
is a representation $\rho:G\rightarrow \Aut(V,\omega)$ of $G$
through automorphisms of a finite-dimensional symplectic vector space
$(V,\omega)$. An {\em equivalence (or isomorphism) of symplectic
representations} from $\rho:G\rightarrow \Aut(V,\omega)$ to
$\rho^{\prime} :G\rightarrow \Aut(V',\omega')$ is an isomorphism of symplectic
vector spaces $\varphi:(V,\omega)\xrightarrow{\sim} (V',\omega')$ such
that $\varphi\circ\rho(g)=\rho'(g)\circ \varphi'$ for all $g\in
G$. We denote by $\SympRep(G)$ the groupoid of symplectic representations
of $G$ equipped with this notion of isomorphism.

\begin{definition}
Let $F$ be a flat system of symplectic vector spaces on $M$.
The {\em holonomy representation} of $F$ at a point $m\in M$ is the
morphism of groups $\hol_m(F):\pi_1(M,m)=\Pi_1(m,m)\rightarrow
\Aut(F(m))=\Symp(F(m),F(m))$ defined through:
\be
\hol_m(F)(\bc)=F(\bc)\, ,  \quad \forall\,\, \bc\in \pi_1(M,m)~~.
\ee
The {\em holonomy group} of $F$ at $m$ is the subgroup of $\Aut(F(m))$
defined through:
\be
\Hol_m(F)\eqdef \im (\hol_m(F))~~.
\ee
\end{definition}

\noindent Notice that $\hol_m(F)$ is a symplectic representation of
$\pi_1(M,m)$ on $F(m)$. Any homotopy class $\bgamma\in \Pi_1(m,m')$ of
paths from $m$ to $m'$ induces an isomorphism of symplectic vector
spaces $F(\bgamma):F(m)\stackrel{\sim}{\rightarrow} F(m')$ which
intertwines the holonomy representations of $F$ at the points $m$ and
$m'$:
\be
F(\bgamma)\circ \hol_m(F)(\bc)=\hol_{m'}(F)(\bgamma\, \bc
\,\bgamma^{-1})\circ F(\bgamma)\, , \quad \forall\,\, \bc\in \pi_1(M,m)~~.
\ee
Since $M$ is connected, it follows that the holonomy representation of
$F$ at a fixed basepoint $m_0\in M$ determines its holonomy
representation at any other point of $M$.  Moreover, the isomorphism
class of the holonomy group of $F$ at $m_0$ does not depend on the
choice of $m_0\in M$.

\begin{prop}
\label{prop:hol}
Let $m_0\in M$ be any point of $M$. Then the map $\hol_{m_0}$ which
sends a $\Symp$-valued local system $F$ defined on $M$ to its holonomy
representation $\hol_{m_0}(F)$ at $m_0$ and sends an isomorphism
$f:F\xrightarrow{\sim} F'$ of $\Symp$-valued local systems to the
equivalence of representations $\hol_{m_0}(f)\eqdef
f(m_0):\hol_{m_0}(F)\xrightarrow{\sim} \hol_{m_0}(F')$ defines an
equivalence of groupoids from $[\Pi_1(M),\Symp]$ to
$\SympRep(\pi_1(M,m_0))$. In particular, isomorphism classes of flat
systems of symplectic vector spaces of rank $2n$ defined on $M$ are in
bijection with the points of the character variety:
\ben
\label{fR}
\fR(\pi_1(M,m_0),\Sp(2n,\R))\eqdef \Hom(\pi_1(M,m_0),\Sp(2n,\R))/\Sp(2n,\R)\, ,
\een
where $\Sp(2n,\R)$ acts on $\Hom(\pi_1(M),\Sp(2n,\R))$ through its
adjoint representation.
\end{prop}

\begin{proof}
An isomorphism $f:F\xrightarrow{\sim} F'$ of $\Symp$-valued local
systems is a collection of isomorphisms of symplectic vector spaces
$f(m):F(m)\rightarrow F'(m)$ for all $m\in M$ which satisfies:
\be
f(m)\circ F(\bgamma)=F'(\bgamma)\circ f(m)\, , \quad \forall\,\, \bgamma\in \Pi_1(m,m')\, , \quad \forall\,\, m,m'\in M~~.
\ee
Taking $m'=m=m_0$ in this relation shows that $f(m)$ is an equivalence
between the holonomy representations of $F$ and $F'$ at $m$:
\be
\hol_{m_0}(F')(\bc)=f(m)\circ \hol_{m_0}(F)(\bc) \circ f(m_0)^{-1}\, , \quad \forall\,\, \bc\in \pi_1(M,m_0)~~.
\ee
Conversely, it is easy to see that any such equivalence of
representations extends to an isomorphism from $F$ to $F'$.
\end{proof}


\subsection{The flat system of symplectic vector spaces defined by a duality structure}

 
\begin{definition}
The {\em parallel transport functor} $\cT_\Delta\in [\Pi_1(M),\Symp]$
of a duality structure $\Delta=(\cS,\omega,\cD)$ is the functor which
associates to each point $m\in M$ the symplectic vector space
$\cT_\Delta(m)\eqdef (\cS_m,\omega_m)$ and to the homotopy class (with
fixed endpoints) $\bc\in \Pi_1(M)$ of any piecewise-smooth curve
$c\colon [0,1] \to M$ the isomorphism of symplectic vector spaces: \be
\cT_\Delta(\bc):(\cS_{c(0)},\omega_{c(0)})
\xrightarrow{\sim} (\cS_{c(1)},\omega_{c(1)}) \ee given by
the parallel transport of $\cD$.
\end{definition}

\noindent Notice that $\cT_\Delta:\Pi_1(M)\rightarrow \Symp$ is a flat
system of symplectic vector spaces. The map which sends $\Delta$ to
$\cT_\Delta$ extends in an obvious manner to an equivalence of groupoids:
\be
\cT:\Dual(M)\stackrel{\sim}{\rightarrow} [\Pi_1(M),\Symp]~~,
\ee
which sends a based isomorphism
$f:\Delta=(\cS,\omega,\cD)\stackrel{\sim}{\rightarrow}
\Delta'=(\cS',\omega',\cD')$ of duality structures to the invertible
natural transformation $\cT(f):\cT_\Delta\xrightarrow{\sim} \cT_{\Delta'}$
given by the isomorphisms of symplectic vector spaces:
\be
\cT(f)(m)\eqdef f_m:(\cS_m,\omega_m)\stackrel{\sim}{\rightarrow} (\cS'_m,\omega_m)\, , \quad \forall\,\, m\in M\, .
\ee
These isomorphisms intertwine $\cT^\Delta(\bc)$ and
$\cT^{\Delta'}(\bc)$ for any $\bc\in \Pi_1(M)(m,m')$ since $f$
satisfies $\cD'\circ f=(\id_{T^\ast M}\otimes f)\circ \cD$. Hence one
can identify duality structures and systems of flat symplectic vector spaces
defined on $M$. For any duality structure $\Delta=(\cS,\omega,\cD)$ on
$M$, the holonomy representation $\hol_m(\cD)$ of the flat connection
$\cD$ at $m\in M$ coincides with the holonomy representation of the
flat system of symplectic vector spaces defined by $\Delta$:
\be
\hol_m(\cD)=\hol_m(\cT_\Delta)~~.
\ee
In particular, the holonomy groups of $\cD$ and $\cT_\Delta$ at $m\in M$ coincide:
\be
\Hol_m(\cD)=\Hol_m(\cT_\Delta)~~.
\ee
Since $M$ is connected, Proposition \ref{prop:hol} implies that
isomorphism classes of duality structures defined on $M$ are in
bijection with the character variety \eqref{fR}.


\subsection{Trivial duality structures}


\begin{definition}
Let $\Delta=(\cS,\omega,\cD)$ be a duality structure defined on $M$. We say
that $\Delta$ is:
\begin{enumerate}
\item {\em topologically trivial}, if the vector bundle $\cS$ is
trivializable, i.e. if it admits a global frame.
\item \emph{symplectically trivial} if the symplectic vector bundle
$(\cS,\omega)$ is isomorphic to the trivial symplectic vector
bundle, i.e. if it admits a global symplectic frame (a global frame in
which $\omega$ has the standard form).
\item {\em trivial} (or \emph{holonomy trivial}), if $\Delta$ admits a
global {\em flat} symplectic frame, i.e. if the holonomy group of $\cD$ is
the trivial group.
\end{enumerate}
\end{definition}

\noindent A symplectically trivial duality structure is automatically
topologically trivial, while a holonomy trivial duality structure is
symplectically trivial. If $M$ is simply-connected then every
duality structure is holonomy trivial. A global flat symplectic frame
of a holonomy-trivial duality structure $\Delta=(\cS,\omega,\cD)$ of
rank $2n$ has the form:
\be
\cE=(e_1,\ldots, e_n,f_1,\ldots, f_n)~~,
\ee
where $e_i,f_j$ are $\cD$-flat sections of $\cS$ such that:
\be
\omega(e_i,e_j)=\omega(f_i,f_j)=0~~,~~\omega(e_i,f_j)=-\omega(f_i,e_j)=\delta_{ij}\, , \quad \forall\,\, i,j=1,\ldots,n~~.
\ee
Any choice of such a frame induces a trivialization isomorphism
$\tau_\cE:\Delta\xrightarrow{\sim} \Delta_n$ between $\Delta$ and the
{\em canonical trivial duality structure}:
\ben
\label{Deltan}
\Delta_n\eqdef
(\underline{\R}^{2n},\underline{\omega}_{2n},\dd)\, ,
\een
of rank $2n$, where $\underline{\omega}_{2n}$ is the constant
symplectic pairing induced on the trivial vector bundle
$\underline{\R}^{2n}=M\times \R^{2n}$ by the canonical symplectic
pairing $\omega_{2n}$ of $\R^{2n}$ and $\dd\colon
\cC^\infty(M,\underline{\R}^{2n})\rightarrow\Omega^1(M,\underline{\R}^{2n})$
is the ordinary differential. Any two flat symplectic frames $\cE$ and
$\cE'$ of $\Delta$ are related by a based automorphism $T\in \Aut_b(\Delta)$:
\be
\cE'=T\cE\, ,
\ee
which corresponds to the constant automorphism $\tau_{\cE'}\circ
\tau_{\cE}^{-1}\in \Aut_b(\Delta_n)$ of $\Delta_n$ induced by an element
${\hat T}\in \Sp(2n,\R)$.


\subsection{Electromagnetic structures}


As before, let $M$ be a connected manifold.

\begin{definition}
An {\em electromagnetic structure} defined on $M$ is a pair $\Xi =
(\Delta,\cJ)$, where $\Delta = (\cS,\omega,\cD)$ is a duality
structure on $M$ and $\cJ\in \End(\cS)$ is a {\em taming} of the
symplectic vector bundle $(\cS,\omega)$, i.e. a complex structure on
$\cS$ which:
\begin{itemize}
\item is compatible with $\omega$, i.e. it satisfies:
\be
\omega(\cJ x,\cJ y)=\omega(x,y)~~\forall (x,y)\in \cS\times_M\cS
\ee
\item has the property that the symmetric bilinear pairing $Q_{\cJ,\omega}$
defined on $\cS$ through:
\ben
\label{Qdef}
Q_{\cJ,\omega}(x,y)\eqdef \omega(\cJ x,y)~~\forall (x,y)\in \cS\times_M \cS
\een
is positive-definite.
\end{itemize}
The {\em rank} of $\Xi$ is the rank of $\Delta$.
\end{definition}

\noindent Notice that the taming $\cJ$ is {\em not} required to be
flat with respect to $\cD$. The bundle-valued one-form:
\be
\Psi_\Xi\eqdef \cD^\ad(\cJ)=\cD\circ \cJ - (\id_{T^\ast M}\otimes \cJ)\circ \cD\in \Omega^1(M,\End(\cS))
\ee
is called the {\em fundamental form} of $\Xi$ and measures the failure
of $\cD$ to preserve $\cJ$. The electromagnetic structure $\Xi$ is
called {\em unitary} if $\Psi_\Xi=0$. We refer the reader to \cite{gesm} for
more detail on the fundamental form. 

\begin{definition}
Let $\Xi=(\Delta,\cJ)$ and $\Xi'=(\Delta',\cJ')$ be two electromagnetic
structures defined on $M$. A {\em based isomorphism of
electromagnetic structures} from $\Xi$ to $\Xi'$ is a based
isomorphism of duality structures $f:\Delta\xrightarrow{\sim}\Delta'$
such that $\cJ'=f\circ \cJ\circ f^{-1}$.
\end{definition}
 
\begin{remark}
A taming $\cJ$ of a duality structure $\Delta$ can also be described
using a \emph{positive complex polarization}. Let $\Delta_{\C} =
(\cS_{\C} , \omega_{\C} , \cD_{\C})$ be the complexification of
$\Delta$. Then $\cJ$ is equivalent to a complex Lagrangian
sub-bundle $L\subset \cS_{\C}$ such that:
\be
\i \omega(x,\bar{x}) > 0\, , \quad \forall\,\, x\in \dot{\cS}_{\C}\, ,
\ee
where $\dot{\cS}_{\C}$ is the complement of the zero section in
$\cS_\C$.  By definition, such a Lagrangian sub-bundle is a positive
complex polarization of the symplectic vector bundle $(\cS,\omega)$. A
detailed description of this correspondence can be found in
\cite[Appendix A]{gesm}. Note that the physics literature sometimes
uses a local version of positive polarizations when discussing Abelian
gauge theory. In our global framework this requires the supplementary
step of complexifying the vector bundle $\cS$.
\end{remark}
 
\begin{definition}
A {\em taming map} of size $2n$ defined on $M$ is a smooth map
$\cJ:M\rightarrow \Mat(2n,\R)$ such that $\cJ(m)$ is a taming of the
standard symplectic form $\omega_{2n}$ of $\R^{2n}$ for all $m\in M$
(in particular, we have $\cJ(m)\in \Sp(2n,\R)$ for all $m$). 
\end{definition}

\noindent We denote by $\fJ_n(M)$ the set of all taming maps of size $2n$. Let 
$\Xi=(\Delta,\cJ)$ be an electromagnetic structure of rank $2n$ defined 
on $M$ whose underlying duality structure $\Delta=(\cS,\omega,\cD)$ is 
holonomy trivial. Choosing a flat global symplectic frame $\cE$, we can identify 
$\Delta$ with the canonical trivial duality structure \eqref{Deltan} 
through an isomorphism $\tau_\cE:\Delta\xrightarrow{\sim}\Delta_n$. This 
identifies $\cJ$ with the section $\tau_{\cE}\circ \cJ\circ
\tau_{\cE}^{-1}$ of the trivial vector bundle $\underline{\R}^{2n}$, which
in turn can be viewed as a smooth map from $M$ to $\R^{2n}$. Since
$\tau_{\cE}$ transports $\omega$ to the constant symplectic form
induced on $\underline{\R}^{2n}$ by $\omega_{2n}$, it is easy to see that
this map is a taming map of size $2n$ defined on $M$. Hence any choice
of global flat symplectic frame identifies the set of tamings of
$\Delta$ with $\fJ_n(M)$.


\subsection{The polarized Hodge operator}


Let $(M,g)$ be an oriented Lorentzian four-manifold. Let
$\Xi=(\Delta,\cJ)$ be an electromagnetic structure of rank $2n$
defined on $M$ and let $Q:=Q_{\cJ,\omega}$ be the Euclidean scalar
product induced by $\cJ$ and $\omega$ on $\cS$ as in equation
\eqref{Qdef}. The Hodge operator $\ast_g$ of $(M,g)$ extends trivially
to an automorphism of the vector bundle
$\wedge^\ast(M,\cS)\eqdef\wedge^\ast T^\ast M\otimes \cS$, which we
denote by the same symbol.

\begin{definition}
The {\em $\cJ$-polarized Hodge operator} of $(M,g,\Xi)$ is the automorphism
$\star_{g,\cJ}$ of the vector bundle $\wedge^\ast(M,\cS)$ defined through:
\be
\star_{g,\cJ}\eqdef\ast_g\otimes \cJ= \cJ\otimes \ast_g \, .
\ee
\end{definition}

\noindent Let $(\cdot,\cdot)_{g}$ be the pseudo-Euclidean scalar
product induced by $g$ on the total exterior bundle
$\wedge^\ast(M)\eqdef \wedge^\ast T^\ast M$. Together with $Q$, this
pairing induces a pseudo-Euclidean scalar product
$(\cdot,\cdot)_{g,Q}$ on the vector bundle $\Lambda^\ast(M,\cS)$,
which is uniquely determined by the condition: \be (\rho_1\otimes
\xi_1,\rho_2\otimes \xi_2)_{g, Q} = \delta_{k_1,k_2}(-1)^{k_1}
Q(\xi_1,\xi_2) (\rho_1, \rho_2)_{g}\, \ee for all $\rho_1\in
\Omega^{k_1}(M)$, $\rho_2\in \Omega^{k_2}(M)$ and $\xi_1,\xi_2\in
\cC^\infty(M,\cS)$. The polarized Hodge operator $\star_{g,\cJ}$ induces
an involutive automorphism of the space $\Omega^2(M,\cS)$. We have a
$(\cdot,\cdot)_{g,Q}$-orthogonal splitting:
\be
\wedge^2(M,\cS) = \wedge^2_{+}(M,\cS) \oplus \wedge^2_{-}(M,\cS)
\ee
into eigenbundles of $\star_{g,\cJ}$ with eigenvalues $+1$ and $-1$
for the polarized Hodge operator. The sections of
$\wedge^2_{+}(M,\cS)$ and $\wedge^2_{-}(M,\cS)$ are called
respectively {\em polarized self-dual} and {\em polarized
anti-self-dual} $\cS$-valued two-forms. Notice that the definition of
these notions does {\em not} require complexification of $\cS$.


\subsection{Classical abelian gauge theory}


Let $(M,g)$ be an oriented Lorentzian four-manifold.

\begin{definition}
\label{def:confabeliangauge0}
The {\em classical configuration space} defined by the duality
structure $\Delta=(\cS,\omega,\cD)$ on $M$ is the vector space of
two-forms valued in $\cS$ which are closed with respect to
$\dd_{\cD}$:
\be
\Conf(M,\Delta) \eqdef  \left\{\cV\in \Omega^2(M,\cS) \,\, |\,\, \dd_\cD \cV= 0\right\} \, .
\ee
\end{definition}

\begin{definition}
\label{def:solabeliangauge0}
The {\em classical abelian gauge theory} determined by an 
electromagnetic structure $\Xi=(\Delta,\cJ)$ on $(M,g)$ is 
defined by the polarized self-duality condition:
\be
\star_{g,\cJ} \cV = \cV
\ee
for $\cV\in \Conf(M,\Delta)$. The solutions of this
equation are called {\em classical electromagnetic field strengths}
and form the vector space:
\be
\Sol(M,g,\Xi)=\Sol(M,g,\Delta,\cJ) \eqdef \left\{ \cV\in \Conf(M,\Delta)\, \vert \, \star_{g,\cJ}\cV=\cV   \right\}\, .
\ee
\end{definition}

\noindent 
Notice that classical abelian gauge theory is formulated in terms of
field strengths.  In later sections, we will formulate the pre-quantum
version of this theory (which is obtained by imposing an appropriate
DSZ quantization condition) in terms of connections defined on certain
principal bundles called \emph{Siegel bundles} (see Definition
\ref{def:affinesymplecticb}).  The classical theory simplifies when
$M$ is contractible, since in this case any duality structure defined
on $M$ is holonomy-trivial.  This case corresponds to the local
abelian gauge theory discussed in Appendix \ref{app:local}, which
makes contact with the formulation used in the physics
literature. Notice that the traditional physics formulation (which is
valid only locally or for holonomy-trivial duality structures) relies
on gauge-kinetic functions, leading to complicated formulas which
obscure the geometric structure displayed by Definition \ref{def:solabeliangauge0}. 
It was shown in \cite{gesm} that elements of $\Sol(M,g,\Xi)$ correspond to 
classical electromagnetic U-folds when the duality structure $\Delta$ is 
not holonomy trivial.


\subsection{Classical duality groups}
\label{subsec:classdual}


Let $(M,g)$ be an oriented Lorentzian four-manifold and $\cS$ be a
vector bundle defined on $M$. Let $\Aut(\cS)$ be the group of those
{\em unbased} vector bundle automorphisms of $\cS$ which cover
orientation-preserving diffeomorphisms of $M$ and let
$\Aut_b(\cS)\subset \Aut(\cS)$ the subgroup of based automorphisms,
i.e. those automorphisms which cover the identity of $M$. Let
$\Diff(M)$ be the group of {\em orientation-preserving}
diffeomorphisms of $M$. Given $u\in\Aut(\cS)$, let $f_u\in \Diff(M)$
be the orientation-preserving diffeomorphism of $M$ covered by
$u$. This defines a morphism of groups:
\be
\Aut(\cS)\ni u \rightarrow f_u\in \Diff(M)~~.
\ee
The group $\Aut(\cS)$ admits an $\R$-linear representation:
\be
\Aut(\cS)\ni u\rightarrow \A_u \in \Aut(\Omega^\ast(M,\cS))
\ee
given by push-forward; we will occasionally also use the notation:
$u\cdot (-) \eqdef \A_u(-)$. When $f_u$ is not the identity the 
push-forward action of $u\in \Aut(\cS)$ must be handled with care; we 
refer the reader to \cite[Appendix D]{gesm} for a detailed description 
of this operation and its properties. On decomposable elements of
$\Omega^\ast(M,\cS)=\Omega^\ast(M)\otimes \cC^\infty(M,\cS)$, it is
given by:
\be
\A_u(\alpha\otimes \xi) = u\cdot (\alpha\otimes \xi) 
 = (f_{u\ast}\alpha)\otimes (u\cdot\xi) = 
 (f_{u\ast}\alpha)\otimes (u\circ\xi\circ f^{-1}_u)\, ,
\ee
where $f_{u\ast}\alpha$ is the push-forward of $\alpha$ on $M$ 
as defined in \cite[Appendix D]{gesm}. For instance, if 
$\alpha\in \Omega^1(M)$ we have:
\begin{equation*}
(f_{u\ast}\alpha)(v) = \alpha(f_{u\ast}^{-1}(v))\circ f^{-1}_u = \alpha(\dd f_{u}^{-1}(v)\circ f_u) \circ f^{-1}_u  \in C^{\infty}(M)\, , \quad \forall\,\, v\in \mathfrak{X}(M)\, .
\end{equation*}
In particular, restricting to a given point $m\in M$ we obtain the familiar 
formula:
\begin{equation*}
(f_{u\ast}\alpha)(v)_m = \alpha(f_{u\ast}^{-1}(v))_{f^{-1}_u(m)} = \alpha_{f^{-1}_u(m)}(\dd f_{u}^{-1}\vert_m (v_{m}))\, .
\end{equation*}
Recall that $f_{u\ast}^{-1}(v)\in \mathfrak{X}(M)$ is again a vector 
field on $M$ whereas $\dd f_{u}^{-1}(v)\colon M\to TM$ is a vector 
field \emph{along} $f_u$. For any duality structure 
$\Delta=(\cS,\omega,\cD)$ defined on $M$, let:
\be
\Aut(\Delta) \eqdef \left\{ u\in \Aut(\cS)\,\, \vert\,\,  \omega_u = \omega\, ,\,\, \cD_u = \cD \right\}\, ,
\ee
be the group of {\em unbased} automorphisms of $\Delta$, defined as
the stabilizer of the pair $(\omega,\cD)$ in $\Aut(\cS)$. Here
$\cD_{u}$ is the push-forward of the connection $\cD$ by $u$, which is
defined through:
\be
(\cD_u)_v(s)= u\cdot \cD_{f_{u\ast}^{-1}(v)}(u^{-1}\cdot s) = u\circ (\cD_{f_{u\ast}^{-1}(v)}(u^{-1}\circ s \circ
f_u))\circ f^{-1}_u\,  , \quad  \forall\,\, s\in \Gamma(\cS)\quad  \forall\,\, v\in \mathfrak{X}(M)\, ,
\ee
where $f_{u\ast}^{-1} \cdot v = \dd f^{-1}_u(v)\circ f_u$. Note that if 
$s\in \Gamma(S)$ is a parallel section of $\cS$ with respect to
$\cD$, then $u\cdot s\in \Gamma(S)$ is parallel with respecto $\cD_u$ for 
all $u\in \Aut(\Delta)$. We have a short exact sequence:
\be
1 \to \Aut_b(\Delta) \to \Aut(\Delta) \to \Diff_\Delta(M)\to 1\, ,
\ee
where $\Diff_\Delta(M)\subset\Diff(M)$ is the subgroup formed by
those orientation-preserving diffeomorphisms of $M$ that can be
covered by elements of $\Aut(\Delta)$. Given a taming
$\cJ\in \End(\cS)$ of $(\cS,\omega)$, we define:
\be
\cJ_u\eqdef u\circ \cJ\circ u^{-1}\in \End(\cS)~~,
\ee
which is a taming of the duality structure $\Delta_u\eqdef (\cS,\omega_u,\cD_u)$.
On the other hand, if $g$ is a Lorentzian metric on $M$, we set:
\be
g_u\eqdef (f_u)_\ast(g)~~.
\ee
Finally, we define $\Xi_u\eqdef (\Delta_u, \cJ_u)$ and we denote by $\Iso(M,g)$ the group
of orientation-preserving isometries of $(M,g)$.

\begin{definition}
\label{def:dualitygroups}
\begin{itemize}
Let $\Xi=(\Delta,\cJ)$ be an electromagnetic structure defined on $M$.
\item The group $\Aut(\Delta)$ of unbased automorphisms of $\Delta$ is
called the {\em unbased pseudo-duality group} defined by
$\Delta$. The {\em unbased pseudo-duality transformation}
defined by $u\in \Aut(\Delta)$ is the linear isomorphism:
\ben
\label{Aconf}
\A_u\colon \Conf(M,\Delta) \xrightarrow{\sim} \Conf(M,\Delta)~~,
\een
which restricts to an isomorphism:
\be
\A_u\colon \Sol(M,g,\Delta,\cJ) \xrightarrow{\sim} \Sol(M,g_u, \Delta,\cJ_u)~~.
\ee
\item The {\em unbased duality group} $\Aut(g,\Delta)$ of
the pair $(g,\Delta)$ is the stabilizer of $g$ in $\Aut(\Delta)$:
\be
\Aut(g,\Delta)\eqdef \Stab_{\Aut(\Delta)}(g)=\left\{u\in \Aut(\Delta)\,\, |\,\, f_u\in \Iso(M,g) \right\}\, .
\ee
The {\em unbased duality transformation} defined by $u\in
\Aut(g,\Delta)$ is the linear isomorphism \eqref{Aconf},
which restricts to an isomorphism:
\be
\A_{u}\colon \Sol(M,g,\Delta,\cJ) \xrightarrow{\sim} \Sol(M,g,\Delta,\cJ_u)~~.
\ee
\item The group $\Aut_b(\Delta)$ is called the {\em (electromagnetic)
duality group}
defined by $\Delta$. The {\em duality transformation} defined by $u\in
\Aut_b(\Delta)$ is the linear isomorphism \eqref{Aconf}, which
restricts to an isomorphism:
\be
\A_{u}\colon \Sol(M,g,\Delta,\cJ) \to \Sol(M,g,\Delta,\cJ_u)~~.
\ee
\end{itemize}
\end{definition}

\begin{remark}
The fact that $\A_u$ restricts as stated above follows from a direct yet subtle computation.
Since the conditions involved are linear, it is enough to verify it on an homogeneous 
element. If $\cV = \alpha\otimes \xi \in \Sol(M,\Delta)$, $u\in \Aut(\Delta)$ and 
$v\in \mathfrak{X}(M)$ we compute:
\begin{eqnarray*}
&\dd_{(\cD_u)_v} \mathbb{A}_u(\cV) = \dd_{(\cD_u)_v}((f_{u \ast}\alpha)\otimes (u\cdot\xi)) =   (\iota_v f_{u \ast}\dd\alpha)\otimes ( u\cdot\xi) + (f_{u \ast}\alpha)\otimes ((\cD_u)_v u\cdot\xi) \\
& = (f_{u \ast}\iota_{f_{u\ast}^{-1}(v)} \dd\alpha)\otimes ( u\cdot\xi)  + (f_{u \ast}\alpha)\otimes (u\cdot \cD_{f_{u\ast}^{-1}(v)} \xi) = \A_u ((\iota_{f_{u\ast}^{-1}(v)}\dd\alpha)\otimes (\xi) + \alpha\otimes (\cD_{f_{u\ast}^{-1}(v)} \xi)) \\
& = \A_u (\dd_{\cD_{f_{u\ast}^{-1}(v)}} \cV) = 0\, .
\end{eqnarray*}

\noindent
On the other hand, if $\cV = \alpha\otimes \xi \in \Sol(M,g,\Xi)$ we compute:
\begin{equation*}
\star_{g_u,\cJ_{u}} \mathbb{A}_u(\cV) =    (\ast_{g_u}f_{u\ast}\alpha) \otimes (\cJ_{u}\circ u\cdot \xi) = f_{u\ast}(\ast_{g}\alpha) \otimes u\circ\cJ(\xi)\circ f_u^{-1} =  \mathbb{A}_u(\star_{g_,\cJ} \cV )  = \mathbb{A}_u(\cV )  \, ,
\end{equation*}
implying that the restrictions of $\A_u$ to $\Sol(M,\Delta)$ and $\Conf(M,g,\Xi)$ 
are well-defined as stated above.
\end{remark}

\noindent
We have obvious inclusions:
\ben
\label{DualitySequence}
\Aut_b(\Delta)\subset \Aut(g,\Delta)\subset \Aut(\Delta)
\een
and a short exact sequence:
\ben
\label{eq:dualitygroupsequence}
1 \to \Aut_b(\Delta) \to \Aut(g,\Delta) \to \Iso_\Delta(M,g)\to 1\, ,
\een
where $\Iso_\Delta(M,g)\subset \Iso(M,g)$ denotes the group formed by those
orientation-preserving isometries of $(M,g)$ that can be covered by
elements of $\Aut(g,\Delta)$.

The classical duality group $\Aut_b(\Delta)$ consists of all based
automorphisms of $\cS$ which preserve both $\omega$ and
$\cD$. Therefore (see \cite[Lemma 4.2.8]{Donaldson}) fixing a point
$m_0\in M$ it can be realized as the centralizer $C_{m_0}(\Delta)$ of
the holonomy group $\Hol_{m_0}(\cD)$ of $\cD$ at $m_0$
inside the group $\Aut(\cS_{m_0},\omega_{m_0})\simeq \Sp(2n,\R)$:
\be
\mathrm{Aut}_b(\Delta) \simeq C_{m_0}(\Delta)~~.
\ee
In particular, $\mathrm{Aut}_b(\Delta)$ is a closed subgroup of
$\Sp(2n,\R)$ and hence it is a finite-dimensional Lie group. The same
holds for $\Iso_\Delta(M,g)$, which is a closed subgroup of the
finite-dimensional\footnote{It is well-known that the isometry group
of any pseudo-Riemannian manifold is a finite-dimensional Lie group.
See for example \cite[Theorem 5.1, p. 22]{Kobayashi}.} Lie group
$\Iso(M,g)$. The sequence \eqref{eq:dualitygroupsequence} implies that
$\Aut(g,\Delta)$ is also a finite-dimensional Lie group. In general,
the groups defined above differ markedly from their local counterparts
described in Appendix \ref{sec:symmetrieslocal}. We stress that the
latter are not the adequate groups to consider when dealing with
electromagnetic U-folds (since in that case the duality structure is not
holonomy trivial).

\begin{definition}
\label{def:unitarygroups}
Let $\Xi=(\Delta,\cJ)$ be an electromagnetic structure defined on $M$.
\begin{itemize}
\item The group:
\be
\Aut(\Xi)\eqdef \{u\in \Aut(\Delta)~\vert~\cJ_u=\cJ\}
\ee
of unbased automorphisms of $\Xi$ is called the {\em unbased
unitary pseudo-duality group} defined by $\Xi$. The {\em unbased
unitary pseudo-duality transformation} defined by $u\in \Aut(\Xi)$
is the linear isomorphism:
\be
\A_u\colon \Sol(M,g,\Xi) \xrightarrow{\sim} \Sol(M,g_u, \Xi)\, .
\ee
\item The {\em unbased unitary duality group} of the pair
$(g, \Xi)$ is the stabilizer of $g$ in $\Aut(\Xi)$:
\ben
\label{eq:metricunitaryglobal}
\Aut(g,\Xi) \eqdef \Stab_{\Aut(\Xi)}(g)=\Stab_{\Aut(g,\Delta)}(\cJ)=\left\{ u\in
\Aut(\Delta) \,\, \vert \,\, \cJ_u = \cJ ~\&~ f_u\in \Iso(M,g)\right\}~~.
\een
The {\em unbased unitary duality transformation} defined by $u\in
\Aut(g,\Xi)$ is the linear automorphism:
\be
\A_{u}\colon \Sol(M,g,\Xi) \xrightarrow{\sim} \Sol(M,g, \Xi)~~.
\ee		
\item The group $\Aut_b(\Xi)$ of based automorphisms of $\Xi$ is
called the {\em classical unitary duality group} defined by $\Xi$:
\be
\Aut_b(\Xi)=\Stab_{\Aut_b(\Delta)}(\cJ) = \left\{ u \in \Aut_b(\Delta)
\,\, \vert \,\, \cJ_u = \cJ \right\}\, .
\ee
The {\em classical unitary duality transformation} defined by $u\in
\Aut_b(\Xi)$ is the linear automorphism:
\be
\A_{u}\colon \Sol(M,g,\Xi) \xrightarrow{\sim} \Sol(M,g, \Xi)
\ee
\end{itemize}
\end{definition}

\noindent We have obvious inclusions:
\be
\Aut_b(\Xi)\subset \Aut(g,\Xi)\subset \Aut(\Xi)
\ee
and a short exact sequence:
\ben
1 \to \Aut_b(\Xi) \to \Aut(g,\Xi) \to \Iso_\Xi(M,g)\to 1\, ,
\een
where $\Iso_\Xi(M,g)$ is the group formed by those
orientation-preserving isometries of $(M,g)$ which are covered by
elements of $\Aut(g,\Xi)$. Arguments similar to those above show that
$\Aut_b(\Xi)$, $\Aut(g,\Xi)$ and $\Iso_\Xi(M,g)$ are
finite-dimensional Lie groups. The previous definitions give global
mathematically rigorous descriptions of several types of \emph{duality
groups} associated to abelian gauge theory on Lorentzian
four-manifolds. In general, these can differ markedly from their
``local'' counterparts described in Appendix
\ref{sec:symmetrieslocal}, which are considered traditionally in the
physics literature.


\subsection{The case of trivial duality structure}


Let $(M,g)$ be an oriented Lorentzian four-manifold. For any $n\geq 0$,
the set $\cC^\infty(M,\Sp(2n,\R))$ of smooth $\Sp(2n,\R)$-valued functions
defined on $M$ is a group under pointwise multiplication, whose group of
automorphisms we denote by $\Aut(\cC^\infty(M,\Sp(2n,\R)))$. 

\begin{lemma}
\label{lemma:trivSunbased} Let $(\cS,\omega)$ be a symplectic vector
bundle of rank $2n$ defined on $M$ which is symplectically
trivializable. Then any symplectic trivialization of $(\cS,\omega)$
induces an isomorphism of groups:
\be
\Aut(\cS,\omega)\simeq \cC^\infty(M,\Sp(2n,\R))\rtimes_\alpha \Diff(M)~~,
\ee
where $\alpha:\Diff(M)\rightarrow \Aut(\cC^\infty(M,\Sp(2n,\R)))$ is the
morphism of groups defined through:
\be
\alpha(\varphi)(f)\eqdef f\circ \varphi^{-1}\, , \quad \forall \,\, \varphi\in
\Diff(M)\, , \,\, \forall\,\, f\in \cC^\infty(M,\Sp(2n,\R))~~.
\ee
In particular, we have a short exact sequence of groups:
\be
1\rightarrow \cC^\infty(M,\Sp(2n,\R))\rightarrow \Aut(\cS,\omega)\rightarrow \Diff(M)\rightarrow 1
\ee
which splits from the right.
\end{lemma}

\begin{proof}
Let $\tau:\cS\stackrel{\sim}{\rightarrow} M\times \R^{2n}$ be a symplectic
trivialization of $(\cS,\omega)$. Then the map
$\Ad(\tau):\Aut(\cS,\omega)\rightarrow \Aut(M\times \R^{2n},\omega_{2n})$ defined
through:
\be
\Ad(\tau)(f)\eqdef \tau\circ f\circ \tau^{-1}\, , \quad \forall \,\, f\in \Aut(\cS,\omega)\, ,
\ee
is an isomorphism of groups. Let $f\in \Aut(\cS,\omega)$ be an unbased
automorphism of $(\cS,\omega)$ which covers the diffeomorphism $\varphi\in
\Diff(M)$. Then $\Ad(\tau)(f)$ is an unbased automorphism of $M\times
\R^{2n}$ which covers $\varphi$ and hence we have:
\be
\Ad(\tau)(f)(m,x)=(\varphi(m), {\tilde f}(m)(x))\, , \quad \forall\,\, (m,x)\in M\times \R^{2n}~~,
\ee
where ${\tilde f}:M\rightarrow \Sp(2n,\R)$ is a smooth map.
Setting $h\eqdef {\tilde f}\circ \varphi^{-1}\in \cC^\infty(M,\Sp(2n,\R))$, we have:
\ben
\label{fcomps}
\Ad(\tau)(f)(m,x)=(\varphi(m), h(\varphi(m))(x))\, , \quad \forall\,\, (m,x)\in M\times \R^{2n}
\een
and the correspondence $f\rightarrow (h,\varphi)$ gives a bijection
between $\Aut(\cS,\omega)$ and the set
$\cC^\infty(M,\Sp(2n,\R))\times \Diff(M)$. If $f_1,f_2\in \Aut(\cS,\omega)$ correspond
through this map to the pairs:
\be
(h_1,\varphi_1), (h_2,\varphi_2)\in \cC^\infty(M,\Sp(2n,\R))\times \Diff(M)\, ,
\ee
then direct computation using
\eqref{fcomps} gives:
\be
\Ad(\tau)(f_1\circ f_2)(m,x)=((\varphi_1\circ \varphi_2)(m), h_1(m) (h_2\circ \varphi_1^{-1})(m)(x))~~,
\ee
showing that $f_1\circ f_2$ corresponds to the pair $(h_1\cdot
\alpha(\varphi_1)(h_2),\varphi_1\circ \varphi_2)$.
\end{proof}

\begin{cor}
\label{cor:AutDelta}
Let $\Delta=(\cS,\omega,\cD)$ be a holonomy trivial duality structure defined on $M$.
Then any trivialization of $\Delta$ induces an isomorphism of groups:
\be
\Aut(\Delta) \simeq \Sp(2n,\R)\times \Diff(M)~~.
\ee
\end{cor}

\begin{proof}
Follows from Lemma \ref{lemma:trivSunbased} by noticing that the
action $\alpha$ of $\Diff(M)$ on $\cC^\infty(M,\Sp(2n,\R))$ restricts
to the trivial action on the subgroup:
\be
\{f\in \cC^\infty(M,\Sp(2n,\R))~\vert~\dd f=0\}\simeq \Sp(2n,\R)
\ee
of constant $\Sp(2n,\R)$-valued functions defined on $M$. 
\end{proof}

Fix an electromagnetic structure $\Xi = (\Delta,\cJ)$ of rank $2n$
defined on $M$ with holonomy-trivial underlying duality structure
$\Delta=(\cS,\omega,\cD)$. Choosing a global flat symplectic frame
$\cE$ of $\cS$, we identify $\Delta$ with the canonical trivial
duality structure \eqref{Deltan} and $\cJ$ with a taming map of size
$2n$. Then $\cD$ identifies with the trivial connection and $\dd_\cD$
identifies with the exterior derivative $\dd \colon
\Omega(M,\R^{2n})\to \Omega(M,\R^{2n})$ extended trivially to
vector-valued forms. Using Lemma \ref{lemma:trivSunbased} and its
obvious adaptation, we obtain:
\beqa
& \Aut_b(\cS) \equiv \cC^{\infty}(M,\GL(2n,\R))\, , \quad \Aut(\cS) \equiv
\cC^{\infty}(M,\GL(2n,\R))\rtimes_{\alpha} \Diff(M)\, ,\\
& \Aut_b(\cS,\omega) \equiv \cC^{\infty}(M,\Sp(2n,\R))\, , \quad \Aut(\cS,\omega) \equiv
\cC^{\infty}(M,\Sp(2n,\R))\rtimes_{\alpha} \Diff(M)\, .
\eeqa
On the other hand, Corollary \ref{cor:AutDelta} gives:
\beqa
& \Aut_b(\Delta)\equiv \Sp(2n,\R)\, , \quad \Aut(\Delta) \equiv \Sp(2n,\R)\times \Diff(M)\, ,\\
& \Aut(g,\Delta) = \Sp(2n,\R)\times \Iso(M,g)\, .
\eeqa
Moreover, we have:
\beqa
&\Aut_b(\Xi)\equiv \U_\cJ(n) \eqdef \{\gamma\in \Sp(2n,\R)\,\,\vert\,\,\gamma\cJ\gamma^{-1}=\cJ\}\, , \\
&\Aut(\Xi)\equiv \{(\gamma,f)\in \Sp(2n,\R)\times \Diff(M)\,\,\vert\,\,\gamma\cJ\gamma^{-1}=\cJ\circ f\}\,,
\eeqa
as well as:
\be
\Aut(g,\Xi) \equiv \{(\gamma,f)\in \Sp(2n,\R)\times \Iso(M)\,\,\vert\,\,\gamma\cJ\gamma^{-1}=\cJ\circ f\}~~.
\ee
Hence we recover the local formulas obtained in Appendix
\ref{sec:symmetrieslocal} for the duality groups of local abelian
gauge theory. Notice that $\Aut_b(\Xi)$ is isomorphic
with the unitary group $\U(n)$ when the electromagnetic structure
$\Xi$ is unitary, which amounts to the map $\cJ$ being constant.


\section{The Dirac-Schwinger-Zwanziger condition}
\label{sec:DQsymplecticabelian}


The previous section introduced classical abelian gauge
theory as a theory of \emph{field strengths}, i.e. a theory of
$\dd_\cD$-closed two-forms valued in the flat symplectic vector bundle
$\cS$ with equation of motion given by the polarized self-duality
condition. Well-known arguments originally due to Dirac as well as the
Aharonov-Bohm effect \cite{Aharonov:1959fk} imply that a consistent
coupling of the theory to quantum charge carriers imposes an
integrality condition on field strength configurations. This is
traditionally called the DSZ ``quantization'' condition, even though
it constrains {\em classical} field strength configurations -- in fact,
only particles which carry the corresponding charges 
are quantized in such arguments, but {\em not} the gauge fields themselves. 
To avoid confusion, we prefer to call it the {\em DSZ integrality
condition}. For local abelian gauge theories of rank $2n$
(which are discussed in Appendix \ref{app:local}), this condition can
be implemented using a full symplectic lattice in the standard
symplectic vector space $(\R^{2n},\omega_{2n})$, as usually done in the
physics literature \cite{Schwinger:1966nj,Zwanziger:1968rs}. For
abelian gauge theories with non-trivial electromagnetic
structure defined on an arbitrary Lorentzian four-manifold, we shall
implement this condition using a {\em Dirac system}, as originally
proposed in \cite{gesm}. We begin with some preliminaries.


\subsection{Principal bundles with discrete structure group}


Let $\Gamma$ be a discrete group and $Q$ be a principal bundle with
structure group $\Gamma$ and projection $p:Q\rightarrow M$. Then the
total space of $Q$ is a (generally disconnected) covering space of
$M$.  Let $U_Q:\Pi_1(M)\rightarrow \Phi_0(Q)$ be the monodromy
transport of the covering map $p:Q\rightarrow M$, where $\Phi_0(Q)$ is
the {\em bare fiber groupoid} of $Q$. By definition, the objects of
$\Phi_0(Q)$ are the fibers of $Q$ while its morphisms are arbitrary
bijections between the latter. By definition, the functor $U_Q$
associates to the homotopy class $\bc\in \Pi_1(M)(m,m')$ of any curve
$c:[0,1]\rightarrow M$ with $c(0)=m$ and $c(1)=m'$ the bijection
$U(\bc):Q_{m}\xrightarrow{\sim} Q_{m'}$ given by $U_Q(\bc)(x)\eqdef
{\tilde c}_x(1)\in Q_{m'}$, where ${\tilde c}_x$ is the the unique
lift of $c$ to $Q$ through the point $x\in Q$ (thus
$\tilde{c}_x(0)=x$). Notice that $x$ and $U(\bc)(x)$ lie on the same
connected component of $Q$ and hence the diffeomorphism
$U_Q(\bc):Q\xrightarrow{\sim} Q$ induces the trivial permutation of
$\pi_0(Q)$. For any $\gamma\in \Gamma$, the curve ${\tilde
c}^\gamma_x$ defined through ${\tilde c}_x^\gamma(t)\eqdef {\tilde
c}(t)\gamma$ for all $t\in [0,1]$ is a lift of $c$ through the point
$x\gamma$. The homotopy lifting property of $p$ implies that ${\tilde
c}^\gamma_x$ and ${\tilde c}_{x\gamma}$ are homotopic and hence
$U_Q(\bc)(x\gamma)=U(\bc)(x)\gamma$. This shows that $U_Q$ acts
through isomorphisms of $\Gamma$-spaces and hence it is in fact a
functor:
\be
U_Q:\Pi_1(M)\rightarrow \Phi(Q)~~,
\ee
where $\Phi(Q)$ is the {\em principal fiber groupoid} of $Q$ (whose
objects coincide with those of $\Phi_0(Q)$ but whose morphisms are
isomorphisms of $\Gamma$-spaces). This implies that $U_Q$ is the
parallel transport of a flat principal connection defined on $Q$,
which we shall call the {\em monodromy connection} of $Q$. The
holonomy morphism:
\be
\alpha_{m}(Q):\pi_1(M,m)\rightarrow \Aut_\Gamma(Q_{m})  
\ee
of this connection at a point $m\in M$ will be called the {\em
	monodromy morphism} of $Q$ at $m$, while its image:
\be
\Hol_{m}(Q)\eqdef \im(\alpha_{m}(Q))\subset \Aut_\Gamma(Q_{m})
\ee
will be called the {\em monodromy group} of $Q$ at $m$.
The monodromy morphism at a fixed point $m_0\in M$ induces a
bijection between the set of isomorphism classes of principal
$\Gamma$-bundles and the character variety:
\be
\fR(\pi_1(M,m_0),\Gamma)\eqdef \Hom(\pi_1(M,m),\Gamma)/\Gamma~~.
\ee

\begin{remark}
For the purposes of this work, the most important class of principal bundles 
with  discrete structure group are principal $\Sp_\frt(2n,\Z)$
bundles, which are naturally associated to Siegel bundles (see Section 
\ref{sec:associatedbundle}). 
\end{remark}


\subsection{Bundles of finitely-generated free abelian groups}


Let $\cF$ be a bundle of free abelian groups of rank $r$ defined on
$M$. Then $\cF$ is isomorphic with the bundle of groups with fiber
$\Z^r$ associated to a principal $\GL(r,\Z)$-bundle $Q$ through
the left action $\ell:\GL(r,\Z)\rightarrow \Aut_\Z(\Z^r)$:
\be
\cF\simeq \cF_r(Q)\eqdef Q \times_\ell \Z^r\simeq \hM \times_{\ell\circ
	\alpha_{m}(Q)} \Z^r~~,
\ee
where $\alpha_{m}(Q):\pi_1(M,m)\rightarrow \GL(r,\Z)$ is the
monodromy morphism of $Q$ at $m$. The monodromy connection of $Q$
induces a flat Ehresmann connection which we shall call the {\em
	monodromy connection of $\cF$} and whose parallel transport:
\be
U_\cF:\Pi_1(M)\rightarrow \Phi(\cF)
\ee
acts by isomorphisms of groups between the fibers
of $\cF$. The holonomy morphism:
\be
\sigma_{m}(\cF):\pi_1(M,m)\rightarrow \Aut_\Z(\cF_m)
\ee
at $m\in M$ can be identified with the morphism $\ell\circ
\alpha_{m}(P):\pi_1(M,m)\rightarrow \Aut_\Z(\Z^r)$ upon choosing
a basis of $\cF_m$. The holonomy group:
\be
\Hol_{m}(\cF)\eqdef \im(\sigma_{m}(\cF))\subset \Aut_\Z(\cF_m)\simeq \GL(r,\Z)
\ee
is called the {\em monodromy group} of $\cF$ at $m$ and identifies
with a subgroup of $\GL(r,\Z)$ upon choosing an appropriate basis of
$\cF_{m}$.

Conversely, let $\Fr(\Z^r)$ be the set of all bases of the free
$\Z$-module $\Z^r$. Then $\GL(r,\Z)$ has a natural free and
transitive left action $\mu$ on this set. Taking the set of bases of
each of fiber gives the {\em bundle of frames} $\Fr(\cF)$ of
$\cF$. This is a principal $\GL(r,\Z)$-bundle whose monodromy
morphism coincides with that of $\cF$. This gives the
following result.

\begin{prop}
The correspondences $Q\mapsto \cF_r(Q)$ and $\cF\mapsto \Fr(\cF)$ extend to
mutually quasi-inverse equivalences between the groupoid of bundles of
free abelian groups of rank $r$ defined on $M$ and the groupoid of
principal $\GL(r,\Z)$-bundles defined on $M$.
\end{prop}


\subsection{Dirac systems and integral duality structures}


\begin{definition}
Let $\Delta=(\cS,\omega,\cD)$ be a duality structure defined on $M$. A {\em
Dirac system} for $\Delta$ is a smooth fiber sub-bundle $\cL\subset
\cS$ of full symplectic lattices in $(\cS,\omega)$ which is preserved
by the parallel transport $\cT_{\Delta}$ of $\cD$. That is,
for any piece-wise smooth path $\gamma:[0,1]\rightarrow M$ we have:
\be
\cT_{\Delta}(\gamma)(\cL_{\gamma(0)}) = \cL_{\gamma(1)}\, .
\ee
A pair:
\be
\bDelta  \eqdef (\Delta, \cL)\, 
\ee
consisting of a duality structure $\Delta$ and a choice of Dirac
system $\cL$ for $\Delta$ is called an {\em integral duality
structure}. 
\end{definition}

\noindent 
Let $\Dual_\Z(M)$ be the groupoid of integral duality structures
defined on $M$, with the obvious notion of isomorphism. For every
$m\in M$, the fiber $(\cS_m,\omega_m,\cL_m)$ of an integral duality
structure $\bDelta = (\Delta, \cL)$ of rank $2n$ is an \emph{integral
symplectic space} of dimension $2n$ (see Appendix \ref{app:symp} for
details). Each such space defines an integral vector (called its
\emph{type}) belonging to a certain subset $\Div^n$ of $\Z_{>0}^n$
endowed with a partial order relation $\leq$ which makes it into a
complete meet semi-lattice. The type of an integral symplectic space
depends only on its isomorphism class (which it determines uniquely)
and every element of $\Div^n$ is realized as a type. Moreover, the
group of automorphisms of an integral symplectic space of type
$\frt\in \Div^n$ is isomorphic with the {\em modified Siegel modular
group} $\Sp_\frt(2n,\Z)$ of type $\frt$ (see Definition
\ref{def:Siegel}). This is a discrete subgroup of $\Sp(2n,\R)$ which
contains the Siegel modular group $\Sp(2n,\Z)$, to which it reduces
when $\frt$ equals the {\em principal type} $\delta=(1,\ldots, 1)$. If
$\frt$ and $\frt'$ are elements of $\Div^n$ such that $\frt\leq
\frt'$, then the lattice $\Lambda$ of any integral symplectic space
$(V,\omega,\Lambda)$ of type $\frt$ admits a full rank sublattice
$\Lambda'$ such that $(V,\omega,\Lambda')$ is an integral symplectic
space of type $\frt'$.  In this case, we have $\Sp_\frt(2n,\Z)\subset
\Sp_{\frt'}(2n,\Z)$.  Since we assume that $M$ is connected and that
$\cD$ preserves $\cL$, the integral symplectic spaces
$(\cS_m,\omega_m,\cL_m)$ are isomorphic to each other through the
parallel transport of $\cD$, hence their type does not depend on the
base-point $m\in M$.

\begin{definition}
The type $\frt_{\bDelta}\in \Div^n$ of an integral duality
structure $\bDelta=(\cS,\omega,\cD,\cL)$ (and of the corresponding
Dirac system $\cL$) is the common type of the integral symplectic
spaces $(\cS_m , \omega_m,\cL_m)$, where $m\in M$.
\end{definition}

\noindent Let $\Dual_\Z^\frt(M)$ be the full subgroupoid of $\Dual_\Z(M)$
consisting of integral duality structures of type $\frt$.

\begin{definition}
A duality structure $\Delta$ is called {\em semiclassical} if it
admits a Dirac system.
\end{definition}
 
\noindent 
Not every duality structure is semiclassical, as the following proposition shows. 

\begin{prop}
\label{prop:obstructiondiracsys}
A duality structure $\Delta=(\cS,\omega,\cD)$ admits a Dirac system of
type $\frt\in \Div^n$ if and only if the holonomy representation of
$\cD$ at some point point (equivalently, at any point) $m\in M$:
\be
\cT_{\Delta}\vert_{\pi_1(M,m)} \colon \pi_1(M,m)\to \Sp(\cS_m, \omega_m)\simeq \Sp(2n,\R)
\ee
can be conjugated so that its image lies inside the modified Siegel modular
group:
\be
\Sp_{\frt}(2n,\Z) \subset \Sp(2n,\R)
\ee
of type $\frt$. In this case, $\Delta$ is semiclassical and the
greatest lower bound of those $\frt\in \Div^n$ with this property is
called the {\em type of $\Delta$} and denoted by
$\frt_\Delta$.
\end{prop}

\begin{proof}
Assume $\Delta = (\cS,\omega,\cD)$ admits a Dirac system $\cL$ of type
$\frt$. Then (as explained in Appendix \ref{app:symp}) the
automorphism group of every fiber $(\cS_m,\omega_m,\cL_m)$ is
isomorphic to $\Sp_{\frt}(2n,\Z)$, which is the automorphism group of
the standard integral symplectic space
$(\R^{2n},\omega_{2n},\Lambda_\frt)$ of type $\frt$. Since $\cL$ is
preserved by the parallel transport of $\cD$, it follows that
we have $\cT_{\Delta}(\pi_1(M,m)) \subset \Sp_{\frt}(2n,\Z)$ after identifying
$(\cS_m,\omega_m,\cL_m)$ with $(\R^{2n},\omega_{2n},\Lambda_\frt)$ and
hence $\Sp(\cS_m , \omega_m,\cL_m)$ with $\Sp(2n,\R)$. The converse
follows immediately from the associated bundle construction.
\end{proof}

\begin{remark}
A duality structure $\Delta=(\cS,\omega,\cD)$ of rank $2n$ admits a
Dirac system $\cL$ of type $\frt=(t_1,\ldots, t_n)\in \Div^n$ if and
only if $M$ admits an open cover $\cU=(U_\alpha)_{\alpha\in I}$ such
that for each $\alpha\in I$ there exists a $\cD$-flat frame
$(e_1^{(\alpha)},\ldots, e_{2n}^{(\alpha)})$ of $\cS\vert_{U_\alpha}$
with the property:
\ben
\label{omegatframe}
\omega(e_i,e_j)=\omega(e_{n+i},e_{n+j})=0\, , \,\, \omega(e_i,e_{n+j})=t_i\delta_{ij}\, , \,\, \omega(e_{n+i},e_j)= -t_i\delta_{ij}\, , \quad \forall\, \, 
i,j=1,\ldots, n\, .
\een
For each $\alpha,\beta\in I$ with $U_\alpha\cap U_\beta\neq\emptyset$ we have:
\be
e_k^{(\beta)}=\sum_{l=1}^{2n}T^{(\alpha\beta)}_{lk} e^{(\alpha)}_l\, , \quad \forall\,\, k=1,\ldots, 2n~~\mathrm{on}~~U_\alpha\cap U_\beta\, ,
\ee
where $T^{(\alpha\beta)}_{lk}\in \Z$ for all $k,l=1,\ldots, 2n$. Furthermore:
\be
\oplus_{k=1}^{2n} \Z e^{(\alpha)}_k(m)=\cL_m\, , \quad \forall\,\,
\alpha\in I\, , \quad \forall\,\, m\in U_\alpha\, ,
\ee
and the matrices $T^{(\alpha\beta)}\eqdef
(T^{(\alpha\beta)}_{kl})_{k,l=1,\ldots, 2n}$ belong to
$\Sp_\frt(2n,\Z)$.
\end{remark}

\noindent 
Every integral duality structure $\bDelta=(\cS,\omega,\cD,\cL)$
on $M$ defines a parallel transport functor:
\be
\cT_{\bDelta}\colon \Pi_1(M) \to \Symp_\Z\, , 
\ee
where $\Symp_\Z$ is the groupoid of integral symplectic spaces defined in
Appendix \ref{app:symp}. This functor associates the integral
symplectic vector space $(\cS_m,\omega_m,\cL_m)$ to every point $m\in
M$ and the isomorphism of symplectic vector spaces $\cT_{\bDelta}(\bc)
\eqdef \cT_\Delta(\bc)$ to every homotopy class $\bc\in
\Pi_1(M)(m,m')$ of curves from $m\in M$ to $m'\in M$. The functor
$\cT_{\bDelta}$ defines a flat system of integral symplectic vector
spaces (that is, a $\Symp_\Z$-valued local system) on $M$. As in
Section \ref{sec:classical}, the correspondence $\bDelta \mapsto
\cT_{\bDelta}$ extends to an equivalence of groupoids:
\be
\cT:\Dual_\Z(M)\xrightarrow{\sim} [\Pi_1(M),\Symp_\Z]
\ee
between $\Dual_\Z(M)$ and the functor groupoid
$[\Pi_1(M),\Symp_\Z]$. Thus one can identify integral duality
structures with $\Symp_\Z$-valued local systems defined on $M$. This
implies the following result, whose proof is similar to that of
Proposition \ref{prop:hol}.

\begin{prop}
For any $m_0\in M$, the set of isomorphism classes of integral duality
structures of type $\frt$ defined on $M$ is in bijection with the
character variety:
\ben
\label{fRt}
\fR(\pi_1(M,m_0),\Sp_\frt(2n,\Z))\eqdef \Hom(\pi_1(M,m_0),\Sp_\frt(2n,\Z))/\Sp_\frt(2n,\Z)~~.
\een
\end{prop}

\noindent For later reference, we introduce the following:

\begin{definition}
An \emph{integral electromagnetic structure} defined on $M$ is a pair:
\be
\bXi \eqdef (\Xi,\cL)\, ,
\ee
where $\Xi=(\cS,\omega,\cD,\cJ)$ is an electromagnetic structure on $M$ and
$\cL$ is a Dirac system for the duality structure $\Delta=(\cS,\omega,\cD)$.
\end{definition}

 
\subsection{Siegel systems}
\label{subsec:Siegel}


Integral duality structures are associated to certain local systems of
free abelian groups of even rank defined on $M$.

\begin{definition}
Let $n\in \Z_{>0}$. A {\em Siegel system} of rank $2n$ on $M$ is a
bundle $Z$ of free abelian groups of rank $2n$ defined on $M$ equipped
with a reduction of its structure group from $\GL(2n,\Z)$ to a
subgroup of some modified Siegel modular group $\Sp_\frt(2n,\Z)$,
where $\frt\in\Div^n$. The greatest lower bound of the set of those
$\frt\in \Div^n$ with this property is called the {\em type} of $Z$
and is denoted by $\frt_Z$.
\end{definition}

\noindent Let $\Sg(M)$ be the groupoid of Siegel systems on $M$ and
$\Sg_\frt(M)$ be the full sub-groupoid of Siegel systems of type
$\frt$.  

\begin{remark}
Let $\underline{\R}$ be the trivial real line bundle on $M$ and
$U_0:\Pi_1(M)\rightarrow \Phi(\underline{\R})$ be the transport
functor induced by its trivial flat connection. The following
statements are equivalent for a bundle $Z$ of free abelian groups of
rank $2n$ defined on $M$:
\begin{enumerate}[(a)]
\item $Z$ is a Siegel system of type $\frt$ defined on $M$.
\item The vector bundle $\cS\eqdef
Z\otimes_\Z\underline{\R}$ carries a symplectic pairing $\omega$ which
is invariant under the parallel transport $U_Z\otimes_\Z U_0$ of the
flat connection induced from $Z$ and which makes the triplet
$(\cS_m,\omega_m,Z_m)$ into an integral symplectic space of type
$\frt$ for any $m\in M$.
\item For any $m\in M$, the $2n$-dimensional vector space
$\cS_{m}\eqdef Z_m\otimes_\Z \R$ carries a symplectic form
$\omega_{m}$ which makes the triplet $(\cS_{m},\omega_{m},Z_{m})$ into
an integral symplectic space of type $\frt$ and we have
$\Hol_{m}(Z)=\Aut(\cS_m,\omega_{m},Z_{m})$.
\end{enumerate}
\end{remark}

\noindent
By definition, any Siegel system $Z$ of type $\frt$ is isomorphic with
the bundle of groups with fiber $\Z^{2n}$ associated to a principal
$\Sp_\frt(2n,\Z)$-bundle $Q$ through the left action
$\ell :\Sp_\frt(2n,\Z)\rightarrow \Aut_\Z(\Z^{2n})$ of
$\Sp_\frt(2n,\Z)$ on $\Z^{2n}$:
\be
Z\simeq Z(Q)\eqdef Q \times_{\ell} \Z^{2n}\simeq \hM \times_{\ell_\frt\circ
  \alpha_{m}(Q)} \Z^{2n}~~,
\ee
where $\alpha_{m}(Q):\pi_1(M,m)\rightarrow \Sp_\frt(2n,\Z)$ is the
monodromy morphism of $Q$ at $m$ and ${\hat M}$ is the universal cover
of $M$. The monodromy morphism $\sigma_{m}(Z)$ of $Z$ at $m$
identifies with $\ell\circ \alpha_{m}(Q)$ upon choosing an integral
symplectic basis of the integral symplectic space
$(\cS_m,\omega_m,Z_{m})$. This also identifies the monodromy group
$\Hol_{m}(Z)\subset \Aut_\Z(Z_{m})$ with
$\Sp_\frt(2n,\Z)$. Conversely, let $\Fr_\frt(\Z^{2n})$ be the set of
those bases of the free $\Z$-module $\Z^{2n}$ in which the standard
symplectic form $\omega_{2n}$ of $\R^{2n}=\Z^{2n}\otimes_\Z \R$ takes
the form $\omega_\frt$ (see Appendix \ref{app:symp}). Then
$\Sp_\frt(2n,\Z)$ has a natural free and transitive left action
$\mu_\frt$ on this set. Taking the set of bases of each of fiber gives
the {\em bundle of frames} $\Fr(Z)$ of the Siegel system $Z$, which is
a principal $\Sp_\frt(2n,\Z)$-bundle. The previous discussion implies
the following result.

\begin{prop}
\label{prop:PrinSiegel}
The correspondences $Q\mapsto Z(Q)$ and $Z\mapsto \Fr(Z)$ extend to
mutually quasi-inverse equivalence of groupoids between
$\Prin_{\Sp_\frt(2n,\Z)}(M)$ and $\Sg_\frt(M)$.
\end{prop}

\noindent In particular, the set of isomorphism classes of Siegel
systems of type $\frt$ defined on $M$ is in bijection with the
character variety $\fR(\pi_1(M,m_0),\Sp_\frt(2n,\Z))$ of equation
\eqref{fRt}. Let $\underline{1}$ be the unit section of the trivial
real line bundle $\underline{\R}=M\times \R$.

\begin{prop}
\label{prop:ZDelta}
Let $Z$ be a Siegel system defined on $M$. Then there exists a unique
integral duality structure $\Delta=(\cS,\omega,\cD,\cL)$ such that
$\cS=Z\otimes_\Z \underline{\R}$ and $\cL=Z\otimes_\Z \Z$ and this
duality structure has the same type as $Z$. Moreover, the parallel
transport $U_\cD:\Pi_1(M)\rightarrow \Phi(\cS)$ of $\cD$ is given by:
\ben
\label{UZ}
U_\cD(\bc)=U_Z(\bc)\otimes_\Z U_0(m,m')\, , \quad \forall\,\, \bc\in \Pi_1(M)(m,m')\, , \quad \forall\,\, m,m'\in M~~,
\een
where $U_Z:\Pi_1(M)\rightarrow \Phi(Z)$ is the monodromy transport of
$Z$ and $U_0:\Pi_1(M)\rightarrow \Phi(\underline{\R})$ is the trivial
transport of $\underline{\R}$. 
\end{prop}

\begin{remark}
\label{rem:Z}
The fiber of $\cL$ at $m\in M$ is given by:
\ben
\label{LZ}
\cL_m \eqdef \{z\otimes_\Z 1 ~\vert ~z\in Z_m\}\equiv Z_m~~, 
\een
where $1$ is the unit element of the field $\R$. It is clear that the
transport $U_\cD$ defined by \eqref{UZ} gives bijections from $\cL_m$
to $\cL_{m'}$ and hence preserves $\cL$. Any locally-constant frame
$(s_1,\ldots, s_{2n})$ of $Z$ defined above a non-empty open set
$V\subset M$ determines a local flat symplectic frame $(e_1,\ldots,
e_{2n})$ of $\Delta$ defined above $V$ given by:
\ben
\label{eidef}
e_i\eqdef s_i\otimes_\Z 1\, , \quad \forall\,\, i=1,\ldots, 2n~~
\een
and the matrix of $\omega$ with respect to this frame is integer-valued.
\end{remark}

\begin{proof}
The restriction of $U_\cD$ to $\cL$ gives isomorphisms of groups
between the fibers \eqref{LZ} of $\cL$ and hence it must agree with
the monodromy transport $U_Z$ of $Z$ in the sense that:
\be
U_\cD(\bc)(z\otimes_\Z 1)=U_Z(\bc)(z)\otimes_\Z 1\, , \quad \forall\,\, \bc\in
\Pi_1(M)(m,m')\, , \quad \forall\,\, m,m'\in M~~.
\ee
This implies \eqref{UZ} since $U_\cD$ is $\R$-linear and $\cL_m$ are
full lattices in $\cS_m=Z_m\otimes_\Z \R$. Remark \ref{rem:Z}
gives a $\cD$-flat symplectic pairing $\omega_Z$ on $\cS_Z$ such that
the integral symplectic spaces $(\cS_m,\omega,\cL_m)$ have type
$\frt_Z$ and such that $\cL$ is preserved by the parallel transport
of $\cD$. 
\end{proof}

\begin{remark}
Notice that $\cS$ identifies with the vector bundle
associated to the frame bundle $\Fr(Z)$ of $Z$ through the linear
representation $q=\varphi\circ \iota:\Sp_\frt(2n,\Z)\rightarrow
\Aut_\R(\R^{2n})$ defined by the inclusion morphism:
\be
\iota:\Sp_\frt(2n,\Z)\hookrightarrow \Sp(2n,\R)~~,
\ee
where $\Sp(2n,\R)$ acts on $\R^{2n}$ through the fundamental
representation $\varphi:\Sp(2n,\R)\rightarrow \Aut_\R(\R^{2n})$.  
The representation $q$ preserves the canonical symplectic form
$\omega_\frt$ of $\R^{2n}$ and the latter induces the symplectic pairing
$\omega$ of $\cS$. 
\end{remark}

\noindent
We denote by $\bDelta(Z)$ the integral duality structure defined by
$Z$ as in Proposition \ref{prop:ZDelta}. Conversely, any integral
duality structure $\bDelta=(\cS,\omega,\cL,\cD)$ defines a Siegel
system $Z(\bDelta)$ upon setting:
\be
Z(\bDelta)\eqdef \cL
\ee
and it is easy to see that $\bDelta$ is isomorphic with $\bDelta(\cL)$.
Therefore, we obtain the following result:

\begin{prop}
The correspondences $Z\mapsto \bDelta(Z)$ and $\bDelta\rightarrow
Z(\bDelta)$ extend to mutually quasi-inverse equivalences of groupoids
between $\Sg_\frt(M)$ and $\Dual^\frt_\Z(M)$.
\end{prop}


\subsection{Bundles of integral symplectic torus groups}


In the following we use the notions of {\em integral symplectic
torus group} discussed in Appendix \ref{app:symp}.

\begin{definition}
A {\em bundle of integral symplectic torus groups of rank $2n$} is a
bundle $\cA$ of $2n$-dimensional torus groups defined on $M$ whose
structure group reduces from $\GL(2n,\Z)$ to a subgroup of some
modified Siegel modular group $\Sp_\frt(2n,\Z)$, where $\frt\in
\Div^n$. The greatest lower bound $\frt_\cA$ of the set
of elements $\frt\in \Div^n$ with this property is called the {\em
type} of $\cA$.
\end{definition}

\noindent
Let $\cA$ be a bundle of integral symplectic torus groups of type
$\frt$. Then the zero elements of the fibers determine a
section $s_0\in \cC^\infty(M,\cA)$. The structure group $\Sp_\frt(2n,\Z)$
acts on $\cA_m$ preserving the distinguished point $s_0(m)$ and the
abelian group structure of each fiber. Since such a bundle is
associated to a principal $\Sp_\frt(2n,\Z)$-bundle, it carries an
induced flat Ehresmann connection whose holonomy group is a subgroup
of $\Sp_\frt(2n,\Z)$ and whose holonomy representation at $m\in M$ we 
denote by:
\ben
\label{rhoX}
\rho_m(\cA):\pi_1(M,m)\rightarrow \Sp_\frt(2n,\Z)\subset \GL(2n,\Z)~~.
\een
The parallel transport of this connection preserves the image of the
section $s_0$ as well as a fiberwise symplectic structure which makes
each fiber into an integral symplectic torus group of type
$\frt$ in the sense of Appendix \ref{app:symp}. We have:
\be
\cA\simeq {\hat M}\times_\rho \left(\R^{2n}/\Z^{2n}\right)~~.
\ee

\begin{remark}
When $M$ is compact, the fiber bundle $\cA$ is virtually trivial by
\cite[Theorem 1.1.]{CMP}, i.e. the pull-back of $\cA$ to some {\em
finite} covering space of $M$ is topologically trivial.
\end{remark}

\noindent It follows from the above that bundles of integral symplectic torus
groups of type $\frt$ defined on $M$ are classified by group morphisms
\eqref{rhoX}, i.e.  the set of isomorphism classes of such bundles is
in bijection with the character variety
$\fR(\pi_1(M,m_0),\Sp_\frt(2n,\Z)$, which also classifies integral
duality structures. This also follows from the results below.

\begin{prop}
Let $\bDelta=(\cS,\omega,\cD,\cL)$ be an integral duality structure of
type $\frt$ defined on $M$. Then the fiberwise quotient:
\be
\cA(\bDelta) \eqdef \cS/\cL\, ,
\ee
is a bundle of integral symplectic torus groups of type $\frt$ defined on $M$. 
\end{prop}

\begin{proof}
It is clear that $\cA(\bDelta)$ is a fiber bundle of even-dimensional
torus groups, whose zero section $s_0$ is inherited from the zero
section of $\cS$. The fiberwise symplectic pairing $\omega$ of $\cS$
descends to a translation-invariant collection of symplectic forms on
the fibers of $\cA(\bDelta)$, making the latter into integral
symplectic torus groups of type $\frt$. Since the parallel transport
of $\cD$ preserves both $\cL$ and $\omega$, this bundle of torus
groups inherits a flat Ehresmann connection which preserves both its
symplectic structure and the image of the section $s_0$ and whose
holonomy group reduces to $\Sp_\frt(2n,\Z)$. In particular, the
structure group of $\cA(\bDelta)$ reduces to $\Sp_\frt(2n,\Z)$.
\end{proof}

\noindent As explained in Appendix \ref{app:symp}, any integral
symplectic torus group $\bA=(A,\Omega)$ of type $\frt$ determines an
integral symplectic space $(H_1(A,\R),\omega,H_1(A,\Z))$, where
$\omega$ is the cohomology class of $\Omega$, viewed as a symplectic
pairing on the vector space $H_1(A,\R)$. 

\begin{prop}
Given a bundle $\cA$ of integral symplectic torus groups of type
$\frt$ defined on $M$, let $\cS_\cA$ be the vector bundle with fiber
at $m\in M$ given by $H_1(\cA_m,\R)$ and $\cL_\cA$ be the bundle of
discrete Abelian groups with fiber at $m\in M$ given by
$H_1(\cA_m,\Z)$. Moreover, let $\omega_\cA$ be the fiberwise
symplectic pairing defined on $\cS_\cA$ through:
\be
\omega_{\cA,m}=\omega_m\, , \quad \forall\,\, m\in M~~,
\ee
where $\omega_m$ is the cohomology class of the symplectic form
$\Omega_m$ of the fiber $\cA_m$. Then the flat Ehresmann connection
of $\cA$ induces a flat linear connection $\cD_\cA$ on $\cS_\cA$ which
makes the quadruplet:
\be
\bDelta(\cA)\eqdef (\cS_\cA,\omega_\cA,\cD_\cA,\cL_\cA)\, ,
\ee
into an integral duality structure of type $\frt$ defined on $M$.
\end{prop}

\begin{proof}
$\cD_\cA$ is the connection induced by the flat Ehresmann connection
of $\cD$ on the bundle of first homology groups of the fibers, which
preserves the bundle $\cL_\cA$ of integral first homology groups of
these fibers. The remaining statements are immediate.
\end{proof}

\noindent The two propositions above imply the following statement.

\begin{prop}
\label{prop:cAclassif}
The correspondences $\bDelta \mapsto \cA(\bDelta)$ and $\cA\mapsto
\bDelta(\cA)$ extend to mutually quasi-inverse equivalences between
the groupoid $\Dual^\frt_\Z(M)$ of integral duality structures of type
$\frt$ defined on $M$ and the groupoid $\T_\frt(M)$ of bundles of
integral symplectic torus groups of type $\frt$ defined on $M$.
\end{prop}

\noindent Combining everything, we have a chain of equivalences of
groupoids:
\be
\Prin_{\Sp_\frt(2n,\Z)}(M)\simeq \Sg_\frt(M)\simeq \T_\frt(M)\simeq \Dual^\frt_\Z(M)~~.
\ee


\subsection{The exponential sheaf sequence defined by a Siegel system}


Let $Z$ be a Siegel system of type $\frt\in \Div^n$ on $M$. Let
$\cS_Z\eqdef Z\otimes_\Z\underline{\R}$ be the underlying vector
bundle of the integral duality structure $\bDelta(Z)$ defined by $Z$
and $\cA_Z\eqdef \cS_Z/Z$ be the associated bundle of integral
symplectic torus groups. The exponential sequence the torus group
$\R^{2n}/\Lambda_\frt\simeq \R^{2n}/\Z^{2n}$ (where the canonical
symplectic lattice $\Lambda_\frt$ of type $\frt$ is defined in
Appendix \ref{app:symp}) induces a short exact sequence of bundles of
abelian groups (where $j$ is the inclusion):
\be
\label{ExpSeq}
0 \rightarrow Z\stackrel{j}{\rightarrow} \cS_Z\stackrel{\exp}{\rightarrow} \cA_Z\rightarrow
0~~.
\ee
In turn, this induces an exact sequence of sheaves of abelian groups:
\be
\label{SheafExpSeq}
0 \rightarrow \cC(Z)\stackrel{j}{\rightarrow}
\cC^\infty(\cS_Z)\stackrel{\exp}{\rightarrow} \cC^\infty(\cA_Z)\rightarrow
0~~,
\ee
where $\cC(Z)$ is the sheaf of continuous (and hence locally
constant) sections of $Z$.  This induces a long exact sequence in
sheaf cohomology, of which we are interested in the following piece:
\ben
\label{LongSeq}
H^1(M,\cC(Z))\stackrel{j_\ast}{\rightarrow}
H^1(M,\cC^\infty(\cS_Z))\stackrel{\exp_\ast}{\rightarrow}
H^1(M,\cC^\infty(\cA_Z))\stackrel{\delta}{\rightarrow}H^2(M,\cC(Z))\rightarrow
H^2(M,\cC^\infty(\cS_Z))\, ,
\een
where $\delta$ is the Bockstein morphism. The sheaf
$\cC^\infty(\cS_Z)$ is fine and hence acylic since $M$ is paracompact
and $\cS_Z$ is a vector bundle on $M$. Thus:
\be
H^j(M,\cS_Z)=0\, , \quad \forall j>0\, ,
\ee
which by the long sequence above implies that $\delta$ is an isomorphism of
groups. We also have $H^\ast(M,\cC(Z))= H^\ast(M,Z)$, where in the
right hand side $Z$ is viewed as a local system of $\Z^{2n}$
coefficients. Hence we can view $\delta$ as an isomorphism of abelian
groups:
\ben
\label{delta}
\delta: H^1(M,\cC^\infty(\cA_Z)) \xrightarrow{\sim} H^2(M,Z)~~.
\een


\subsection{The lattice of charges of an integral duality structure}


For every integral duality structure $\bDelta=(\Delta,\cL)$ on $M$, the 
sheaf $\cC_\fl^\infty(\cA)$ of {\em flat} smooth local 
sections of the bundle $\cA \eqdef \cA(\bDelta)=\cS/\cL$ of integral
symplectic torus groups defined by $\bDelta$ fits into the short exact
sequence of sheaves of abelian groups:
\be
1\to \cC(\cL) \xrightarrow{j_0} \cC^\infty_\fl(\cS) \xrightarrow{\exp} \cC^\infty_\fl(\cA) \to 1\, ,
\ee
where $j$ is the inclusion. This induces a long exact sequence in
sheaf cohomology, of which we are interested in the following terms:
\be
\ldots \rightarrow H^1(M,\cC^\infty_\fl(\cA))\xrightarrow{\delta_0} H^2(M, \cC(\cL)) \xrightarrow{j_{0\ast}} H^2(M,\cC^\infty_\fl(\cS))\xrightarrow{\exp_\ast}  H^2(M,\cC^\infty_\fl(\cA))\rightarrow \ldots~~,
\ee
where $\delta_0$ is the Bockstein morphism.
Notice that $H^\ast(M,\cC(\cL))\simeq H^\ast(M,Z)$, where $Z=\cL$ is the
Siegel system defined by $\cL$, which we view of a local system of $\Z^{2n}$
coefficients. Moreover, we have $ H^2(M,\cC^\infty_\fl(\cA))=H^2(M,\cA_\disc)$, the
right hand side being the cohomology with coefficients in the local system
defined by $\cA$ when the fibers of the
latter are endowed with the {\em discrete} topology. Since
$H^\ast(M,\cC_\fl^\infty(\cS))=H^\ast_\cD(M,\cS)$, the sequence
above can be written as:
\ben
\label{FlatHExpSeq}
\ldots \rightarrow H^1(M,\cA_\disc)\xrightarrow{\delta_0} H^2(M,Z) \xrightarrow{j_{0\ast}} H^2_\cD(M,\cS)\xrightarrow{\exp_\ast}  H^2(M,\cA_\disc)\rightarrow \ldots
\een

\noindent
Denote by $H^2(M,Z)^\tf\subset H^2(M,Z)$ the torsion free part of $H^2(M,Z)$.

\begin{definition}
The lattice:
\be
L_\bDelta\eqdef j_{0\ast}(H^2(M,Z))=j_{0\ast}(H^2(M,Z)^\tf)\subset H^2_\cD(M,\cS)
\ee
is called the {\em lattice of charges} defined the integral duality
structure $\bDelta$. Elements of this lattice are called
\emph{integral cohomology classes} or {\em charges}.
\end{definition}

\begin{prop}
\label{prop:HcDiso}
There exists a natural isomorphism:
\ben
\label{cDiso}
H^k_{\cD}(M,\cS)\simeq H^k(M,Z)\otimes_{\Z[\bpi]} \R\simeq H^k(M,Z)^\tf\otimes_{\Z[\bpi]} \R~~,
\een
for all $k$. In particular, the kernel of $j_{0\ast}$ coincides with the
torsion part of $H^2(M,Z)$ and $j_\ast(H^2(M,Z))$ is
a full lattice in $H^2_\cD(M,\cS)$.
\end{prop}

\begin{proof}
Let $\bpi\eqdef \pi_1(M,m)$ and $\Z[\bpi]$ be the group ring of
$\bpi$.  The universal coefficient theorem for cohomology with local
coefficients of \cite{Greenblatt} gives a short exact sequence:
\be
0\to H^k(M,Z)\otimes_{\Z[\bpi]} \R \to H^k_{\cD}(M,\cS)  \to  \Tor_{\Z[\bpi]}(H^{k+1}(M,Z), \R)\to 0\, ,
\ee
where $\R$ is the $\Z[\bpi]$-module corresponding to the trivial
representation of $\bpi$ in $\R$. Since the latter module is free, we
have $\Tor_{\Z[\bpi]}(H^{k+1}(M,Z), \R)=0$ and the sequence above gives the
natural isomorphism \eqref{cDiso}.
\end{proof}

\begin{remark}
We have a commutative diagram with exact rows:
\be
\label{diag1}
\scalebox{1.0}{
\xymatrix{
0 \ar[r] & \cC(Z)~\ar[d]^{\id} \ar[r]^{j_0}~
&~\cC_\fl^\infty(\cS_Z) ~\ar[r]^{~~\exp} \ar[d]^{\tau}
&\cC_\fl^\infty(\cA_Z) \ar[d]^{\iota}~\ar[r] & 0\\ 0 \ar[r]~ &
~\cC(Z)~\ar[r]^j& ~\cC^\infty(\cS_Z) ~\ar[r]^{~~\exp}&
\cC^\infty(\cA_Z) \ar[r]& 0\\ } }~~
\ee
where $j_0,j$ and $\tau,\iota$ are inclusions. In turn, this
induces a commutative diagram with exact rows:
\be
\label{diag2}
\scalebox{1}{ \xymatrix{ &
H^1(M,Z)~\ar[d]^{\id} \ar[r]^{j_{0\ast}}~
&~H^1_{\cD_P}(M,\cS_Z)\ar[d]^{\tau_\ast} ~\ar[r]^{\exp_\ast} &
H^1(M,\cA_{Z,\disc}) \ar[d]^{\iota_\ast}~\ar[r]^{\delta_0} &
H^2(M,Z)\ar[d]^{\id}\ar[r]^{j_{0\ast}}
&~H^2_{\cD_P}(M,\cS_Z)\ar[d]^{\tau_\ast} \\
 &~H^1(M,Z)~\ar[r]^{j_\ast}& ~0
~\ar[r]^{\exp_\ast}& H^1(M,\cC^\infty(\cA_Z))\ar[r]^{\delta}& H^2(M,Z)
\ar[r]^{j_\ast} & 0 \\ } }
\ee
In particular, we have $\delta_0=\delta\circ\iota_\ast$.
\end{remark}


\subsection{The DSZ integrality condition}


Let $(M,g)$ be an oriented and connected Lorentzian four-manifold.
Given an integral duality structure $\bDelta=(\Delta,\cL)$, we
implement the DSZ condition by restricting the configuration space
$\Conf(M,\Delta)$ to a subsets determined by the charge lattice
$L_\bDelta$. We will show in later sections that integral field
strengths are adjoint curvatures of connections defined on a certain
principal bundle.

\begin{definition}
Let $\bDelta=(\Delta,\cL)$ be an integral duality structure on $(M,g)$
with underlying duality structure $\Delta=(\cS,\omega,\cD)$. The {\em
set of integral electromagnetic field strength configurations} defined
by $\bDelta$ on $M$ is the following subset of $\Conf(M,\Delta)$:
\be
\Conf(M,\bDelta) \eqdef \left\{ \cV \in
\Conf(M,\Delta) \,\, \vert \,\, 2\pi [\cV]_{\cD} \in L_\bDelta\right\}\, ,
\ee
where $[\cV]_{\cD} \in H^2_\cD(M,\cS)$ is the $\dd_\cD$-cohomology
class of the $\cS$-valued two-form $\cV\in \Conf(M,\Delta)$ and
$L_\bDelta\subset H^2_\cD(M,\cS)$ is the lattice of charges defined by
$\bDelta$.
\end{definition}

\begin{definition}
Let $\bXi=(\bDelta,\cJ)$ be an integral electromagnetic structure
defined on $(M,g)$ with electromagnetic structure $\Xi=(\Delta,\cJ)$
and integral duality structure $\bDelta=(\Delta,\cL)$. The {\em set of
integral field strength solutions} defined by $\bXi$ on $(M,g)$ is
the subset of $\Sol(M,g,\Xi)$ defined through:
\be
\Sol(M,g,\bXi)\eqdef \Sol(M,g,\Xi)\cap \Conf(M,\bDelta)
\ee
and hence consists of those elements of $\Conf(M,\bDelta)$ which satisfy
the equations of motion (i.e. the polarized self-duality condition) of
the classical abelian gauge theory defined by $\Xi$.
\end{definition}


\subsection{Integral duality groups}


The DSZ integrality condition restricts the classical duality groups
of Subsection \ref{subsec:classdual} to certain subgroups.

\begin{definition}
Fix an integral duality structure $\bDelta$ on $(M,g)$.
\begin{itemize}
\item The {\em integral unbased pseudo-duality group} defined by
$\bDelta$ is the group $\Aut(\bDelta)\subset \Aut(\Delta)$ formed by
those elements $u\in \Aut(\Delta)$ which satisfy
$u(\cL)=\cL$.
\item The {\em integral unbased duality group} defined by $\bDelta$ is
the subgroup $\Aut(g,\bDelta)$ of $\Aut(\bDelta)$ which covers
$\Iso(M,g)$.
\item The {\em integral duality group} defined by $\bDelta$ is the
subgroup $\Aut_b(\bDelta)$ of $\Aut(\bDelta)$ consisting of those
elements which cover the identity of $M$.
\end{itemize}
\end{definition}

\begin{definition}
Fix an integral electromagnetic structure $\bXi=(\bDelta,\cJ)$ on $(M,g)$.
\begin{itemize}
\item The {\em integral unbased unitary pseudo-duality group} defined
by $\bXi$ is the group:
\be
\Aut(\bXi)\eqdef\{u\in \Aut(\bDelta)~\vert~\cJ_u=\cJ\}
\ee
\item The {\em integral unbased unitary duality group} defined by $\bXi$ is:
\be
\Aut(g, \bXi) \eqdef \left\{ u\in \Aut(\bXi)  \,\, \vert \,\, g_u=g\right\}~~.
\ee
\item The {\em integral unitary duality group} defined by
  $\bXi$ is the subgroup $\Aut_b(\bXi)$ of $\Aut(\bXi)$ consisting of those
  elements which cover the identity of $M$.
\end{itemize}
\end{definition}

\noindent It is easy to check that $\A_u$ with $u$ belonging to the
groups defined above restrict to transformations similar to those of
Subsection \ref{subsec:classdual} between the sets of {\em integral}
configurations and solutions. The discrete duality groups introduced
above are the global counterparts of the discrete duality group
considered in the physics literature on local abelian gauge
theory. The latter is usually taken to be $\Sp(2n,\Z)$ due to the
fact that the symplectic lattice of charges appearing in the local 
treatment of abelian gauge theory is traditionally assumed to have 
principal type $\frt=\delta=(1,\ldots,1)$. As explained in \cite{gesm} 
and recalled in Appendix \ref{app:local}, $\Sp(2n,\Z)$ is not always 
the correct duality group even in the local case, since the local lattice of
charges need not be principal. In Section \ref{sec:globaldualitygroups}, 
we consider a natural gauge-theoretic extension of the discrete duality 
groups defined above, which clarifies the geometric origin of electromagnetic 
duality.

\subsection{Trivial integral duality structures}
\label{subsec:trivD}

Let $Z$ be a trivializable Siegel system of type $\frt\in \Div^n$ and
$\Delta=(\cS,\omega,\cD)$ be the associated duality structure, where
$\cS=Z\otimes_\Z\underline{\R}$. Pick a flat trivialization
$\tau:\cS\xrightarrow{\sim} M\times S$ of $\cS$, where $S\simeq
\R^{2n}$. This takes $\omega$ into a symplectic pairing $\omega_S$
on the vector space $S$ and restricts to an isomorphism $\tau_0:Z\xrightarrow{\sim}
M\times \Lambda$ between $Z$ and $M\times \Lambda$, where $\Lambda$ is
a full symplectic lattice in $(S,\omega_S)$. Let $A\eqdef S/\Lambda$
be the torus group defined by $(S,\Lambda)$ and $\cA\eqdef \cS/Z$ be
the bundle of torus groups defined by $(\cS,Z)$. Then $\tau$ induces a
trivialization $\bar{\tau}:\cA\xrightarrow{\sim} M\times A$ of $\cA$,
which fits into a commutative diagram of fiber bundles:
\be
\scalebox{1.0}{
\xymatrix{
Z~\ar[d]^{\tau_0} \ar[r]^{j}~&~ \cS ~\ar[r] \ar[d]^{\tau}
& \cA \ar[d]^{\bar{\tau}}\\
M\times \Lambda~\ar[r]^{i}& ~M\times S ~\ar[r]& M\times A\\
} }~~
\ee
where $i$ and $j$ are inclusions. Since $\tau$ identifies $\cD$ with
the trivial connection on $M\times S$, it induces an isomorphism
of graded vector spaces:
\be
\tau_\ast:H^\ast_{\cD}(M,\cS)\xrightarrow{\sim} H^\ast(M,S)
\ee
whose restriction coincides with the isomorphism of graded abelian
groups:
\be \tau_{0\ast}:H^\ast(M,Z)\xrightarrow{\sim}
H^\ast(M,\Lambda)
\ee
induced by $\tau_0$. Moreover, $\bar{\tau}$ induces an isomorphism of
graded abelian groups $\bar{\tau}_\ast:H^\ast(M,\cA)\xrightarrow{\sim}
H^\ast(M,A)$. Hence the diagram above induces an isomorphism of long
exact sequences of abelian groups:
\be
\label{trivdiag}
\scalebox{1.0}{
\xymatrix{
\ldots \ar[r] & \ar[r] H^1(M,\cA) \ar[d]^{\bar{\tau}_\ast}& H^2(M,Z)~\ar[d]^{\tau_{0\ast}} \ar[r]^{j_\ast}~
&~H^2(M,\cS) ~\ar[r] \ar[d]^{\tau_\ast} & H^2(M,\cA) \ar[d]^{\bar{\tau}_\ast}~\ar[r] & \ldots \\
\ldots \ar[r] & \ar[r] H^1(M,A) & ~H^2(M,\Lambda)~\ar[r]^{i_\ast} & ~H^2(M,S) ~\ar[r]&
H^2(M,A) \ar[r]& \ldots \\ } }~~.
\ee
Since $\Lambda$ is free while $\cS$ is a vector space, we have isomorphisms of abelian groups:
\be
H^\ast(M,S)\simeq H^\ast(M,\R)\otimes_\R S~~,~~H^\ast(M,\Lambda)\simeq_\Z H^\ast(M,\Z)\otimes_\Z \Lambda~~.
\ee
and:
\be
H^\ast(M,S)\simeq H^\ast(M,\Lambda)\otimes_\Z \R\simeq H^\ast(M,\Lambda)^\tf\otimes_\Z \R~~.
\ee
The latter agrees with the isomorphism \eqref{cDiso} through the maps $\tau_{0\ast}$ and $\tau_\ast$. 
The map $i_\ast:H^k(M,\Lambda)\rightarrow H^k(M,S)$ is obtained by
tensoring the map $H^k(M,\Z)\rightarrow H^k(M,\R)$ with the inclusion
$\Lambda\subset S$, while its restriction
$i_\ast^\tf:H^k(M,\Lambda)^\tf\rightarrow H^k(M,S)$ is obtained by
tensoring the inclusion $\Lambda\subset S$ with the map
$H^l(M,\Z)^\tf\rightarrow H^k(M,\R)$. Since the latter is injective,
it follows that $i_\ast^\tf$ is injective and hence
$H^l(M,\Lambda)^\tf$ identifies with a full lattice in $H^k(M,S)$.
Since $A$ and $(S,+)$ are divisible groups while $H_0(M,\Z)=\Z$ and
$\Lambda$ are free, the universal coefficient sequence for cohomology
gives isomorphisms:
\begin{eqnarray}
\label{isos}
& H^k(M,S)\simeq \Hom_\Z(H_k(M,\Z),S)=\Hom_\Z(H_k(M,\Z)^\tf,S)\, , \nonumber\\
& H^k(M,\Lambda)\simeq \Hom_\Z(H_k(M,\Z),\Lambda)=\Hom_\Z(H_k(M,\Z)^\tf,\Lambda)
\end{eqnarray}
and:
\be
H^k(M,A)\simeq \Hom_\Z(H_k(M,\Z),A)\simeq \Hom_\Z(H_k(M,\Z)^\tf,A)
\ee
for all $k$. The first of these is the period isomorphism:
\be
\per(\omega)(c):=\per_c(\omega)=c\cap \omega=\int_c\omega\, , \quad \forall \,\, \omega\in H^k(M,S)\, , \quad \forall\,\, c\in H_k(M,\Z)\, .
\ee
The map $i_\ast:H^k(M,\Lambda)\rightarrow H^k(M,\cS)$ agrees with the
injective map induced by the inclusion $\Lambda\hookrightarrow S$
though the isomorphisms \eqref{isos}. Hence:
\be
H^k(M,\Lambda)\simeq
\per^{-1}(\Hom_\Z(H_k(M,\Z),\Lambda))=i_\ast(H^k(M,\Lambda))=\{\omega\in
H^k(M,S)~\vert~\per_c(\omega)\in \Lambda\,\&\, c\in H_k(M,\Z)\}~~.
\ee


\section{Siegel bundles and connections}
\label{sec:associatedbundle}


In this section we use the notion of {\em integral affine symplectic torus}, 
for which we refer the reader to Appendix \ref{app:symp}.


\subsection{Automorphisms of integral affine symplectic tori}


We denote by $\Aff_\frt$ the group of affine symplectomorphisms of the
integral affine symplectic torus $\bfA_\frt=(\fA,\Omega_\frt)$ of type
$\frt\in \Div^n$. Here $\fA$ is the underlying $2n$-dimensional affine
torus (which is a principal homogeneous space for the torus group
$\U(1)^{2n}\simeq \R^{2n}/\Z^{2n}$), while $\Omega_\frt$ is the
integral symplectic form of type $\frt$ on $\fA$, which is
translationally-invariant. As explained in Appendix \ref{app:symp},
$\Aff_\frt$ is a non-compact disconnected Lie group whose connected
component of the identity is the $2n$-dimensional torus group
$\U(1)^{2n}$. We have $\pi_0(\Aff_\frt)\simeq \Sp_\frt(2n,\Z)$ and:
\ben
\label{Aff}
\Aff_\frt \simeq \U(1)^{2n} \rtimes \Sp_\frt(2n,\Z)\, ,
\een
where $\Sp_\frt(2n,\Z)$ acts on $\U(1)^{2n}$ through the restriction of the
group morphism defined in equation \eqref{Taction} of Appendix
\ref{app:symp}, an action which we denote by juxtaposition. Thus
$\Aff_\frt$ identifies with the set $\U(1)^{2n}\times \Sp_\frt(2n,\Z)$, endowed
with the composition rule:
\be
(a_1,\gamma_1)\,(a_2,\gamma_2 ) =(a _1+ \gamma_1 a_2,
\gamma_1\gamma_2)\, , \quad \forall \,\, a_1,a_2\in \U(1)^{2n}\, , \quad \forall\,\,
\gamma_1,\gamma_2\in \Sp_\frt(2n,\Z)\, .
\ee
Let $\ell:\Aff(t)\rightarrow \Diff(\Aff_\frt)$ be the left action of
$\Aff_\frt$ on itself:
\be
\ell(g)(g')\eqdef g g'\, , \quad \forall\,\, g,g'\in \Aff_\frt\, ,
\ee
and let $\mathrm{pr}_1:\Aff_\frt\rightarrow \U(1)^{2n}$ and $\mathrm{pr}_2:
\Aff_\frt\rightarrow \Sp_\frt(2n,\Z)$ be the projections of the set-theoretic 
decomposition $\Aff_\frt=\U(1)^{2n}\times \Sp_\frt(2n,\Z)$. Notice that $\mathrm{pr}_2$ 
is a morphism of groups. Define left actions $\ell_1$ and $\ell_2$ of $\Aff_\frt$
on $\U(1)^{2n}$ and $\Sp_\frt(2n,\Z)$ through:
\be
\ell_1(g)(a)\eqdef \mathrm{pr}_1(\ell(g)(a,1))\, ,\quad \ell_2(g)(\gamma)\eqdef \mathrm{pr}_2(\ell(g)(0,\gamma))\, ~,~\forall g\in \Aff_\frt~,~\forall a\in \U(1)^{2n}~,~\forall \gamma\in \Sp_\frt(2n,\Z)~~.
\ee
Then $\ell_1$ is given by:
\ben
\label{ell1}
\ell_1(a,\gamma)(a')=a+\gamma a'\, , \quad \forall\,\, \gamma \in \Sp_\frt(2n,\Z)\, , \quad \forall\, \, a, a'\in \U(1)^{2n} ~~.
\een
This action is transitive with stabilizer isomorphic with
$\Sp_\frt(2n,\Z)$. On the other hand, $\ell_2$ is given by:
\ben
\label{ell2}
\ell_2(a,\gamma)(\gamma')=\gamma\gamma'=p_2(a,\gamma)\gamma'\, , \quad \forall \,\, \gamma,\gamma' \in \Sp_\frt(2n,\Z)\, , \quad \forall\,\, a\in \U(1)^{2n}\, 
\een
and is transitive with stabilizer isomorphic to $\U(1)^{2n}$. This gives
the right-split short exact sequence:
\be
0\rightarrow \U(1)^{2n}\rightarrow \Aff_\frt\rightarrow \Sp_\frt(2n,\Z)\rightarrow 1~~.
\ee
Notice that $\ell=\ell_1\times \ell_2$. The Lie algebra $\aff_\frt$ of $\Aff_\frt$ 
is abelian and coincides with the Lie algebra of $\U(1)^{2n}$:
\be
\aff_\frt =\R^{2n}\simeq H_1(\fA_\frt,\R)~~.
\ee
The exponential map $\exp:\aff_\frt\rightarrow \Aff_\frt$ has kernel
$\Lambda_\frt\simeq H_1(\fA_\frt,\Z)$ and image $A$, giving the
exponential sequence:
\ben
\label{ExpSeq0}
0\rightarrow  \Lambda_\frt \rightarrow \aff_\frt\stackrel{\exp}{\rightarrow} \U(1)^{2n} \rightarrow 0~~.
\een

\begin{lemma}
\label{lemma:adjointrep}
The adjoint representation $\Ad\colon \Aff_\frt \to \GL(2n,\R)$ of 
$\Aff_\frt$ coincides with its fundamental linear representation, 
that is:
\ben
\label{adjoint}
\Ad(a,\gamma)(v)=\gamma(v)\, , \quad \forall\,\, (a,\gamma)\in \Aff_\frt\, , \quad \forall \,\, v\in \R^{2n}\, .
\een
In particular, we have $\Ad=j \circ \mathrm{pr}_2$, where 
$j:\Sp_\frt(2n,\Z)\rightarrow \GL(2n,\R)$ is the fundamental 
representation of $\Sp_\frt(2n,\Z)$.
\end{lemma}

\begin{proof}
Let $\alpha = (\gamma,1) \colon I\to \Aff(\frt) =\U(1)^{2n}\times \Sp_\frt(2n,\Z)$ be a smooth path on $\Aff(\frt)$ such that $\alpha(0) = \mathrm{Id}$. Set:
\be
\frac{\dd}{\dd t} \alpha(t)\vert_{t=0} = v\in \R^{2n}\, .
\ee
	
\noindent
For every $x=(x_1,x_2)\in \Aff(\frt) = \U(1)^{2n}\times\Sp_\frt(2n,\Z)$ we have:
\be
x \, \alpha(t) \, x^{-1} = (x_1,x_2)\, (\gamma, 1) \, (- x^{-1}_2 x_1, x_2^{-1}) = (x_2  \gamma , 1)\, .
\ee
	
\noindent
Hence:
\be
\frac{\dd}{\dd t} (x\,\alpha(t)\, x^{-1})\vert_{t=0} = x_2(v)\, ,
\ee
	
\noindent
which immediately implies:
\be
\Ad(x) = j\circ \mathrm{pr}_2(x)\, ,
\ee
	
\noindent
for every $x\in\Aff(\frt)$ and hence we conclude. 
\end{proof}


\subsection{Siegel bundles}


Let $M$ be a connected manifold.

\begin{definition}
\label{def:affinesymplecticb}
A {\em Siegel bundle} $P$ of rank $n$ and type $\frt\in \Div^n$ is a
principal bundle on $M$ with structure group $\Aff_\frt$. An {\em
isomorphism of Siegel bundles} is a based isomorphism of principal bundles.
\end{definition}

\noindent Let $\Sieg(M)$ be the groupoid of Siegel bundles defined on
$M$ and $\Sieg_\frt(M)$ be the full subgroupoid of Siegel bundles of type $\frt$.
Fix a Siegel bundle $P$ of type $\frt\in \Div^n$, whose projection
we denote by $\pi$. We introduce several fiber bundles 
associated to $P$.


\subsubsection{The bundle of integral affine symplectic tori defined by $P$}


\begin{definition}
A fiber bundle $\mathfrak{A}$ defined on $M$ is called a {\em bundle
  of integral affine symplectic tori of rank $n$} if its fibers are
$2n$-dimensional tori and the structure group of $\mathfrak{A}$
reduces to a subgroup of $\Aff_\frt$ for some $\frt\in \Div^n$. The
smallest element $\frt_\mathfrak{A}\in \Div^n$ with this property is
called the {\em type} of $\mathfrak{A}$.
\end{definition}

\noindent Notice that bundles of integral symplectic torus groups
coincide with those bundles of integral affine symplectic tori which
admit a smooth global section. Indeed, such a section gives a further
reduction of structure group from $\Aff_\frt$ to
$\Sp_\frt(2n,\Z)$. Given a Siegel bundle $P$ of type $\frt\in \Div^n$
defined on $M$, the fiber bundle:
\be
\mathfrak{A}(P)\eqdef P\times_{\ell_1} \U(1)^{2n}
\ee
associated to $P$ through the action \eqref{ell1} is a bundle of
integral affine symplectic tori of type $\frt$. The fibers of the
latter admit integral symplectic forms of type $\frt$ which vary
smoothly over $M$. 
The group $\Sp(V,\omega)$ acts freely and transitively on
the set $\Fr(V,\omega,\Lambda)$ of integral symplectic bases of
any integral symplectic space $(V,\omega,\Lambda)$.
Any bundle $\mathfrak{A}$ of integral affine symplectic tori of type $\frt$ is
associated through the action $\ell_1$ to its Siegel bundle $P(\mathfrak{A})$
of {\em unpointed torus symplectic frames}, which has type $\frt$ and
whose fiber at $m\in M$ is defined through:
\be
P(\mathfrak{A})_m \eqdef \Fr(H_1(\mathfrak{A}_m,\R),H_1(\mathfrak{A}_m,\Z),\omega_m)\times
\mathfrak{A}_m~~.
\ee
Here $\omega_m\eqdef [\Omega_m]\in H^2(\mathfrak{A}_m,\R)\simeq \wedge^2 H_1(\mathfrak{A}_m,\R)^\vee$
is the cohomology class of the symplectic form $\Omega_m$ of $\mathfrak{A}_m$,
viewed as a symplectic pairing defined on $H_1(\mathfrak{A}_m,\R)$.  More
precisely, we have:

\begin{prop}
\label{prop:SiegelBundlesCorrespondence}
The correspondences $P\rightarrow \mathfrak{A}(P)$ and $\mathfrak{A}\rightarrow
P(\mathfrak{A})$ extend to mutually quasi-inverse equivalences of groupoids
between $\Sieg(M)$ and the groupoid of bundles of integral affine symplectic tori  
and these equivalences preserve type.
\end{prop}

\noindent This statement parallels a similar correspondence which
holds for affine torus bundles (see \cite{Baraglia1} as well as
Theorem 2.2. and Remark 2.3 in \cite{Baraglia2}).


\subsubsection{The Siegel system defined by $P$}


Given a Siegel bundle of type $\frt\in \Div^n$, consider the bundle of
discrete abelian groups defined through:
\be
Z(P)_m\eqdef H_1(\mathfrak{A}(P)_m,\Z)\, , \quad \forall\,\, m\in M~~.
\ee
Since torus translations act trivially on $H_1(\mathfrak{A}(P)_m,\Z)$,
the structure group of $Z(P)$ reduces to $\Sp_\frt(2n,\Z)$. Thus
$Z(P)$ is a Siegel system on $M$. Moreover, $Z(P)$ is isomorphic
with the bundle of discrete abelian groups associated to $P$ through the
projection morphism $p_2:\Aff_\frt\rightarrow \Sp_\frt(2n,\Z)$, when
the latter is viewed as a left action of $\Aff_\frt$ through
automorphisms of the group $(\Z^{2n},+)$.

\begin{definition}
$Z(P)$ is called the {\em Siegel system defined by $P$}.
\end{definition}

\noindent Notice that the the monodromy of $P$ at a point $m\in M$ acts
through automorphisms of the integral symplectic space
$(H_1(\mathfrak{A}(P)_m,\R), H_1(\mathfrak{A}(P)_m,\Z),[\Omega_m])$.
 

\subsubsection{The adjoint bundle and integral duality structure of $P$}


The adjoint bundle $\ad(P)$ of $P$ can be identified with the
tensor product $Z(P)\otimes_\Z \underline{\R}$, whose fiber at $m\in
M$ is given by:
\be
\ad(P)_m=Z(P)_m\otimes_\Z \R\simeq H_1(\mathfrak{A}(P)_m,\R)~~.
\ee
Notice that $\ad(P)$ carries the fiberwise symplectic pairing
$\omega_P$ given by $(\omega_P)_m\eqdef [\Omega_m]$ for all $m\in M$
(see Lemma \ref{lemma:adjointrep}).  Since the Lie algebra of
$\Aff_\frt$ is abelian, the structure group of $\ad(P)$ reduces to
$\Sp_\frt(2n,\Z)$. The Siegel system $Z(P)$ is naturally a sub-bundle
of $\ad(P)$ whose fibers $Z(P)_m=H_1(\mathfrak{A}(P)_m,\Z)$ are full
symplectic lattices with respect to $[\Omega_m]$. The monodromy of
$Z(P)$ induces a unique flat connection $\cD_P$ on $\ad(P)$ whose
parallel transport preserves $Z(P)$. Setting $\cS_P\eqdef \ad(P)$, it
follows that the system:
\be
\bDelta(P)\eqdef (\Delta(P),Z(P))~~,~~\mathrm{where}~~ \Delta(P)\eqdef (\cS_P,\omega_P,\cD_P)\, ,
\ee
is an integral duality structure of type $\frt$, whose underlying duality structure 
is $\Delta(P)$.
  
\begin{prop}
There correspondence defined above extends to an essentially surjective functor:
\be
\bDelta\colon \Sieg(M) \to \Dual_\Z(M)\, ,
\ee
which associates to every Siegel bundle $P$ of type $\frt\in \Div^n$
defined on $M$ the integral duality
structure $\bDelta(P)$, which has type $\frt$.
\end{prop}

\begin{proof}
It is clear that the correspondence extends to a functor. Given
$\bDelta=(\Delta,\cL)\in \Dual_\Z(M)$, denote by $Q$ the frame bundle
of the Siegel system defined by $\cL$, which is a principal
$\Sp_{\frt}(2n,\Z)$-bundle (see Proposition
\ref{prop:PrinSiegel}). Let $P$ be the Siegel bundle associated to
$Q$ through the natural left action $l$ of $\Sp_{\frt}(2n,\Z)$ on
$\Aff_\frt$:
\be
P= Q\times_l \Aff_\frt\, .
\ee
Then $\bDelta(P) = \bDelta$, showing that the functor is essentially surjective. 
\end{proof}


\subsubsection{The bundle of integral symplectic torus groups defined by $P$}


Consider a Siegel bundle $P$ defined on $M$.

\begin{definition}
The {\em bundle of integral symplectic torus groups} defined by $P$ is
the bundle:
\be
\cA(P)=\cA(\bDelta(P))=\ad(P)/Z(P)
\ee
of integral symplectic torus groups defined by the integral duality
structure $\bDelta(P)$.
\end{definition}

\subsubsection{Siegel bundles with trivial monodromy}

\begin{prop}
\label{prop:SimplyConnected} Let $P$ be a Siegel bundle of rank $n$
and type $\frt\in \Div^n$ defined on $M$. Then the following
statements are equivalent:
\begin{enumerate}[(a)]
\item The Siegel system $Z(P)$ has trivial monodromy.
\item $Z(P)$ is trivial as a bundle of discrete Abelian groups.
\item The structure group of $P$ reduces to the torus group
$\U(1)^{2n}$.
\item The structure group of the bundle of integral symplectic affine
tori $\mathfrak{A}(P)$ reduces to $\U(1)^{2n}$.
\item The structure group of the bundle of integral symplectic torus
groups $\cA(P)$ is trivial.
\item The duality structure $\Delta(P)$ is holonomy-trivial.
\end{enumerate}
In this case, $\mathfrak{A}(P)$ identifies with a principal torus bundle 
and $\cA(P)$ is a trivial bundle of integral symplectic torus
groups. Moreover, $P$ is isomorphic with the fiber product
$\mathfrak{A}(P)\times_M \underline{\Gamma}$, where $\underline{\Gamma}$ 
is the trivial $\Sp_\frt(2n,\Z)$-bundle defined on $M$. Thus $P$
identifies with a countable collection of copies of the principal
torus bundle $\mathfrak{A}(P)$, indexed by elements of $\Sp_\frt(2n,\Z)$.
\end{prop}

\begin{proof}
The fact that $(a)$ implies $(b)$ follows from the standard characterization 
of flat bundles in terms of holonomy representations of the fundamental group of 
the underlying manifold. If $Z(P)$ is trivial as a bundle of discrete groups then
the holonomy representation preserves a global frame of $Z(P)$, which in turn implies,
using the explicit form of the adjoint representation of $\Aff_{\frt}$, that the 
holonomy representation takes values in $\U(1)^{2n} \subset \Aff_{\frt}$. The
associated holonomy bundle defines a reduction of $P$ to a principal torus bundle with 
structure group $\U(1)^{2n}$. This immediately implies $(d)$, $(e)$ and $(f)$. Since 
the flat connection on $\Delta(P)$ is by definition the real linear extension of
the flat connection of $Z(P)$ the latter has trivial monodromy if and only if
$\Delta(P)$ has trivial monodromy, that is, if and only if $\Delta(P)$ is holonomy-trivial.
This proves $(f) \Rightarrow (a)$.
\end{proof}


\subsection{Classification of Siegel bundles}
\label{subsec:SiegelClassif}


Let $P$ be a Siegel bundle of type $\frt\in \Div^n$ defined on $M$ and
$\bDelta:=\bDelta(P)=(\cS_P=\ad(P),\omega_P,\cD_P,Z(P))$ be the
integral duality structure defined by $P$. The Bockstein isomorphism
\eqref{delta} reads:
\be
\delta: H^1(M,\cC^\infty(\cA(P))) \xrightarrow{\sim} H^2(M,Z(P))~~.
\ee
It was shown in \cite{Baraglia1} that $P$ determines a {\em primary characteristic class} $c'(P)\in
H^1(M,\cC^\infty(\cA(P)))$.
\begin{definition}
The {\em twisted Chen class} of $P$ is:
\be
c(P)\eqdef \delta(c'(P))\in H^2(M,Z(P))~~.
\ee
\end{definition}

\noindent
Recall from Proposition \ref{prop:SiegelBundlesCorrespondence} that
isomorphism classes of Siegel bundles defined on $M$ are in bijection
with isomorphism classes of bundles of integral affine symplectic
tori.  This allows one to classify Siegel bundles by adapting the
classification of affine torus bundles given in \cite[Section
2]{Baraglia1} (see also \cite[Theorem 2.2]{Baraglia2}). Since the
modifications of the argument of loc. cit. are straightforward, we
simply describe the result. Adapting the argument of \cite{Baraglia1}
we obtain:

\begin{thm}
\label{thm:SiegelClassif}
Consider the set:
\be
\Sigma(M) \eqdef \left\{ (Z,c) \, \vert \, Z\in \Ob[\Sg(M)] \,\, \& \,\, c\in H^2(M,Z) \right\}/_{\sim},
\ee
where $(Z,c)\sim (Z',c')$ if and only if there exists an isomorphism
of Siegel systems $\varphi:Z\rightarrow Z'$ such that
$\varphi_\ast(c)=c'$. Then the map:
\be
P\mapsto (Z(P), c(P))
\ee
induces a bijection between the set of isomorphism classes of Siegel
bundles defined on $M$ and the set $\Sigma(M)$.
\end{thm}

\noindent 
A more conceptual explanation of this result is given in
\cite{wa}. Let $\rho:\Sp_\frt(2n,\Z)\rightarrow\Aut(\U(1)^{2n})$
denote the action of $\Sp_\frt(2n,\Z)$ on $\U(1)^{2n}$ and
$\rho_0:\Sp_\frt(2n,\Z)\rightarrow \Aut_\Z(\Z^{2n})$ be the
corresponding action on $\Z^{2n}$. Then the classifying space
of the $\Aff_{\frt}$ is a {\em twisted Eilenberg-McLane space}
$L:=L_{\Sp_\frt(2n,\Z)}(\Z^{2n},2)$ in the sense of
\cite{Gitler}, which is a sectioned fibration over $\rB\Gamma\simeq
K(\Gamma,1)$ whose fibers are homotopy equivalent with $\rB
\U(1)^{2n}\simeq K(\Z^{2n},2)\simeq (\C\P^\infty)^{\times 2n}$.
This space is a homotopy two-type with:
\be
\pi_1(L)=\Sp_\frt(2n,\Z)~~,~~\pi_2(L)=\Z^{2n}~~.
\ee
The results of \cite{Gitler} are used in \cite{wa} to show that
isomorphism classes of principal $\Aff_\frt$-bundles $P$ defined on a
pointed space $X$ are in bijection with isomorphism classes of pairs
$(\alpha,c)$, where $\alpha:\pi_1(X)\rightarrow \Sp_{\frt}(2n,\Z)$ is
a morphism of groups and $c\in H^2(X,Z_\alpha)$, where $Z_\alpha$ is
the local system with fiber $\Z^{2n}$ and monodromy action at the
basepoint of $X$ given by $\rho_0\circ \alpha:\pi_1(X)\rightarrow
\Aut_\Z(\Z^{2n})$.  When $X=M$ is a manifold, this local system
coincides with $Z(P)$, while the cohomology class $c$ coincides with
$c(P)$.


\subsection{Principal connections on Siegel bundles}


Let $P$ be a Siegel bundle of type $\frt\in \Div^n$ defined on a
connected manifold $M$, whose projection we denote by $\pi:P\rightarrow M$.  For
ease of notation, we set $G\eqdef \Aff_\frt$, $\Gamma\eqdef
\Sp_\frt(2n,\Z)$ and $A\eqdef \U(1)^n$. We denote the abelian Lie
algebra of $A$ by $\fa\simeq \R^{2n}$.  Let $\bDelta=(\Delta,Z)$ be
the integral duality structure $\bDelta(P)$ determined by $P$, where
$\Delta=\Delta(P)=(\cS,\omega,\cD)$ with $\cS=\cS_P=\ad(P)$,
$\omega=\omega_P$ and $\cD=\cD_P$ and $Z=Z(P)$ is the Siegel system
determined by $P$. Let:
\be
\Conn(P)=\left\{\cA\in \Omega^1(P,\fa) \,\, \vert \,\,
r_{a,\gamma}^\ast(\cA)=\gamma^{-1} \cA~~\,\&\,~~\cA_y(X^a_y)=a~~\,\forall
y\in P~\,\,\forall \,\, (a,\gamma)\in G\right\}\, ,
\ee
be the set of principal connections on $P$, where $r_g$ denotes the
right action of $g\in G$ on $P$ and we used the fact that the
adjoint representation $\Ad:G\rightarrow \Aut(\fa)$ is given by
\eqref{adjoint}. Let:
\be
\Omega_\Ad(P,\fa)\eqdef \left\{\eta\in \Omega(P,\fa)\,\, \vert \,\,
r_{a,\gamma}^\ast(\eta)=\gamma^{-1}\eta\,\, \& \,\, \iota_X\eta=0\,\, \forall
(a,\gamma)\in G~\forall X\in V(P)\right\}\, , 
\ee
be the set of equivariant horizontal forms on $P$, where $V(P)$ is the
space of vertical vector fields defined on $P$. Then
$\Omega_\Ad(P,\fa)$ is naturally isomorphic with $\Omega(M,\cS)$. In
particular, the curvature $\Omega_\cA=\dd_\cA\cA\in
\Omega^2_\Ad(P,\fa)$ of any principal connection $\cA\in \Conn(P)$
identifies with a $\cS$-valued two-form $\cV_\cA\in \Omega^2(M,\cS)$,
which is the adjoint curvature of $\cA$. Since $\fa$ is an abelian Lie
algebra, we are in the situation considered in \cite{wa}. Hence the
covariant exterior derivative defined by $\cA$ restricts to the
ordinary exterior derivative on the space $\Omega_\Ad(P,\fa)$:
\ben
\label{dB}
\dd_\cA\vert_{\Omega_\Ad(P,\fa)}=\dd:
\Omega_\Ad(P,\fa)\rightarrow \Omega_\Ad(P,\fa)~~.
\een
Moreover, the principal curvature of $\cA$ is given by:
\be
\Omega_\cA=\dd\cA
\ee
and the Bianchi identity $\dd_\cA\Omega_\cA=0$ reduces to:
\ben
\label{Bianchi}
\dd \Omega_\cA=0~~.
\een
As explained in \cite{wa}, relation \eqref{dB}
implies that all principal connections on $P$ induce the same linear
connection on the adjoint bundle $\ad(P)=\cS$, which coincides with
the flat connection $\cD$ of the duality structure $\Delta$
defined by $P$. Moreover, the adjoint curvature $\cV_\cA\in
\Omega^2(M,\cS)$ satisfies:
\be
\dd_{\cD}\cV_\cA=0~~.
\ee

\noindent
Let $\Sieg^{c}(M)$ be the groupoid of Siegel bundles with connection,
whose objects are pairs $(P,\cA)$ where $P$ is a Siegel bundle and
$\cA\in \Conn(P)$ and whose morphisms are connection-preserving based
isomorphisms of principal bundles. Let $\Dual_\Z^c(M)$ be the groupoid
of pairs $(\bDelta,\cV)$, where $\bDelta$ is an integral duality
structure on $M$ and $\cV\in \Conf(M,\bDelta)$ is a $\dd_\cD$-closed
$\cS$-valued 2-form whose $\dd_\cD$-cohomology class belongs to the
charge lattice of $\bDelta$. There exists a natural functor:
\be
\bDelta^c\colon \Sieg^{c}(M) \to \Dual_\Z^c(M)
\ee
which sends $(P,\cA)$ to the pair $(\bDelta(P),\cV_\cA)$.  Let
$\Sieg^{c}(M)_0\subset\Sieg^{c}(M)$ be the full subgroupoid consisting
of Siegel bundles with flat connection. 

\begin{thm}
\label{thm:diracquantization}
There exists a short exact sequence of groupoids:
\be
1 \to \Sieg^{c}(M)_0 \xrightarrow{\kappa} \Sieg^{c}(M)
\xrightarrow{\curv} \Dual_\Z^c(M)  \to 1~~,
\ee
where $\kappa$ is the inclusion and $\curv$ is the curvature map. In
particular, for every integral duality structure $\bDelta$ on $M$ and
every $\cV\in \Conf(M,\bDelta)$, there exists a Siegel bundle with
connection $(P,\cA)$ such that:
\be
\cV_\cA  = \cV\, ,
\ee
and the set of pairs $(\cA,\cV_\cA)$ with this property is a torsor
for $\Sieg^{c}(M)_0$.
\end{thm}

\begin{proof}
It is clear that an object in $\Sieg^{c}(M)$ defines an integral duality 
structure and a cohomology class in $H^2_{\cD}(M,\cS)$, whence it defines an 
object in $\Sieg^{c}(M)$. Functoriality of this assignment follows from invariance
of the aforementioned cohomology class under gauge transformations. This is proved 
in Lemma \ref{eq:realclassinvariant}. The key ingredient of the proof is now to show 
that this cohomology class is in fact \emph{integral}, that is, belongs to 
$j_\ast(H^2(M,Z))$, where $Z$ is the Siegel system  defined by $P$. This is a
technical point which is proved in detail in \cite{wa}, to which we refer the reader 
for more details. Once this is proven, it follows from  Theorem \ref{thm:SiegelClassif}  
that the curvature map is surjective onto  $\Dual_\Z^c(M)$ and that its kernel is  
precisely the pairs of integral duality structures and flat connections.
\end{proof}

\noindent The previous theorem shows that \emph{integral} electromagnetic 
field strengths can always be realized as curvatures of principal connections 
defined on Siegel bundles, which therefore provide the geometric realization of
integral configurations of abelian gauge theory.

\begin{remark}
Theorem \ref{thm:diracquantization} can be elaborated to obtain the Dirac
quantization of abelian gauge theory in terms of a certain
\emph{twisted} differential cohomology theory, though we do not pursue
this here. Recall that the DSZ quantization of various gauge
theories using the framework of algebraic quantum field theory and 
differentiable cohomology has been considered before in the literature,
see \cite{Szabo:2012hc,BBSS} and references therein.
\end{remark}


\subsection{Polarized Siegel bundles and polarized self-dual connections}


Let $M$ be a connected manifold.

\begin{definition}
A {\em polarized Siegel bundle} is a pair $\bP=(P,\cJ)$, where
$P$ is a Siegel bundle and $\cJ$ is a taming of the duality
structure $\Delta:=\Delta(P)$ defined by $P$.
\end{definition}

\noindent A polarized Siegel bundle $\bP=(P,\cJ)$ determines an
integral electromagnetic structure $\bXi_\bP\eqdef (\bDelta(P),\cJ)$,
where $\bDelta(P)=(\Delta(P),Z(P))$ is the integral duality structure
defined by $P$.

\begin{definition}
Let $\bP=(P,\cJ)$ be a polarized Siegel bundle. A principal connection $\cA
\in \Conn(P)$ is called \emph{polarized selfdual}, respectively
\emph{polarized anti-selfdual} if its adjoint curvature satisfies:
\be
\star_{g,\cJ}\cV_{\cA} = \cV_{\cA}\, ~~,~~\mathrm{respectively}~~\star_{g,\cJ}\cV_{\cA}  = - \cV_{\cA}\, .
\ee
\end{definition}

\noindent Recall the definitions:
\be
\Omega^2_{\pm,\cJ}(M,\cS)\eqdef \left\{\cV\in \Omega^2(M,\cS)\, \vert \, \star_{g,\cJ}\cV_{\cA} = \pm \cV_{\cA} \right\}~~,
\ee
where $\Delta(P)=(\cS,\omega,\cD)$.

\begin{remark}
The polarized selfduality condition is a first-order partial
differential equation for a connection on a Siegel bundle
which, to the best of our knowledge, has not been studied in the
literature on mathematical gauge theory.
\end{remark}


\section{Prequantum abelian gauge theory}
\label{sec:DSZquantization}


Let $(M,g)$ be an oriented and connected Lorentzian four-manifold.  As
explained in the previous section, imposing the DSZ integrality condition on
an abelian gauge theory allows us to identify its prequantum gauge degrees of
freedom with principal connections on a Siegel bundle. More
precisely, let $\bP=(P,\cJ)$ be a polarized Siegel bundle on $(M,g)$
and $\bDelta:=\bDelta(P)=(\Delta,Z)$ be the integral duality structure
defined by $P$, where $\Delta:=\Delta(P)=(\cS,\omega,\cD)$ (with
$\cS=\ad(P)$) and $Z:=Z(P)$ are the duality structure and Siegel
system defined by $P$. Let $\bXi:=\bXi(P)=(\bDelta,\cJ)$ be the
integral electromagnetic structure defined by $\bP$ and
$\Xi:=\Xi(P)=(\Delta,\cJ)$ be its underlying electromagnetic
structure. By Theorem \ref{thm:diracquantization}, the set of integral
field strength configurations determined by the integral duality
structure $\bDelta=\bDelta(P)$ coincides with the set of adjoint
curvatures of principal connections defined on $P$:
\be
\Conf(M,\bDelta)=\{\cV_\cA\,\vert\,\cA\in \Conn(P)\}~~,
\ee
while the set of integral field strength solutions determined by
$\bXi:=\bXi(P)=(\bDelta,\cJ)$ is:
\be
\Sol(M,g,\bXi)=\{\cV_\cA\,\vert\,\cA\in \Conn(P)~\& ~\star_{g,\cJ} \cV_\cA = \cV_\cA\}~~.
\ee
This motivates the following.

\begin{definition}
\label{def:confabeliangauge}
The {\em set of prequantum gauge configurations} determined by $P$ is the affine set:
\be
\cConf(M,P)\eqdef \Conn(P)
\ee
of principal connections defined on $P$. The adjoint curvature $\cV_\cA\in
\Conf(M,\bDelta)$ of a principal connection $\cA\in \Conn(P)$ is
called the {\em integral field strength configuration} defined $\cA$.
The {\em prequantum abelian gauge theory} defined by
$\bP$ on $(M,g)$ is described by the condition:
\ben
\label{eq:selfdualsymplectic}
\star_{g,\cJ} \cV_\cA = \cV_\cA~~\mathrm{for}~~ \cA\in \Conn(P)\, .
\een
The solutions $\cA$ of this equation are called {\em gauge potentials}
or {\em polarized self-dual connections} and form the set:
\be
\cSol(M,g,\bP)=\cSol(M,g,P,\cJ)\eqdef \{\cA\in \Conn(P)\,\vert\, \star_{g,\cJ} \cV_\cA = \cV_\cA\}\subset \cConf(M,P)~~.
\ee
\end{definition}

\noindent We have:
\be
\Sol(M,g,\bXi)=\{\cV_\cA\,\vert\,\cA\in\cSol(M,g,\bP)\}
\ee
and the adjoint curvature map of $P$ gives surjections:
\be
\cConf(M,P) \to \Conf(M,\bDelta)~~\mathrm{and}~~\cSol(M,g,\bP) \to \Sol_\Z(M,g,\bXi)~~.
\ee

\begin{remark}
Condition \eqref{eq:selfdualsymplectic} reduces locally to a system of
first-order differential equations for $2n$ real-valued one-forms,
which describe the local electromagnetic and magnetoelectric potentials
of the theory (see Appendix \ref{app:local}). Also notice that any
$\cA\in \cSol(M,g,\bP)$ satisfies the following second order
equation of Abelian Yang-Mills type:
\be
\dd_\cD \star_{g,\cJ}\cV_{\cA} = 0\, .
\ee
\end{remark}


\subsection{The duality hierarchy of prequantum abelian gauge theory}
\label{sec:globaldualitygroups}


In this subsection we discuss the duality groups of prequantum abelian
gauge theory. Let $P$ be a Siegel bundle of type $\frt\in \Div^n$
defined on $M$. For simplicity, we use the notations:
\be
A\eqdef \U(1)^{2n}~~,~~\Gamma\eqdef \Sp_\frt(2n,\Z)~~,~~G\eqdef \Aff_\frt=A\rtimes \Gamma
\ee
and denote the Abelian Lie algebra of $G$ by $\fg=\aff_\frt\simeq
\R^{2n}$. Let $q: G \rightarrow \Gamma$ be the epimorphism entering
the short exact sequence of groups:
\be
1\rightarrow A\rightarrow G\xrightarrow{q} \Gamma\rightarrow 1~~,
\ee
which splits from the right.

\begin{definition}
The {\em discrete remnant bundle} of $P$ is the
principal $\Gamma$-bundle $\Gamma(P)\eqdef P\times_q\Gamma$.
\end{definition}

\noindent We denote the adjoint representation of $G$ by
$\Ad:G\rightarrow \Aut_\R(\fg)$ and the adjoint action of $G$ (i.e the
action of $G$ on itself by conjugation) by $\Ad_G:G\rightarrow
\Aut(G)$.  The restriction of the latter to the normal subgroup
$A\subset G$ is denoted by $\Ad_G^A:G\rightarrow \Aut(A)$. Since $A$
is Abelian, this factors through $q$ to the {\em
characteristic morphism} $\rho:\Gamma\rightarrow \Aut(A)$:
\be
\Ad_G^A=\rho\circ q~~,
\ee
while the adjoint representation factors through $q$ to the {\em
reduced adjoint representation} $\brho:\Gamma\rightarrow
\Aut_\R(\fg)$:
\ben
\label{brho}
\Ad=\brho\circ q~~.
\een
This representation of $\Gamma$ on $\fg$ preserves the canonical
symplectic form on $\R^{2n}\simeq \fg$.  The exponential map of $G$
gives a surjective morphism of Abelian groups $\exp_G:\fg\rightarrow
A$ whose kernel is a full symplectic lattice $\Lambda$ of $\fg$ which
identifies with $\Lambda_\frt$. This lattice is preserved by the
reduced adjoint representation, which therefore induces a morphism of
groups:
\be
\rho_0:\Gamma\rightarrow \Aut_\Z(\Lambda)~~.
\ee
Accordingly, $\Ad=\brho\circ q$ also preserves $\Lambda$ and hence induces
a morphism of groups:
\be
\Ad_0=\rho_0\circ q:G\rightarrow \Aut_\Z(\Lambda)~~.
\ee
Let $\bDelta=(\Delta, Z)$ be the duality structure defined by $P$,
where $\Delta=(\cS,\omega,\cD)$ (with $\cS=\ad(P)$) and $Z=Z(P)$ are
the duality structure and Siegel system defined by $P$. We have:
\be
\ad(P)=P\times_\Ad \fg=\Gamma(P)\times_\rho \fg~~,~~Z(P)=P\times_{\Ad_0}\Lambda=\Gamma(P)\times_{\rho_0}\Lambda~~,~~\cA(P)=P\times_{\Ad_G^A} A~~,
\ee
where $\cA(P)=\cA(\bDelta(P))=\ad(P)/Z(P)=\cS/Z$ is the bundle of
integral symplectic torus groups defined by $P$. As shown in \cite{wa},
the connection $\cD$ coincides with the flat connection induced on
$\cS$ by the monodromy connection of $\Gamma(P)$ and the symplectic
pairing $\omega$ of $\cS=\ad(P)=P\times_{\rho_0} \fg$ coincides with
that induced by the canonical symplectic pairing of $\R^{2n}\simeq
\fg$.

Let $\Aut(P)$ be the group of those unbased automorphisms of $P$ which
cover {\em orientation-preserving} diffeomorphisms of $M$. We have a
short exact sequence:
\be
1 \to \Aut_b(P) \to \Aut(P) \to \Diff_P(M)\to 1\, ,
\ee
where $\Diff_P(M)\subset \Diff(M)$ is the group formed by those
orientation-preserving diffeomorphisms of $M$ that can be covered by
elements of $\Aut(P)$. Here $\Aut_b(P)$ is the group of based
automorphisms of $P$. For any $u\in \Aut(P)$, denote by $f_u\in
\Diff(M)$ the orientation-preserving diffeomorphism of $M$ covered by
$u$. Every $u\in \Aut(P)$ induces an unbased automorphism $\ad_u\in
\Aut(\cS)$ of the adjoint bundle $\cS=\ad(P)$ defined through:
\be
\ad_u([y,v])\eqdef [u(y),v]\, , \quad \forall\,\, [y,v]\in \cS=\ad(P)=P\times_\ad \fg\, .
\ee
Notice that $\ad_u$ covers $f_u$. 

\begin{prop}
For every $u\in \Aut(P)$, the map $\ad_u: \cS\rightarrow \cS$ is an
unbased automorphism of the integral duality structure $\bDelta$
defined by $P$. Moreover, the map $\ad_P:\Aut(P) \rightarrow
\Aut(\bDelta)$ defined through:
\be
\ad_P(u)=\ad_u~~\forall u\in \Aut(P)
\ee
is a morphism of groups. 
\end{prop}

\begin{proof} 
It is clear from its definition that $\ad_u$ preserves $\omega$ and
$Z$. It also preserves $\cD$, since the latter is induced by the
monodromy connection of $\Gamma(P)$, which is unique. The fact that $\ad_P$
is a morphism of groups is immediate. 
\end{proof}

\noindent Notice that $\ad_P$ restricts to a morphism $\ad_P\colon \Aut_b(P) \to
\Aut_b(\bDelta)$. We set:
\be
g_u\eqdef (f_u)_\ast(g)~~\forall u\in \Aut(P)~~.
\ee
Let $\A\colon \Aut(P)\times\Conn(P)\rightarrow \Conn(P)$ be the affine
left action of $\Aut(P)$ on $\Conn(P)$ defined through:
\be
\A_u(\cA)\eqdef u_{\ast}(\cA)~~\forall \cA\in
\Conn(P)~~\forall u\in \Aut(P)~~,
\ee
where $u_{\ast}\colon \cC^\infty(P,T^{\ast}P\otimes
\fg)\to \cC^\infty(P, T^{\ast}P\otimes \fg)$ denotes the push-forward
of $u$ extended trivially to $\fg$-valued forms defined on $P$.

\begin{lemma}
\label{eq:realclassinvariant}
For every $u\in \Aut(P)$, we have a commutative diagram of affine 
spaces and affine maps:
\be
\scalebox{1.0}{
	\xymatrix{
		\Conn(P)\ar[r]^{\A_u}\ar[d]^{\cV} &\Conn(P) \ar[d]^{\cV}\\
		2\pi j_{0\ast}(c(P)) \ar[r]_{\ad_u}& 2\pi j_{0\ast}(c(P))\\
	}
}
\ee
where $\cV:\Conn(P)\rightarrow \Omega_{\dcl}(M,\cS)$ is the adjoint
curvature map of $P$, $c(P)\in H^2(M,Z)$ is the twisted Chern class of
$P$ and the map $j_{0\ast}:H^2(M,Z)\rightarrow
H^2(M,\cC^\infty_\fl(\cS))=H^2_\cD(M,\cS)$ is induced by the sheaf
inclusion $\cC(Z)\hookrightarrow \cC^\infty_\fl(\cS)$ (see the exact
sequence \eqref{FlatHExpSeq}). Here $ 2\pi j_{0\ast}(c(P))$ is viewed
as an affine subspace of $\Omega_{\dcl}(M,\cS)$ consisting of
$\dd_\cD$-closed $\cS$-valued forms which differ by $\dd_\cD$-exact
$\cS$-valued forms and hence as an affine space modeled on the vector
space $\Omega^2_{\dex}(M,\cS)$.
\end{lemma}
 
\begin{proof}
It is shown in \cite{wa} that the curvature map $\cV$, which is clearly
affine, takes values in $2\pi j_{\ast}(c(P))$. Therefore, it only remains 
to prove that if:
\begin{equation*}
[\cV_{\cA}] = 2\pi\,j_{\ast}(c(P))\, ,
\end{equation*}

\noindent
then:
\begin{equation*}
[\cV_{\mathbb{A}_u(\cA)}] = 2\pi\,j_{\ast}(c(P))\, ,
\end{equation*}

\noindent
or, equivalently, that $[\cV_{\mathbb{A}_u(\cA)}] = [\cV_{\cA}]$ for
every $u\in \Aut_b(P)$. Since $\cA$ is a connection, $\mathbb{A}_u(\cA)$ 
is also a connection on $P$, 
whence there exists an equivariant and horizontal one-form 
$\hat{\tau} \in \Omega^1(P,\mathfrak{a})$ such that:
\begin{equation*}
	\mathbb{A}_u(\cA) = \cA + \hat{\tau} \, .
\end{equation*}

\noindent
Hence, $\dd\mathbb{A}_u(\cA) = \dd\cA + \dd\hat{\tau}$, which descends 
to $M$ as follows:
\begin{equation*}
	\cV_{\mathbb{A}_u(\cA)} = \cV_{\cA} + \dd_{\cD} \tau\, ,
\end{equation*}

\noindent
where $\tau \in \Omega^1(M,\cS_P)$ denotes the one-form with values 
in $\cS_P$ defined by $\tau$. On the other hand, considering
$\cV_{\A_u(\cA)} \in \Omega^2(P,\aff_\frt)$ as a two-form on $P$
taking values in $\aff_\frt$ a direct computation shows that:
\begin{equation*}
\cV_{\A_u(\cA)} = u_{\ast}(\cV_{\cA}) \in \Omega^2(P,\aff_\frt) \, ,
\end{equation*}
which, by the equivariance properties of the latter, immediately
implies:
\begin{equation*}
	\cV_{\A_u(\cA)} = \ad_u\cdot\cV_{\cA} \in \Omega^2(M,\Delta) \, ,
\end{equation*}
where the \emph{dot} action of an automorphism of $\Delta$ on two-
forms taking values in $\Delta$ was defined in subsection 
\ref{subsec:classdual}.
\end{proof}

\begin{prop}
For any $u\in \Aut(P)$, the map $\A_u:\Conn(P)\rightarrow \Conn(P)$
restricts to a bijection:
\be
\mathbb{A}_u\colon \cSol(M,g,P,\cJ) \to \cSol(M,g_u,P,\cJ_{u})~~.
\ee
\end{prop}
 
\begin{proof}
Since $\cV_{\A_u(\cA)}=\ad_u \cdot \cV_\cA$, it suffices to prove the relation:
\ben
\label{starad}
\star_{g_u,\cJ_{u}}\circ \ad_u=\ad_u\circ \star_{g,\cJ}~~.
\een
For $\alpha\in \Omega^k(M)$ and $\xi\in \cC^\infty(M,\cS)$,  we compute:
\be
\star_{g_u,\cJ_{u}} \ad_u\cdot (\alpha\otimes \xi) =
(\ast_{g_u}f_{u \ast}\alpha) \otimes (\cJ_{u}\circ\ad_ u(\xi)\circ
f_u^{-1}) = f_{u\ast}(\ast_{g}\alpha) \otimes \ad_u\circ\cJ(\xi)\circ
f_u^{-1} = \ad_u \cdot (\star_{g_,\cJ} (\alpha\otimes \xi))~~,
\ee
which implies \eqref{starad}
\end{proof}

\begin{definition}
\label{def:dualitygroupscY}
Let $\bP=(P,\cJ)$ be a polarized Siegel bundle defined on $M$.
\begin{itemize}
\item The group $\Aut(P)$ is the {\em unbased gauge group}
of $P$. For any $u\in \Aut(P)$, the map:
\be
\A_{u}\colon \cSol(M,g,P,\cJ) \to \cSol(M,g_u,P,\cJ_{u})
\ee
is the {\em unbased gauge transformation} induced by $u$.
\item The group:
\be
\Aut(g,P)\eqdef \left\{u\in \Aut(P) \, \vert \, f_u\in \Iso(M,g)\right\}
\ee
is the {\em unbased gauge duality group} defined by $P$ and
$g$. For any $u\in \Aut(g,P)$, the map:
\be
\A_u:\cSol(M,g,P,\cJ) \to \cSol(M,g,P,\cJ_{u})
\ee
is the {\em unbased gauge duality transformation} induced by
$u$.
\item The gauge group $\Aut_b(P)$ of $P$ is the {\em gauge 
(electromagnetic) duality group} of the abelian gauge theories with
underlying Siegel bundle $P$. For any $u\in \Aut_b(P)$, the map:
\be
\A_u:\cSol(M,g,P,\cJ) \to \cSol(M,g,P,\cJ_{u})
\ee
is called the {\em gauge duality transformation} induced by $u$.
\end{itemize}
\end{definition}

\noindent
Lemma \ref{eq:realclassinvariant} implies that for any 
$u\in \Aut(P)$ we have a commutative diagram:
\be
\scalebox{1.0}{
\xymatrix{
\cSol(M,g,P,\cJ)\ar[r]^{\A_u}\ar[d]^{\cV} & \cSol(M,g_u,P,\cJ_u) \ar[d]^{\cV}\\
\Sol(M,g,\bDelta,\cJ) \ar[r]_{\ad_u}& \Sol(M,g_u,\bDelta,\cJ_{u})\\
}
}
\ee
and similar diagrams for the other groups in the previous
definition. Hence gauge transformations of $P$ induce integral
pseudo-duality and duality transformations of the abelian gauge theory
defined by $(P,\cJ)$.

\begin{definition}
Let $\bP=(P,\cJ)$ be a polarized Siegel bundle on $M$.
\begin{itemize}
\item The group: 
\ben
\Aut(\bP) \eqdef \left\{ u\in \Aut(P) \,\, \vert \,\ \cJ_{u} = \cJ\right\}
\een
is the {\em unbased unitary gauge group} defined by $\bP$ on
$(M,g)$. For any $u\in \Aut(\bP)$, the map:
\be
\A_u:\cSol(M,g,\bP) \to \cSol(M,g_u,\bP)
\ee
is the {\em unbased unitary gauge transformation} induced by
$u$.
\item The group: 
\ben
\label{eq:relativeunitaryglobalcY}
\Aut(g,\bP) \eqdef \left\{ u\in \Aut(P) \,\, \vert \,\, g_u=g~~\mathrm{and}~~\cJ_{u} = \cJ\right\}
\een
is the {\em unbased unitary gauge duality group} defined by
$\bP$ on $(M,g)$. For any $u\in \Aut(g,\bP)$, the map:
\be
\A_u:\cSol(M,g,\bP) \to \cSol(M,g,\bP)
\ee
is the {\em unbased unitary gauge duality transformation} induced by
$u$.
\item The group:
\be
\Aut_b(\bP) \eqdef \left\{  u \in \Aut_b(P) \,\, \vert \,\,   \cJ_{u} = \cJ  \right\}\, .
\ee
is the {\em unitary gauge group} defined by $\bP$ on $M$. For any $u\in
\Aut_b(\bP)$, the map:
\be
\A_u:\cSol(M,g,\bP) \to \cSol(M,g,\bP)~~. 
\ee
is the  {\em unitary gauge transformation} induced by $u$.
\end{itemize}
\end{definition}

\noindent Let $\bXi\eqdef (\bDelta,\cJ)$. For any $u\in \Aut(\bP)$, we have a
commutative diagram:
\be
\scalebox{1.0}{
\xymatrix{
\cSol(M,g,\bP)\ar[r]^{\A_u}\ar[d]^{\cV} & \cSol(M,g_u,\bP) \ar[d]^{\cV}\\
\Sol(M,g,\bXi) \ar[r]_{\ad_u}& \Sol(M,g_u,\bXi)\\
}
}
\ee
and similar diagrams for the other groups in the previous
definition. We have a short exact sequence:
\ben
\label{eq:dualitygroupsequencecY}
1 \to \Aut_b(P) \rightarrow \Aut(g,P) \rightarrow \Iso_P(M,g)\to 1\, ,
\een
where $ \Iso_P(M,g)\subset \Iso(M,g)$ is the group formed by those
orientation-preserving isometries that can be covered by elements of
$\Aut(g,P)$. Similarly, we have an exact sequence:
\be
1 \to \Aut_b(\bP) \rightarrow \Aut(g,\bP)\rightarrow \Iso_\bP(M,g)\to 1\, ,
\ee
where $\Iso_\bP(M,g)$ is the group formed by those
orientation-preserving isometries of $M$ which are covered by elements
of $\Aut(g, \bP)$.

\begin{definition}
The \emph{standard subgroup} of the unbased gauge group of $P$
is defined through:
\be
\mC(P) \eqdef \ker (\ad_P)\subset \Aut(P)\, .
\ee
\end{definition}

\noindent When $\dim M>0$ and $\dim A>0$, the group $\mC(P)$ is 
infinite-dimensional. The classical duality 
group of a duality structure was shown to be a finite dimensional 
Lie group in Section \ref{sec:classical}. This is no longer true of 
the gauge groups introduced above. Instead, they are infinite-dimensional
extensions of the integral duality groups introduced in Section
\ref{sec:DQsymplecticabelian}.

\begin{prop}
\label{prop:exactseqAut}
The gauge group of $P$ fits into the short exact sequence of groups:
\ben
\label{CPseq}
1 \to \mC(P) \hookrightarrow \Aut_b(P)\xrightarrow{\ad_P} \Aut_b(\bDelta)\to 1~~.
\een
\end{prop}

\begin{remark}
There exist similar short exact sequences for the remaining groups introduced 
in Definition \ref{def:dualitygroupscY}.
\end{remark}

\begin{proof}
It suffices to prove that $\ad_P(\Aut_b(P))=\Aut_b(\bDelta)$.  Recall
that $\cL = P\times_{\Ad_0} \Lambda=\Gamma(P)\times_{\rho_0}\Lambda$
and $\cS=P\times_{\Ad}\fg=\Gamma(P)\times_{\brho}\fg$, where
$\Lambda\equiv\Lambda_\frt$ and $\fg\equiv\R^{2n}$.  Also recall that
the gauge group $\Aut_b(P)$ is naturally isomorphic with the group
$\cC^{\infty} (P,G)^G$ of $G$-equivariant maps from $P$ to $G$, where
$G=A\times \Gamma=\U(1)^{2n}\rtimes \Sp_\frt(2n,\mathbb{Z})$ acts on
itself through conjugation. This isomorphism takes $u\in \Aut_b(P)$ to
the equivariant map $f\in \cC^{\infty} (P,G)^G$ which satisfies:
\ben
\label{fdef}
u(p)=pf(p)~~\forall p\in P~~\forall g\in G~~.
\een
Since the action of $G$ on $P$ is free and we have
$\Gamma\equiv\Aut(\R^{2n},\omega_{2n},\Lambda_\frt)\simeq
\Aut(\cS_m,\omega_m,\cL_m)$ for all $m\in M$ while the reduced adjoint
representation $\brho$ is faithful, every automorphism $\varphi\in
\Aut(\bDelta)$ determines a map $\bar{\varphi}:P \to \Gamma$ which
satisfies:
\ben
\label{barphidef}
\varphi([p,v]) = [p , \brho(\bar{\varphi}(p))(v)]~~\forall p\in P~~\forall v\in \fg
\een
as well as:
\be
\bar{\varphi}(pg)=q(g)^{-1}\bar{\varphi}(p)q(g)~~\forall p\in P~~\forall g\in G~~.
\ee
The last relation follows from \eqref{barphidef} and from the
condition $\varphi([pg,v])=\varphi([p,\brho(q(g))(v)])$ of invariance
under change of representative of the equivalence class, where we used
\eqref{brho}. Let $u\in \Aut_b(P)$ be the based automorphism of $P$
which corresponds to the $G$-equivariant map $f:P\rightarrow G$
defined through:
\be
f(p)\eqdef (0_A,\bar{\varphi}(p))\in G=A\rtimes \Gamma~~\forall p\in P~~.
\ee
For any $p\in P$ and $v\in \fg$, we have: 
\be
\ad_P(u)([p,v])=[u(p),v]=[p f(p),v]=[p, \Ad(f(p))(v)]=[p, \brho(\bar{\varphi}(p))(p)]=\varphi([p,v])~~,
\ee
where we used \eqref{fdef} and \eqref{barphidef}. This shows that
$\ad_P(u)=\varphi$. Since $\varphi\in \Aut_b(\bDelta)$ is arbitrary,
we conclude that $\ad_P(\Aut_b(P))=\Aut_b(\bDelta)$.
\end{proof}

\noindent 
The previous proposition clarifies the geometric origin of
electromagnetic duality as a `discrete
remnant' of gauge symmetry, a notion which is discussed in more detail in \cite{wa}. In
particular, $\Aut_b(P)$ is an extension of $\Aut_b(\bDelta)$ by the
continuous group $C(P)$. Intuitively, elements of the latter
correspond to the gauge transformations of a principal torus bundle.


\subsection{Duality groups for Siegel bundles with trivial monodromy}


Let $M$ be a connected and oriented four-manifold.

\begin{lemma}
\label{lemma:trivunbased}
Let $P$ be a trivial principal $G$-bundle over $M$. Then any
trivialization of $P$ induces an isomorphism of groups:
\be
\Aut(P)\simeq \cC^\infty(M,G)\rtimes_\alpha \Diff(M)~~,
\ee
where $\alpha:\Diff(M)\rightarrow \Aut(\cC^\infty(M,G))$ is the
morphisms of groups defined through:
\be
\alpha(\varphi)(f)\eqdef f\circ \varphi^{-1}\, , \quad \forall \,\, \varphi\in
\Diff(M)\, , \quad \forall\,\, f\in \cC^\infty(M,G)~~.
\ee
In particular, we have a short exact sequence of groups:
\be
1\rightarrow \cC^\infty(M,G)\rightarrow \Aut(P)\rightarrow \Diff(M)\rightarrow 1
\ee
which is split from the right.
\end{lemma}

\begin{proof}
Let $\tau:P\stackrel{\sim}{\rightarrow} M\times G$ be a trivialization of $P$. Then
the map $\Ad(\tau):\Aut(P)\rightarrow \Aut(M\times G)$ defined through:
\be
\Ad(\tau)(f)\eqdef \tau\circ f\circ \tau^{-1}\, , \quad \forall\,\, f\in \Aut(P)\, ,
\ee
is an isomorphism of groups. Let $f\in \Aut(P)$ be an unbased
automorphism of $P$ which covers the diffeomorphism $\varphi\in
\Diff(M)$.  Then $\Ad(\tau)(f)$ is an unbased automorphism of $M\times
G$ which covers $\varphi$ and hence we have:
\be
\Ad(\tau)(f)(m,g)=(\varphi(m), {\hat f}(m,g))\, , \quad  \forall\,\, (m,g)\in M\times G~~,
\ee
where ${\hat f}:M\times G\rightarrow G$ is a smooth map which satisfies:
\be
{\hat f}(m,g_1g_1)={\hat f}(m,g_1)g_2\, , \quad \forall\,\, m\in M\, , \quad \forall\,\, g_1,g_2\in G~~.
\ee
The last relation is equivalent with the condition that ${\hat f}$ has the form:
\be
{\hat f}(m,g)={\tilde f}(m) g\, , \quad \forall \,\, (m,g)\in M\times G~~,
\ee
where ${\tilde f}:M\rightarrow G$ is a smooth function which can be
recovered from ${\hat f}$ through the relation:
\be
{\tilde f}(m)={\hat f}(m,1)\, , \quad \forall\,\, m\in M~~.
\ee
Setting $h\eqdef {\tilde f}\circ \varphi^{-1}\in \cC^\infty(M,G)$, we have:
\ben
\label{fcomponents}
\Ad(\tau)(f)(m,g)=(\varphi(m), h(\varphi(m))g)\, , \quad \forall\,\, (m,g)\in M\times G\, ,
\een
and the correspondence $f\rightarrow (h,\varphi)$ gives a bijection
between $\Aut(P)\simeq \Aut(M\times G)$ and the set
$\cC^\infty(M,G)\times \Diff(M)$. If $f_1,f_2\in \Aut(P)$ correspond
through this map to the pairs $(h_1,\varphi_1), (h_2,\varphi_2)\in
\cC^\infty(M,G)\times \Diff(M)$, then direct computation using
\eqref{fcomponents} gives:
\be
\Ad(\tau)(f_1\circ f_2)(m,g)=((\varphi_1\circ \varphi_2)(m), h_1(m) (h_2\circ \varphi_1^{-1})(m) g)~~,
\ee
showing that $f_1\circ f_2$ corresponds to the pair $(h_1\cdot
\alpha(\varphi_1)(h_2),\varphi_1\circ \varphi_2)$.
\end{proof}

\noindent Let $P$ be a Siegel bundle or rank $n$ and type $\frt\in
\Div^n$ on $(M,g)$ and let $\bDelta=(\Delta,Z)$ be the integral
duality structure defined by $P$. Suppose that $P$ is topologically
trivial and that $Z=Z(P)$ has trivial monodromy, so that $\Delta$ is
holonomy trivial. Choosing a trivialization of $P$ gives:
\be
P \equiv M\times \Aff_\frt\, , \qquad \bDelta \equiv (M\times \R^{2n},\omega_{2n},\dd,M\times \Lambda_\frt)
\ee
and:
\be
\Aut_b(P) \equiv \cC^\infty(M,\Aff_\frt) \, , \qquad  \Aut(P) \equiv \cC^\infty(M,\Aff_\frt)\rtimes_{\alpha} \Diff(M)~~,
\ee
where the last identification follows from Lemma
\ref{lemma:trivunbased}. Since $\Aff_\frt = \U(1)^{2n}\rtimes \Sp_{\frt}(2n,\Z)$ 
and $\Sp_{\frt}(2n,\Z)$ is discrete, we have:
\be
\cC^\infty(M,\Aff_\frt)=\cC^\infty(M,\U(1)^{2n})\rtimes \Sp_\frt(2n,\Z)~~.
\ee
In particular, maps $h\in \cC^\infty(M,\Aff_\frt)$ can be identified
with pairs $(f,\gamma)$, where $f\in \cC^\infty(M,\U(1)^{2n})$ and $\gamma\in
\Sp_\frt(2n,\Z)$. The unbased gauge duality group is given by:
\be
\Aut(g,P) \equiv \cC^\infty(M,\Aff_\frt)\rtimes_{\alpha} \Iso(M,g)~~.
\ee
The integral pseudo-duality, relative duality and duality groups of
$\bDelta$ are in turn given by:
\beqa
& \Aut(\bDelta)\equiv \Sp_\frt(2n,\Z)\times \Diff(M)\, ,\nn\\
& \Aut(g,\bDelta)\equiv \Sp_\frt(2n,\Z)\times \Iso(M,g)\, ,\nn\\
&\Aut_b(\bDelta)\equiv \Sp_\frt(2n,\Z)\, ,
\eeqa
and we have short exact sequences:
\beqa
& 1\to \cC^\infty(M,\U(1)^{2n}) \to \Aut(P)\to \Aut(\bDelta)\to 1\, , \\
& 1\to \cC^\infty(M,\U(1)^{2n}) \to \Aut(g,P)\to \Aut(g,\bDelta)\to 1\, ,\\
& 1\to \cC^\infty(M,\U(1)^{2n}) \to \Aut_b(P)\to \Aut_b(\bDelta)\to 1\, .
\eeqa
Let us fix a taming of $\Delta$, which we view as a map
$\cJ\in\cC^\infty(M,\Sp(2n,\R))$. Then the unbased unitary gauge group
of the tamed Siegel bundle $\bP=(P,\cJ)$ is:
\be
\Aut(g,\bP) \equiv \left\{(f,\gamma,\varphi)\in
\left[\cC^\infty(M,\U(1)^{2n})\times \Sp_\frt(2n,\Z)\right]\rtimes_\alpha \Iso(M,g)
\, \vert\,\, \gamma \cJ\gamma^{-1}=\cJ\circ \varphi \right\}~~,
\ee
while its unitary gauge group is:
\be
\Aut_b(\bP) \equiv \left\{ (f,\gamma)\in \cC^\infty(M,\U(1)^{2n})\times
\Sp_\frt(2n,\Z) \, \vert \, \gamma \cJ\gamma^{-1}=\cJ \right\}~~.
\ee
The integral unbased unitary duality group of the integral electromagnetic structure
$\bXi=(\bDelta,\cJ)$ is:
\be
\Aut(g,\bXi) \equiv \left\{(\gamma,\varphi) \in \Sp_{\frt}(2n,\Z)\times
\Iso(M,g) \, \, \vert \,\, \gamma \cJ \gamma^{-1} = \cJ\circ
\varphi\right\}\, ,
\ee
while the integral unitary duality group of $\bXi$ is:
\be
\Aut_b(\bXi) \equiv \left\{\gamma \in \Sp_{\frt}(2n,\Z) \, \, \vert \,\, \gamma \cJ \gamma^{-1} = \cJ\right\}\subset
\Sp_\frt(2n,\Z)\, .
\ee
We have short exact sequences:
\beqa
& 1\to \cC^\infty(M,\U(1)^{2n}) \to \Aut(g,\bP)\to \Aut(g,\bXi)\to 1\, \\
& 1\to \cC^\infty(M,\U(1)^{2n}) \to \Aut_b(\bP)\to \Aut_b(\bXi)\to 1~~.
\eeqa


\section{Time-like dimensional reduction and polarized Bogomolny equations}
\label{sec:TimelikeReduction}


This section investigates the time-like dimensional reduction of the
equations of motion of abelian gauge theory on an oriented
static space-time $(M,g)$ of the form:
\be
(M,g) = (\R \times \Sigma, - \dd t^2 \oplus h)\, ,
\ee
where $t$ is the global coordinate on $\R$ and $(\Sigma,h)$ is an oriented 
Riemannian three-manifold. We show that the reduction produces an equation of
Bogomolny type, similar to the dimensional reduction of the ordinary
self-duality equation on a Riemannian four-manifold. Unlike that
well-known case, here we reduce the \emph{polarized} 
self-duality condition on a {\em Lorentzian} four-manifold.


\subsection{Preparations}


Consider the time-like exact one-form $\theta=\dd t\in \Omega^1(M)$.
Let $\nu_h$ be the volume form of $(\Sigma,h)$ and orient $(M,g)$ such
that its volume form is given by:
\be
\nu_g=\theta \wedge \nu_h~~.
\ee
Let $\ast_g$ and $\ast_h$ be the Hodge operators of $(M,g)$ and
$(\Sigma,h)$ and $\langle~,~\rangle_g$, $\langle~,~\rangle_h$ be
the non-degenerate bilinear pairings induced by $g$ and $h$ on
$\Omega^\ast(M)$ and $\Omega^\ast(\Sigma)$. Let $p:M\rightarrow
\Sigma$ be the projection of $M=\R\times \Sigma$ on the second factor
and consider the distribution $\cD=p^\ast(T\Sigma)\subset TM$, endowed
with the fiberwise Euclidean pairing given by the bundle pullback
$h^p$ of $h$. This distribution is integrable with leaves gives by the
spacelike hypersurfaces:
\be
M_t\eqdef \{t\}\times \Sigma\, , \quad \forall\,\, t\in \R~~,
\ee
on which $h^p$ restricts to the metric induced by $g$. Using $h^p$, we
extend $\langle~,~\rangle_h$ and $\ast_h$ in the obvious manner to the
space $\cC^\infty(M, \wedge^\ast\cD^\ast)\subset \Omega^\ast(M)$. We have:
\be
\cC^\infty(M,\wedge^\ast \cD^\ast)=\{\omega\in \Omega^\ast(M)~\vert~\iota_{\partial_t}\omega=0\}~~.
\ee
Since $g$ has signature $(3,1)$ while $h$ has signature
$(3,0)$, we have:
\be
\ast_g\circ \ast_g=-\pi\, , \quad \ast_h\circ \ast_h=\id_{\cC^\infty(M,\wedge\cD^\ast)}~~,
\ee
where $\pi\eqdef \oplus_{k=0}^4{(-1)^k\id_{\Omega^k(M)}}$ is the
signature automorphism of the exterior algebra
$(\Omega^\ast(M),\wedge)$ (see \cite{ga1}). Moreover, we have $\langle
\theta,\theta\rangle_g=-1$ and:
\be
\langle \nu_g,\nu_g\rangle_g=-1\, , \quad \langle \nu_h,\nu_h\rangle_h=+1~~.
\ee
Any polyform $\omega\in \Omega^\ast(M,g)$ has a unique decomposition:
\ben
\label{omegadec}
\omega=\omega_\parallel+\omega_\perp~~,
\een
such that $\omega_\parallel, \omega_\perp\in \Omega^\ast(M)$ satisfy \cite{ga1}:
\ben
\label{omegacomp}
\theta\wedge\omega_\parallel=0\, , \quad \iota_{\partial_t}\omega_\perp=0~~.
\een
The second of these conditions amounts to the requirement
that $\omega_\perp\in \cC^\infty(M,\wedge^\ast \cD^\ast)$, while the
first is solved by:
\be
\omega_\parallel=\theta\wedge \omega_\top~~\mathrm{where} ~~
\omega_\top\eqdef -\iota_{\partial_t}\omega \in \cC^\infty(M,\wedge^\ast\cD^\ast)~~.
\ee
As shown in loc. cit., the map $\omega\rightarrow
(\omega_\top,\omega_\perp)$ gives a linear isomorphism between
$\Omega^\ast(M)$ and $\cC^\infty(M,\wedge^\ast \cD^\ast)^{\oplus
2}$. For any $k=0,\ldots, 4$ and any $\omega,\eta\in \Omega^k(M)$, we
have:
\be
\langle \omega,\eta\rangle_g=-\langle
\omega_\top,\eta_\top\rangle_h+\langle
\omega_\perp,\eta_\perp\rangle_h\, , \quad \langle
\omega_\parallel,\eta_\parallel\rangle_g=-\langle
\omega_\top,\eta_\top\rangle_h~~.
\ee
Hence the isomorphism above identifies the quadratic space
$(\Omega^\ast(M),\langle~,~\rangle_h)$ with the direct sum of
quadratic spaces $(\cC^\infty(M,\wedge^\ast \cD^\ast),
-\langle~,~\rangle_h)\oplus (\cC^\infty(M,\wedge^\ast \cD^\ast),
\langle~,~\rangle_h)$. Notice that $\nu_h=\iota_{\partial_t}\nu_g=-(\nu_g)_\top$.
An easy computation gives:

\begin{lemma}
\label{lemma:astdec}
For any polyform $\omega\in \Omega^\ast(M)$, we have:
\ben
\label{astpt}
(\ast_g\omega)_\top=\ast_h \pi(\omega_\perp)\, , \quad (\ast_g\omega)_\perp=-\ast_h\omega_\top
\een
and hence:
\ben
\label{astdec}
\ast_g\omega=-\ast_h\omega_\top+\theta\wedge \ast_h \pi(\omega_\perp)~~.
\een
\end{lemma}

\noindent Given any vector bundle $V$ defined on $M$, the Hodge
operators of $g$ and $h$ extend trivially to operators
$\ast_g:\Omega^\ast(M,V)\rightarrow \Omega^\ast(M,V)$ and
$\ast_h:\cC^\infty(M,\wedge^\ast \cD^\ast\otimes V)\rightarrow
\cC^\infty(M,\wedge^\ast \cD^\ast\otimes V)$. The decomposition
\eqref{omegadec} holds for any $\omega\in \Omega^\ast(M,V)$, with
components $\omega_\parallel,\omega_\perp\in \Omega^\ast(M,V)$
satisfying \eqref{omegacomp}. We have $\omega_\perp\in
\cC^\infty(M,\wedge^\ast \cD^\ast\otimes V)$ and
$\omega_\parallel=\theta\wedge \omega_\top$ with $\omega_\top\eqdef
-\iota_{\partial_t}\omega \in \cC^\infty(M,\wedge^\ast\cD^\ast\otimes
V)$. Finally, Lemma \ref{lemma:astdec} holds for any $\omega\in
\Omega^\ast(M,V)$.


\subsection{Timelike dimensional reduction of abelian gauge theory}


Let $P$ be a Siegel bundle of type $\frt\in \Div^n$ defined on
$\Sigma$, whose projection we denote by $\pi:P\rightarrow \Sigma$. Let
$\hP\eqdef p^\ast(P)$ be the $p$-pullback of $P$ to $M$, whose
projection we denote by $\hpi$. The map $\varphi:\hP\rightarrow
\R\times P$ defined through:
\be
\varphi(t,\sigma,y)\eqdef (t,y)\, , \quad \forall\,\, (t,\sigma)\in \R\times \Sigma\, , \quad \forall\,\, y\in P_\sigma=\pi^{-1}(\sigma)
\ee
allows us to identify $\hP$ with the principal $\Aff_\frt$-bundle
with total space given by $\R\times P$, base $M=\R\times \Sigma$ and
projection given by $\R\times P \ni (t,y)\rightarrow (t,\pi(y))\in
M$. We make this identification in what follows. Accordingly, we
have:
\be
\hpi(t,y)=(t,\pi(y))\, , \quad \forall\,\, (t,y)\in \hP\equiv \R\times P\, .
\ee
Let $\tau:\hP\rightarrow P$ be the unbased morphism of
principal $\Aff_\frt$-bundles given by projection on the second
factor:
\be
\tau(t,y)\eqdef y\, , \quad \forall\, \, (t,y)\in \hP\, ,
\ee
which covers the map $p:M\rightarrow \Sigma$:
\be
\pi\circ \tau=p\circ \hpi~~.
\ee
Consider the action $\rho:\R\rightarrow \Aut(\hP)$ of $(\R,+)$
through unbased automorphisms of $\hP$ given by:
\be
(t,y)\rightarrow \rho(a)(t,y)\eqdef (t+a,y)\, , \quad \forall\,\, (t,y)\in \hP\equiv \R\times P\, .
\ee
This covers the action of $\R$ on $M=\R\times \Sigma$ given by 
time translations:
\be
\hpi(\rho(a)(t,y))=(t+a,\pi(y))\, , \quad \forall \,\, (t,y)\in \hP\equiv \R\times P~~.
\ee

\begin{definition}
A principal connection on $\hcA\in \Conn(\hP)$ is called {\em
time-invariant} if it satisfies:
\be
\rho(a)^\ast(\hcA)=\hcA\, , \quad \forall\,\, a\in \R~~.
\ee
\end{definition}

\noindent Notice that time-invariant principal connections defined on
$\hP$ form an affine subspace $\Conn^s(\hP)$ of $\Conn(\hP)$. The
timelike one-form $\theta \in \Omega^1(M)$ pulls back through $\hpi$
to an exact one-form defined on $\hP$ which we denote by
$\htheta\eqdef p^\ast(\theta)\in \Omega^1(\hP)$.

Let $\Delta=(\cS,\omega,\cD)$ be the duality structure defined by
$P$ on $\Sigma$. Then it is easy to see that the duality structure
$\hDelta=(\hcS,\homega,\hcD)$ defined by $\hP$ on $M$ is given by:
\be
\hcS=p^\ast(\cS)\, ,\quad \homega=p^{\ast}(\omega)\, , \quad \hcD=p^\ast(\cD)\, .
\ee

\begin{lemma}
\label{lemma:hcAdec}
A connection $\hcA\in \Conn(\hP)$ is time-invariant if and only if it
can be written as:
\ben
\label{hcAdec}
\hcA=-(\Psi^\pi\circ \tau) \htheta+ \tau^\ast(\cA)
\een
for some $\Psi\in \cC^\infty(\Sigma,\cS)$ and some $\cA\in \Conn(P)$.
In this case, $\Psi$ and $\cA$ are determined uniquely by $\hcA$ and
any pair $(\Psi,\cA)\in \cC^\infty(\Sigma,\cS)\times\Conn(P)$ determines
a time-invariant connection on $\hP$ though this relation.  Moreover,
the curvature of $\hcA$ is given by:
\ben
\label{hcAcurv}
\cV_\hcA=\theta\wedge p^\ast(\dd_{\cD} \Psi)+p^\ast(\cV_\cA)\in \Omega^2(\Sigma,\hat{\cS})
\een
and we have:
\ben
\label{cVAflat}
\dd_\cD\cV_\cA=0~~.
\een
\end{lemma}

\begin{remark}
Relation \eqref{hcAcurv} gives the decomposition \eqref{omegadec} of
$\cV_\hcA$ since $\iota_{\partial_t}
p^\ast(\cV_A)=\iota_{\partial_t}p^\ast(\dd_\cD\Psi)=0$.  Thus:
\ben
\label{cVtp}
(\cV_\hcA)_\top=p^\ast(\dd_\cD\Psi)\, , \quad (\cV_\hcA)_\perp=p^\ast(\cV_\cA)~~.
\een
\end{remark}

\begin{proof}
Any principal connection $\cA\in \Conn(\hP)\subset \Omega^1(\hP,\fa)$
decomposes uniquely as:
\be
\hcA=-\Phi \htheta+\cA_\perp~~,
\ee
where $\Phi\in \Omega^0_\Ad(\hP,\fa)$ and $\cA_\perp\in
\Conn(\hP,\fa)$ satisfies $\cA_\perp(\partial_t)=0$.  It is clear
that $\hcA$ is time-invariant if and only if $\Phi=\Psi'\circ
\tau$ for some $\Psi'\in \Omega^0_\Ad(P,\fa)$ and
$\hcA_\perp=\tau^\ast(\cA)$ for some $\cA\in \Conn(P)$. Since
$\Omega^0_\Ad(P,\fa)\simeq \cC^\infty(\Sigma,\cS)$, we have
$\Psi'=\Psi^\pi$ for some $\Psi\in \cC^\infty(\Sigma,\cS)$.
Since $\dd\htheta=0$, the principal curvature of $\hcA$ reads:
\be
\Omega_\hcA=\dd \hcA=\dd\Phi\wedge \htheta+\tau^\ast(\Omega_\cA)~~,
\ee
which is equivalent with \eqref{hcAcurv}. Relation \eqref{cVAflat}
follows from the results of \cite{wa}. The remaining
statements are immediate.
\end{proof}

\begin{definition}
The {\em Bogomolny pair} of a time-invariant connection $\hcA\in
\Conn^s(\hP)$ is the pair $(\Psi,\cA)\in
\cC^\infty(\Sigma,\cS)\times\Conn(P)$ defined in Lemma
\ref{lemma:hcAdec}. The section $\Psi\in \cC^\infty(\Sigma,\cS)$ is
called the {\em Higgs field} of the pair.
\end{definition}

\noindent Let $\cJ$ be a taming of $\Delta$. Then the $p$-pullback
$\hcJ$ of $\cJ$ defines a time-invariant taming of $\hcS$, thus
$(\hP,\hcJ)$ is a polarized Siegel bundle. Let
$\star_{h,\cJ}=\ast_h\otimes \cJ$ be the polarized Hodge operator
defined by $h$ and $\cJ$. Since $\cJ^2=-\id_\cS$ while $\ast_h$
squares to the identity on $\Omega^\ast(\Sigma)$, we have:
\be
\star_{h,\cJ}\circ \star_{h,\cJ}=-\id_{\Omega^\ast(\Sigma,\cS)}~~.
\ee

\begin{prop}
A time-invariant connection $\hcA\in \Conn^s(\hP)$ is polarized
self-dual with respect to $\hcJ$ if and only if its Bogomolny pair
$(\Psi,\cA)$ satisfies the {\em polarized Bogomolny equation} with
respect to $\cJ$:
\ben
\label{eq:Bogomolny}
\star_{h,\cJ}\cV_\cA=\dd_\cD\Psi \Longleftrightarrow \star_{h,\cJ}\dd_\cD\Psi=-\cV_\cA~~.
\een
A polarized Bogomolny pair $(\Psi,\cA)$ which satisfies this equation
is called a {\em polarized abelian dyon} relative to $\cJ$.
\end{prop}

\begin{proof}
Relations \eqref{cVtp} and \eqref{astdec} give:
\be
\ast_g\cV_\hcA=-p^\ast(\ast_h\dd_\cD\Psi)+\theta\wedge p^\ast(\ast_h\cV_\cA)~~,
\ee
which implies:
\be
\star_{g,\hcJ}\cV_\hcA=-p^\ast(\star_{h,\cJ}\dd_\cD\Psi)+\theta\wedge p^\ast(\star_{h,\cJ}\cV_\cA)~~.
\ee
Comparing this with \eqref{hcAcurv} shows that the polarized
self-duality condition for $\cV_\hcA$ amounts to \eqref{eq:Bogomolny},
where we used uniqueness of the decomposition \eqref{omegadec}.
\end{proof}

\noindent Let:
\be
\Dyons(\Sigma,h,\bP)\eqdef \{(\Psi,\cA)\in \cC^\infty(\Sigma,\cS)\times \Conn(P)~\vert~\star_{h,\cJ}\cV_\cA=\dd_\cD\Psi\}
\ee
be the set of all polarized abelian dyons relative to $\cJ$.

\begin{remark} Equation \eqref{eq:Bogomolny} is reminiscent of the
usual Bogomolny equations obtained by dimensional reduction of the
self-duality equations for a connection on a principal bundle over a
four-dimensional Riemannian manifold. However, it differs from the
latter in two crucial respects:

\begin{itemize}
\item The usual Bogomolny equations arise by dimensional reduction of
the self-duality equations (which are first order equations for a
connection) in four \emph{Euclidean} dimensions. By contrast, equation
\eqref{eq:Bogomolny} is the reduction along a timelike direction of
the complete second order equations of motion defining abelian gauge
theory in four \emph{Lorentzian} dimensions, once these equations have
been re-written as first-order equations by doubling the number of
variables through the inclusion of both electromagnetic and
magnetoelectric gauge potentials. In particular, our reduction yields
a system of first-order differential equations, despite originating in
a theory that was initially defined by local second-order PDEs (see
Appendix \ref{app:local}).
\item Equation \eqref{eq:Bogomolny} is modified by the action of the
taming $\cJ$, which is absent in the usual Bogomolny equations.
\end{itemize}
\end{remark}


\subsection{Gauge transformations of polarized abelian dyons} 


As explained in Subsection \ref{sec:globaldualitygroups} (see
\cite{wa} for more detail), the gauge group $\Aut_b(P)$ of the Siegel
bundle $P$ over $\Sigma$ has an action:
\be
\ad_P:\Aut_b(P)\rightarrow \Aut_b(\cS)
\ee
through based automorphisms of the vector bundle $\cS=\ad(P)$. This
action agrees through the adjoint curvature map with the pushforward
action:
\be
\mathbb{A}\colon\Aut_b(P)\rightarrow \Aff(\Conn(P))
\ee
of the gauge group on the space of principal connections defined on $P$. For any
principal connection $\cA\in\Conn(P)$ and any $ u\in \Aut_b(P)$, we have:
\be
\cV_{\mathbb{A}_u(\cA)}=\ad_P( u)(\cV_\cA)~~.
\ee
Similar statements hold for the gauge group of the Siegel bundle
$\hP$ defined on $M=\R\times \Sigma$.

\begin{definition}
A gauge transformation $\hpsi\in \Aut_b(\hP)$ of $\hP$ is called
{\em time-invariant} if:
\be
 u\circ \rho(a)=\rho(a)\circ  u\, , \quad \forall \,\, a\in \R~~.
\ee
\end{definition}

\noindent Notice that time-invariant gauge transformations of $\hP$
form a subgroup of $\Aut_b(\hP)$, which we denote by
$\Aut_b^s(\hP)$. Such transformations stabilize the affine
subspace $\Conn^s(P)$ of time-invariant principal connections
defined on $\hP$. Since $\hP=p^\ast(P)$, we have
$\Ad_G(\hP)=p^\ast(\Ad_{P})$.  It is easy to see that $\hpsi\in
\Aut_b(\hP)$ is time-invariant if and only if the corresponding section
$\sigma_{\hpsi}\in \cC^\infty(M,\Ad_G(\hP))$ is the bundle pull-back by
$p$ of a section $\sigma\in \cC^\infty(\Sigma,\Ad_G(P))$. The latter
corresponds to a gauge transformation of $P$ which we denote by
$ u\in \Aut_b(P)$. We have:
\ben
\label{psihpsi}
\sigma_{\hpsi}=(\sigma_ u)^p\, ,
\een
(where the subscript $p$ denotes bundle pullback by $p$) as well as:
\ben
\label{tauhpsi}
\tau\circ \hpsi= u\circ \tau~~.
\een
Conversely, any gauge transformation $ u$ of $P$ determines a gauge
transformation $\hpsi$ of $\hP$ by relation \eqref{psihpsi} and
$\hpsi$ satisfies \eqref{tauhpsi}. The correspondence $ u\rightarrow
\hpsi$ gives an isomorphism of groups between $\Aut_b(P)$ and
$\Aut_b^s(\hP)$.

The following proposition shows that the map which takes a time-invariant
principal connection defined on $\hP$ to its Bogomolny pair
intertwines the action of $\Aut_b^s(\hP)\simeq \Aut(P)$ on $\Conn^s(P)$ with the
action of $\Aut_b(P)$ on the set $\cC^\infty(\Sigma,\cS)\times \Conn(P)$ given by:
\be
\mu\eqdef \ad_P\times \A:\Aut_b(P)\rightarrow \Aut_\R(\cC^\infty(\Sigma,\cS))\times \Aff(\Conn(P))~~.
\ee

\begin{prop}
\label{prop:GaugeBogomolny}
Let $\hcA$ be a time-invariant principal connection on $\hP$ and
$(\Psi,\cA)\in \cC^\infty(\Sigma,\cS)\times \Conn(P)$ be the Bogomolny
pair defined by $\hcA$. For any $ u\in \Aut_b(P)$, the Bogomolny
pair $(\Psi_ u,\cA_ u)$ of the time-invariant connection
$\mathbb{A}_{\hpsi}(\hcA)$ obtained from $\hcA$ by applying the 
time-invariant gauge transformation $\hpsi$ is given by:
\be
\Psi_ u=\ad_P( u)(\Psi)\, , \quad\cA_ u=\mathbb{A}_u(\cA)\, .
\ee
In particular, we have:
\be
\cV_{\cA_ u}=\ad_P( u)(\cV_\cA)~~.
\ee
\end{prop}

\begin{proof}
We have $\mathbb{A}_{\hpsi}(\hcA)=( \hpsi^{-1})^\ast(\hcA)$. Using Lemma
\ref{lemma:hcAdec}, this gives:
\be
\mathbb{A}_{\hpsi}(\hcA)=-(\Psi^\pi\circ  u^{-1}\circ \tau) \htheta+
\tau^\ast(( u^{-1})^\ast(\cA))=-(\Psi_ u^\pi\circ \tau) \htheta+
\tau^\ast(\mathbb{A}_u(\cA))~~,
\ee
where we used relation \eqref{tauhpsi} and noticed that
$(\hpsi^{-1})^\ast(\htheta)=\theta$ since $\htheta=\hpi^\ast(\theta)$
and $\hpi\circ\hpsi^{-1}=\hpsi$ because $\hpsi^{-1}$ is a based
automorphism of $\hP$.
\end{proof}

\noindent Notice that the discrete remnant (see \cite{wa})
of any gauge transformation of $\hP$ is time-invariant since $\Sigma$ 
(and hence $M$) is connected and therefore any discrete gauge transformation 
of $\hP$ is constant on $M$.  In particular, the groups of discrete 
gauge transformations of $\hP$ and $P$ can be identified. 
As explained in \cite{wa}, we have:
\be
\ad_P( u)=\ad_{\Gamma(P)}(\bar{ u})\, , \quad \forall\,\,  u\in \Aut_b(P)~~,
\ee
where $\bar{ u}$ is the discrete remnant of $ u$. This implies
that $\Aut_b(P)$ acts on $\cS$ through based automorphism of the
integral duality structure $\bDelta$ determined by $P$:
\be
\ad_P( u)\in \Aut_b(\bDelta)\, , \quad \forall \,\,  u\in \Aut_b(P)
\ee
and hence induces an integral duality transformation of $\Psi$. 

\begin{prop}
Let $\cJ$ be a taming of $\cS$. Then $(\Psi,\cA)\in
\cC^\infty(\Sigma,\cS)\times \Conn(P)$ is a polarized abelian dyon
with respect to $\cJ$ if and only if $(\Psi_ u,\cA_ u)$ is a polarized
abelian dyon with respect to the taming $\cJ_ u\eqdef \ad_P( u)\circ
\cJ\circ \ad_P( u)^{-1}$. In particular, the action $\mu$ of
$\Aut_b(P)$ restricts to bijections:
\be
\mu( u): \Dyons(M,g,P,\cJ)\xrightarrow{\sim} \Dyons(\Sigma,g,P,\cJ_ u)\, , \quad \forall \,\, u\in \Aut_b(P)\, .
\ee
\end{prop}

\begin{proof}
Applying $\ad_P( u)$ shows that the polarized Bogomolny equation
\eqref{eq:Bogomolny} relative to $\cJ$ is equivalent with the
polarized Bogomolny equation relative to $\cJ_ u$:
\be
\star_{h,\cJ_ u}\cV_{\cA_ u}=\dd_\cD\Psi_ u~~,
\ee
where we used Proposition \ref{prop:GaugeBogomolny} and the fact that
$\ad_P( u)$ commutes with $\cD$, which is proved in \cite{wa}.
\end{proof}


\subsection{The case when $Z(P)$ has trivial monodromy on $\Sigma$}


Suppose that $Z(P)$ has trivial monodromy on $\Sigma$, so the duality
structure $\Delta:=\Delta(P)$ is holonomy-trivial and its Siegel system is
trivial. Then there exists a flat trivialization of $\Delta$ which
identifies the integral duality structure $\bDelta=(\Delta,\cL)$ of
$P$ with $(\Sigma\times \R^{2n}, \omega_{2n},\dd, \Sigma\times
\Lambda_\frt)$. Notice that a Higgs field identifies with a smooth map
from $\Sigma$ to $\R^{2n}$, which we decompose into maps
$\Phi,\Upsilon:\Sigma\rightarrow \R^n$ according to:
\ben
\label{PsiDec}
\Psi=\begin{pmatrix} -\Phi \\ \Upsilon \end{pmatrix}
\een

\begin{prop}
A pair $(\Psi, \cA)\in \cC^\infty(\Sigma, \R^{2n})\times \Conn(P)$
satisfies the polarized Bogomolny equations iff:
\ben
\label{HiggsPot}
\dd\Psi=\begin{pmatrix} E \\ -\cR E -\cI \ast_h B \end{pmatrix}~~,
\een
where $E\in \Omega^1(\Sigma,\R^n)$ and $B\in \Omega^2(\Sigma,\R^n)$
are determined by $\cV_\cA$ through the relation:
\ben
\label{cVEB}
\cV_\cA=\begin{pmatrix} B \\ -\cR B+\cI \ast_h E\end{pmatrix}\, .
\een
\end{prop}

\begin{proof}
In the chosen trivialization of $\Delta$, tamings identify
with taming maps $\cJ:\Sigma\rightarrow \GL(2n,\R)$. By
Proposition \ref{prop:cNTaming}, the latter have the form:
\be
\cJ=  
\begin{pmatrix} 
  \cI^{-1} \cR & \cI^{-1} \\
-\cI  - \cR\cI^{-1}\cR & - \cR \cI^{-1} 
\end{pmatrix}~~,
\ee
where $\cN=\cR+\i \cI:\Sigma\rightarrow \mathbb{S}\H^n$ is the
corresponding period matrix map. A similar equation relates the taming
$\hcJ=\cJ\circ p$ of $\hDelta$ to the period matrix map $\hcN\eqdef
\cN\circ p=\hcR+\i \hcI:M\rightarrow \GL(2n,\R)$ (where $\hcR\eqdef
\cR\circ p$ and $\hcI\eqdef \cI\circ p$) defined on $M$.  By Lemma
\ref{lemma:twistedselfdual}, the adjoint curvature of $\cA$ is twisted
self-dual with respect to $\hcJ$ if and only if it has the form:
\ben
\label{hcVF}
\cV_\hcA =\begin{pmatrix} \hF\\G_g(\hcN,\hF)\end{pmatrix}
\een
for some $\hF\in \Omega^2(M,\R^n)$, where:
\be
G_g(\hcN,\hF) \eqdef - \hcR\, \hF - \hcI \ast_g \hF\, .
\ee
Lemma \ref{lemma:astdec} gives:
\be
G_g(\hcN,\hF)_\top= - \hcR\hF_\top -\hcI \ast_h\hF_\perp\, , \quad G_g(\hcN,\hF)_\perp= - \hcR\hF_\perp+\hcI \ast_h \hF_\top\, .
\ee
Thus \eqref{hcVF} amounts to:
\ben
\label{eqcVhat}
(\cV_\hcA)_\top =\begin{pmatrix} \hF_\top\\ - \hcR\hF_\top -\hcI \ast_h\hF_\perp \end{pmatrix}\, , \quad (\cV_\hcA)_\perp =\begin{pmatrix} \hF_\perp\\ - \hcR\hF_\perp+\hcI \ast_h \hF_\top\end{pmatrix}~~. 
\een 
These relations imply:
\be
\hF_\top=p^\ast(E)\, , \quad \hF_\perp=p^\ast(B)
\ee
for some $E\in \Omega^1(\Sigma,\R^n)$ and $B\in
\Omega^2(\Sigma,\R^n)$. Using \eqref{eqcVhat}, we conclude that
\eqref{cVtp} is equivalent with:
\be
\dd\Psi=\begin{pmatrix} E \\ - \cR E -\cI \ast_h B \end{pmatrix}\, , \quad \cV_\cA=\begin{pmatrix} B \\ -\cR B+\cI \ast_h E\end{pmatrix}~~.
\ee
Notice that the pair $(E,B)$ is uniquely determined by $\cV_\cA$
through the second of these relations.
\end{proof}

\noindent
We will refer to the forms $E\in \Omega^1(\Sigma,\R^n)$ and $B\in
\Omega^2(\Sigma,\R^n)$ as the {\em electrostatic} and {\em magnetostatic} 
field strengths of $\cA$. Equation \eqref{HiggsPot} 
is equivalent to the system:
\be
\dd \Phi=E\, , \quad \dd\Upsilon= - \cR E -\cI \ast_h B\, ,
\ee
which in turn amounts to:
\ben
\label{eq:sem}
\vec{E}=-\grad_h\Phi\, , \quad \cI \vec{B} + \cR\vec{E}=-\grad_h\Upsilon\, ,
\een
where we defined $\vec{E},\vec{B}\in \cC^\infty(\Sigma,T\Sigma)\otimes
\R^n$ through:
\be
\vec{E}=E^\sharp\, , \quad \vec{B}=(\ast_h B)^\sharp\, .
\ee
Here $\sharp$ denotes the musical isomorphism given by raising of
indices with the metric $h$. Equation $\dd \cV_\cA=0$ amounts to
the system:
\ben
\label{eq:M}
\dd B=0\, , \quad \dd(-\cR B+\cI \ast_h E)=0\Longleftrightarrow \div_h \vec{B}=0\, ,\quad \div_h(\cR\, \vec{B} + \cI \, \vec{E})=0\, .
\een


\subsection{Polarized dyons in prequantum electrodynamics}


Prequantum electrodynamics defined on $M$ corresponds to setting
$n=1$ and $\cR=\frac{\theta}{2\pi}, \cI=\frac{4\pi}{g^2}$ with constant $\theta\in \R$ in the
previous subsection (see Appendix \ref{app:local}). Then:
\be
\cJ=\cJ_\theta\eqdef \begin{pmatrix}  \frac{g^2\theta}{8\pi^2} &
\frac{g^2}{4\pi}\\ -\frac{4\pi}{g^2}-\frac{g^2\theta^2}{16\pi^3} &
-\frac{g^2\theta}{8\pi^2}\end{pmatrix}~~,
\ee
and relation \eqref{cVEB} becomes:
\ben
\label{cVEB2}
\cV_\cA=\begin{pmatrix} B \\ - \frac{\theta}{2\pi} B+\frac{4\pi}{g^2} \ast_h E\end{pmatrix}~~.
\een
In this case, relations \eqref{eq:sem} reduce to:
\ben
\label{potentials}
\vec{E}=-\grad_h \Phi\, , \quad \frac{4\pi}{g^2}\vec{B} + \frac{\theta}{2\pi}\vec{E}=-\grad_h\Upsilon\, ,
\een
which imply $\vec{B}=-\frac{g^2}{4\pi} \, \grad_h(\Upsilon - \frac{\theta}{2\pi} \Phi)$ and:
\ben
\label{M1}
\curl_h\vec{E}=\curl_h\vec{B}=0~~.
\een
On the other hand, relations \eqref{eq:M} become:
\ben
\label{M2}
\dd B=\dd (\ast_h E)=0\Longleftrightarrow \div_h \vec{B}=\div_h \vec{E}=0~~.
\een
Equations \eqref{M1} and \eqref{M2} describe source-free Maxwell
electromagnetostatics on $\Sigma$, where the vector fields
$\vec{E},\vec{B}\in \cA(\Sigma)$ are the classical static electric and
magnetic fields. Relation \eqref{potentials} shows that $\Phi$ and
$\frac{g^2}{4\pi} \Upsilon-\frac{g^2\theta}{8\pi^2} \Phi$ are globally-defined
classical electrostatic and magnetostatic scalar potentials. Since
$H^1(\Sigma,\R)$ need not vanish, relation \eqref{M1} need not imply
\eqref{potentials}. Hence polarized dyons describe special
potential electromagnetostatic configurations, i.e.  solutions of the
static Maxwell equations on $\Sigma$ which admit both a
globally-defined electric scalar potential and a globally-defined
magnetic scalar potential. Even though $\vec{E}$ and $\vec{B}$ satisfy
\eqref{M2}, such solutions need {\em not} admit a globally-defined
vector electric or vector magnetic potential, since $H^2(\Sigma,\R)$
may be non-zero. The condition that the configuration
$(\vec{E},\vec{B})$ admits globally-defined electric and magnetic
scalar potentials is a consequence of the fact that a polarized
dyon originates from a principal connection defined on $\hP$,
which itself is a consequence of the DSZ integrality condition. We
formalize this as follows.

\begin{definition}
An {\em electromagnetostatic configuration} defined on $(\Sigma,h)$ is
a pair of vector fields $(\vec{E},\vec{B})\in \cA(\Sigma)\times
\cA(\Sigma)$ which satisfies the static source-free Maxwell equations
\eqref{M1} and \eqref{M2}. Such a configuration is said to be {\em
potential} if there exist real-valued smooth functions $\Phi,\Upsilon$
defined on $\Sigma$ such that relations \eqref{potentials} hold.
\end{definition}

\noindent Let $\EMC(\Sigma,h)$ denote the vector space of all
electromagnetostatic configurations defined on $(\Sigma,h)$ and
$\EMC_\pot(\Sigma,h)$ denote the subspace of those configurations
which are potential. Consider the map:
\be
H_\theta:\Dyons(\Sigma,h,P,\cJ_\theta)\rightarrow \EMC_\pot(\Sigma,h)
\ee
which associates the electromagnetostatic configuration
$(\vec{E},\vec{B})$ defined through relation \eqref{cVEB2} to the
polarized dyon $(\Psi,\cA)$. As in Subsection \ref{subsec:trivD}, the
chosen trivialization of $\Delta$ induces isomorphisms:
\be
H^2_{\cD}(\Sigma,\cS)\simeq H^2(M,\R^{2n})\, , \quad H^2(\Sigma,\cL)\simeq H^2(\Sigma,\Lambda_\frt)\, ,
\ee
which allow us to identify the morphism $j:H^2(\Sigma,\cL)\rightarrow
H^2_{\cD}(\Sigma,\cS)$ with the morphism $\iota$ appearing in the long
exact sequence:
\be
\ldots \rightarrow H^1(\Sigma,A)\rightarrow H^2(\Sigma,\Lambda_\frt)\xrightarrow{i} H^2(\Sigma,\R^{2n})\rightarrow H^2(\Sigma,A)\rightarrow \ldots
\ee
induced by the exponential sequence:
\be
0\rightarrow \Lambda_\frt \rightarrow \R^{2n} \rightarrow A \rightarrow 0~~.
\ee
The relation $[\cV_\cA]_{\cD}=2\pi j_\ast(c(P))$ becomes:
\be
[\cV_\cA]= 2\pi i_\ast(c(P))\, ,
\ee
where $[\omega]$ denotes the de Rham cohomology class of a form
$\omega\in \Omega^k(M,S)$.  Conversely, any closed two-form $\cV\in
\Omega^2_\cl(\Sigma,\R^{2n})$ which satisfies $[\cV]= 2\pi
i_\ast(c(P))$ identifies with the curvature of some principal
connection $\cA$ on $P$. For any $(\vec{E},\vec{B})\in
\EMC(\Sigma,h)$, let:
\be
\cV^\theta_{\vec{E},\vec{B}}\eqdef \begin{pmatrix} B \\ - \frac{\theta}{2\pi} B+\frac{4\pi}{g^2} \ast_h E\end{pmatrix}~~,
\ee
where $E\eqdef \vec{E}^\flat$ and $B\eqdef -\ast_h (\vec{B}^\flat)$.

\begin{thm}
The image of the map $H_\theta$ coincides with the set:
\be
\EMC_\pot(\Sigma,h;\theta, c(P))\eqdef \{(\vec{E},\vec{B})\in
\EMC_\pot(\Sigma,h)~\vert~[\cV^\theta_{\vec{E},\vec{B}}]= 2\pi  \iota_\ast(c(P))\}~~.
\ee
Moreover, the $H_\theta$-preimage of a configuration
$(\vec{E},\vec{B})\in \EMC_\pot(\Sigma,h;\theta, c(P))$ is given by:
\be
H_\theta^{-1}(\{(\vec{E},\vec{B})\})=\Big\{(\Psi,A)\in
\Dyons(\Sigma,h,\cJ_\theta)~\vert~\cV_\cA=\cV^\theta_{\vec{E},\vec{B}}
~\&~\grad_h\Psi=\begin{pmatrix}\vec{E}\\
-\frac{\theta}{2\pi} \vec{E}-\frac{4\pi}{g^2}\vec{B}\end{pmatrix}\Big\}~~.
\ee
\end{thm}

\begin{proof}
Let $(\Psi,\cA)$ be a polarized abelian dyon. Then
$H_{\theta}(\Psi,\cA)\in\EMC_\pot(\Sigma,h)$ and by equation
\eqref{cVEB2}, there exists an electromagnetostatic configuration
$(\vec{E},\vec{B})$ such that:
\be
\cV^\theta_{\cA} = \begin{pmatrix} B \\ - \frac{\theta}{2\pi} B+\frac{4\pi}{g^2} \ast_h E\end{pmatrix}\, .
\ee
Since $\cV^\theta_{\cA}$ is the curvature of a connection $\cA$ on $P$, 
it satisfies $[\cV^{\theta}_\cA]_{\cD}= 2\pi j_\ast(c(P))$ and hence 
$(\vec{E},\vec{B}) \in \EMC_\pot(\Sigma,h;\theta, c(P))$. 
On the other hand, if $(\vec{E},\vec{B}) \in \EMC_\pot(\Sigma,h;\theta, c(P))$, 
equation $[\cV^\theta_{\vec{E},\vec{B}}]=\iota_\ast(c(P))$ 
implies that there exists a connection $\cA$ on $P$ whose curvature 
is $\cV^\theta_{\vec{E},\vec{B}}$ and since $(\vec{E},\vec{B})$ is a 
potential electromagnetostatic solution, choosing potentials $\Phi$ 
and $\Upsilon$ and defining:
\ben
\Psi=\begin{pmatrix} -\Phi \\ \Upsilon \end{pmatrix}
\een

\noindent
we conclude that $(\Psi,\cA)$ is a polarized abelian dyon whence 
$(\vec{E},\vec{B}) \in \mathrm{Im}(H_{\theta})$. It is now clear that 
the only freedom in choosing a preimage of $(\vec{E},\vec{B})$ by $H_{\theta}$ 
lies only in choosing the potentials $\Phi$ and $\Upsilon$ as prescribed in the 
statement of the theorem and hence we conclude.
\end{proof}


\subsection{Polarized Bogomolny equations on the punctured Euclidean 3-space}


In this subsection, we construct families of solutions (which generalize the
dyons of ordinary electromagnetism) on the punctured Euclidean space
$\Sigma = \R^3_0 \eqdef \R^3\backslash \left\{0\right\}$. Spherical
coordinates give a diffeomorphism:
\be
\R^{3}_0 \simeq \R_{>0}\times \rS^2\, ,
\ee
and we denote by $r\in \R_{>0}$ the radial coordinate. The metric $h$
reads:
\be
h = \dd r^2 + r^2 h_{\rS^2}\, ,
\ee
where $h_{\rS^2}$ is the unit round metric of $\rS^2$. We have
$\nu_h=r^2\dd r\wedge \nu_{\rS^2}$, where $\nu_{\rS^2}$ is the volume
form of $\rS^2$. Let $q:\R^3_0\rightarrow \rS^2$ be the map defined
though:
\be
q(x)\eqdef \frac{x}{||x||}~~,
\ee
where $||x||$ is the Euclidean norm of $x$. Up to isomorphism, any
Siegel bundle of rank $n$ and type $\frt\in \Div^n$ defined on $P$
on $\R^3_0$ is the pull-back through $q$ of a Siegel bundle of
the same rank and type defined on $\rS^2$, which reduces to a
principal torus bundle since $\rS^2$ is simply connected. Accordingly,
the duality structure $\Delta$ of $P$ is holonomy trivial and we can
fix a global flat trivialization in which $\Delta$ identifies with:
\ben
\label{DeltaTriv}
\Delta \equiv (\Sigma\times \R^{2n}, \omega_{2n},\dd)\, ,
\een
and the Dirac system $\cL$ determined by $P$ identifies with
$\Sigma\times \Lambda_\frt$. Hence a Bogomolny pair consists
of a map $\Psi\in \cC^\infty(\Sigma,\R^{2n})$ and a connection $\cA\in
\Conn(P)$ whose curvature can be viewed as a vector-valued two-form
$\cV_\cA\in \Omega^2(\Sigma,\R^{2n})$. Moreover, a taming of $\Delta$
can be viewed as a map $\cJ\in \cC^\infty(\Sigma,\GL(2n,\R))$ which
squares to minus the identity everywhere and satisfies
$\cJ^t(\sigma){\hat \omega}_{2n}\cJ(\sigma)=\hat{\omega}_{2n}$ for all
$\sigma\in \Sigma$.

Let us assume that $\cJ$ and $\Psi$ depend only on $r$. Then
$\dd\Psi=\partial_r\Psi\dd r$ and $\ast_h(\dd r)=\iota_{\partial_r}
\nu_h=r^2 q^\ast(\nu_{\rS^2})$, hence the polarized Bogomolny equation reads:
\ben
\label{eq:eqRS2}
\cV_\cA = - r^2\cJ(r)\frac{\dd\Psi}{\dd r}\, q^\ast(\nu_{S^2})\, ,
\een
Since $\cV_\cA$ and $\nu_{\rS^2}$ are closed, a necessary condition
for this equation to admit solutions $\cA\in \Conn(P)$ is:
\be
\frac{\dd}{\dd r}(r^2\cJ\frac{\dd}{\dd r}\Psi\,) = 0\, .
\ee
The general solution of this integrability condition is:
\be
\Psi = - \int \frac{\cJ v}{2 r^2} \dd r + v^{\prime}\, ,
\ee
for constant vectors $v , v^{\prime} \in \R^{2n}$. Using this in
\eqref{eq:eqRS2} gives:
\be
\cV_\cA = - \frac{1}{2} v\, q^\ast(\nu_{\rS^2})\, .
\ee
The last relation determines $\cA$ up to a transformation of the form:
\be
\cA\rightarrow \cA+\omega~~
\ee
where $\omega\in \Omega^1_\Ad(P,\R^{2n})$ is closed and hence
corresponds to a closed 1-form $\omega'\in
\Omega^1(\Sigma,\R^{2n})$ (recall that $\dd_{\cD}=\dd$ in our
trivialization of $\cS$). Notice that $\omega'$ need not be exact
since $\Sigma=\R^3_0$ is not contractible.

\subsubsection{The integrality condition for $v$}

Since $\rS^2$ is a deformation retract of $\Sigma$, the map $q^\ast:
H^\ast(\rS^2,\mathbb{R}^{2n}/\Lambda_{\frt})\rightarrow 
H^\ast(\Sigma,\mathbb{R}^{2n}/\Lambda_{\frt})$ induced by $q$ on
cohomology is an isomorphism of groups for any abelian group of
coefficients $A$. Since $H_1(\rS^2,\Z)=0$, the universal coefficient
theorem for cohomology gives:
\ben
\label{iso1}
H^2(\rS^2,\Lambda_\frt)\simeq_\Z \Hom_\Z(H_2(\rS^2,\Z), \Lambda_\frt)\simeq_\Z \Lambda_\frt~~,
\een
where the last isomorphism is given by evaluation on the fundamental
class $[\rS^2]\in H_2(\rS^2,\Z)$. This allows us to view the
characteristic class $c(P)=q^\ast(c(P_0))$ as an element of
$\Lambda_\frt$.  Moreover, we have isomorphisms of vector spaces:
\ben
\label{iso2}
H^2(\rS^2,\R^{2n})\simeq_\R H^2(\rS^2,\Z)\otimes_\Z \R^{2n} \simeq_\R \R^{2n}~~,
\een
the last of which takes $u\otimes_\Z x$ to $x$ for all $x\in \R^{2n}$.
Through the isomorphisms \eqref{iso1} and \eqref{iso2}, the map
$i_\ast:H^2(\rS^2,\Lambda_\frt)\rightarrow H^2(\rS^2,\R^{2n})$
corresponds to the inclusion $\Lambda_\frt\subset \R^{2n}$ and
$2\pi c(P)$ identifies with the de Rham cohomology class of
$\cV_\cA$.  The free abelian group $H^2(\rS^2,\Z)$ is
generated by half of the Euler class $e(\rS^2)$
of $\rS^2$, which satisfies:
\be
\int_{\rS^2} e(\rS^2)=\chi(\rS^2)=1+(-1)^2=2~~.
\ee
On the other hand, the de Rham cohomology class $u\eqdef
\left[\frac{\nu_{\rS^2}}{4\pi}\right]\in H^2(\rS^2,\R^{2n})$
satisfies:
\be
u=i_\ast\left(\frac{1}{2}e(\rS^2)\right)~~.
\ee
We have:
\be
[\cV_\cA] = 2\pi i_\ast(c(P))~~,
\ee
which implies $v\in \Lambda_\frt$ as well as:
\be
c(P_0)=\frac{1}{2} v \, e(\rS^2)~~\mathrm{and}~~i_\ast(c(P_0))=v u~~.
\ee
Hence we obtain an integrality condition for the integration constant 
$v$, which guarantees the existence of a solution and determines through the
previous equation the topological type of the bundle $P$ carrying the solution.


\appendix


\section{Local abelian gauge theory}
\label{app:local}


This appendix discusses the duality-covariant formulation and global
symmetries of source-free $\U(1)^n$ abelian gauge theory on a
{\em contractible} oriented four-manifold $M$ endowed with an arbitrary
metric of signature $(3,1)$, with the goal of motivating the geometric
model introduced in Section \ref{sec:classical} and of making contact
with the physics literature. The local study of electromagnetic
duality for Maxwell electrodynamics in four Lorentzian dimensions has
a long history, see \cite{Deser:1976iy,Deser:1981fr} as well as
\cite{Olive:1995sw} and its references and citations. Let $\ast_g$
be the Hodge operator of $(M,g)$.

The prototypical example of local abelian gauge theory is given by
classical electrodynamics, which is defined by the Lagrangian density
functional:
\be
\mathfrak{L}[A] = - \frac{4\pi}{g^2} (F_A)_{ab}
(F_A)^{ab}  + \frac{\theta}{2\pi} (F_A)_{ab}\, (\ast_g F_A)^{ab}\, ,
\qquad F=\dd A~~\mathrm{with}~~A\in \Omega^1(M)~~,
\ee 
where $\theta\in\R$ is the theta angle. Classical Maxwell theory is
obtained for $\theta=0$. Below, we discuss the more general
construction of local abelian gauge theories, which is motivated by
the local structure of four-dimensional supergravity and
supersymmetric field theory and which involves a finite number of
abelian gauge fields.


\subsection{Lagrangian formulation}


Given $n\in \Z_{>0}$ and a Lorentzian metric $g$ on $M$, consider
the following generalization of the previous Lagrangian density, which
allows all couplings to be smooth real-valued functions:
\be
\mathfrak{L}[A^1,\hdots, A^n] = - \cI_{\Lambda \Sigma}\,
\langle F^{\Lambda}_A, F^{\Sigma}_A\rangle_g + \cR_{\Lambda \Sigma}\,
\langle F^{\Lambda}_A, \ast_g F^{\Sigma}_A\rangle_g~~.
\ee
Here $\Lambda,\Sigma = 1,\hdots , n$ and :
\be
F^{\Lambda}_A = \dd A^{\Lambda}\, , \qquad \Lambda = 1, \hdots n\, ,
\ee
where $A^\Lambda\in \Omega^1(M)$ and $\cR_{\Lambda \Sigma},
\cI_{\Lambda \Sigma} \in \cC^\infty(M)$ are the components of an
$\R^n$-valued one-form $A\in \Omega^1(M,\R^n)$ and of functions
$\cR,\cI\in \cC^\infty(M, \Mat_s(n,\R))$ valued in the space
$\Mat_s(n,\R)$ of symmetric square matrices of size $n$ with real
entries. To guarantee a positive-definite kinetic term we must assume
that $\cI$ is positive-definite everywhere. Notice that classical
electrodynamics corresponds to $n=1$ with $\cR=\frac{\theta}{2\pi}$
and $\cI=\frac{4\pi}{g^2}$. Since $\ast_g^2F^\Lambda=-F^\Lambda$, the
action is equivalent to:
\be
\cS[A^1,\ldots, A^n]=\int \mathfrak{L}[A^1,\hdots, A^n]\vol_g= - \int
\left(\cI_{\Lambda \Sigma}\, F^{\Lambda}_A\wedge \ast_g F^{\Sigma}_A
+ \cR_{\Lambda \Sigma}\, F^{\Lambda}_A\wedge F^{\Sigma}_A\right)~~.
\ee
The partial differential equations obtained as critical points of the
variational problem with respect to compactly supported variations
are:
\be
\dd(\cR_{\Lambda\Sigma} F^{\Sigma}_A + \cI_{\Lambda \Sigma} \ast_g F^{\Sigma}_A) = 0\, ,
\ee
or, in matrix notation:
\be
\dd(\cR F_A + \cI \ast_g F_A) = 0\, ,
\ee
where $F_A=\dd A\in \Omega^2_\cl(M,\R^n)=\Omega^2_\ex(M,\R^n)$ is an $\R^n$-valued
closed (and hence exact) two-form. These equations define {\em
classical local abelian gauge theory}. Since both $\cR$ and $\cI$ are
symmetric and $\cI$ is positive definite, the pair $(\cR,\cI)$ is
equivalent to a map:
\be
\cN \eqdef \cR + \i\cI \colon M\to \mathbb{SH}_n\, , 
\ee
where $\mathbb{SH}_n$ denotes the Siegel upper half space of symmetric
$n\times n$ complex matrices with positive definite imaginary part.

\begin{definition}
A {\em period matrix map} of size $n$ on $M$ is a smooth function
$\cN\in\cC^\infty(M,\mathbb{SH}_n)$ valued in $\mathbb{SH}_n$. We
denote the set of such maps by $\Per_n(M)$.
\end{definition}

\noindent When the metric $g$ is fixed, classical local abelian gauge
theory is uniquely determined by a choice of period matrix map.
 
\begin{definition}
\label{def:localabeliangaugeth}
Let $\cN = \cR + \i\cI$ be a period matrix map of size $n$. The {\em
local abelian gauge theory} associated to $\cN$ is defined through the
following system of equations:
\ben
\label{eq:localabelianeqsI}
\dd(\cR\, F_A + \cI \ast_g F_A) = 0\, , 
\een
with unknowns given by the vector valued one-form $A\in
\Omega^1(M,\R^n)$, where $F_A \eqdef \dd A \in \Omega^2_\cl(M,\R^n)$.
\end{definition}

\begin{remark}
The equations of motion of local abelian gauge theory have a
\emph{gauge symmetry} consisting of transformations of the type:
\be
A\mapsto A + \dd \alpha\, , \qquad \alpha\in \cC^{\infty}(M,\R^n)\, .
\ee
Variables $A\in \Omega^1(M,\R^n)$ related by such
gauge transformations should be viewed as physically
equivalent. 
\end{remark}

\noindent Let us write the equations of motion
\eqref{eq:localabelianeqsI} in a form amenable to geometric
interpretation. Given a period matrix map $\cN = \cR + \i\cI$ and a
vector of two-forms $F\in \Omega^2(M,\R^n)$, define:
\be
G_g(\cN,F) \eqdef - \cR\, F - \cI \ast_g F\, .
\ee
Then the condition $\dd F=0$ together with the equations of motion
\eqref{eq:localabelianeqsI} are equivalent with the single equation:
\ben
\label{eq:localVeq}
\dd \cV = 0\, , 
\een
where the $\R^{2n}$-valued two-form $\cV$ is related to $F$ by:
\ben
\label{eq:localVeqII}
\cV = \begin{pmatrix} 
F  \\
G_g(\cN,F) 
\end{pmatrix} \in \Omega^2(M,\R^{2n}) = \Omega^2(M,\R^{n}\oplus \R^n)\, .
\een
As the following lemma shows, not every vector-valued two-form $\cV
\in \Omega^2(M,\R^{2n})$ can be written as prescribed by equation
\eqref{eq:localVeqII}.

\begin{lemma}
\label{lemma:twistedselfdual}
Let $\cN=\cR+ \i \cI\in \Per_n(M)$ be a period matrix map of size $n$
on $M$. A vector valued two-form $\cV\in \Omega^2(M,\R^{2n})$ can be
written as:
\be
\cV =   
\begin{pmatrix} 
F  \\
G_g(\cN,F)  
\end{pmatrix} 
\ee
for some $F\in \Omega^2(M,\R^{n})$ if and only if:
\ben
\label{eq:twistedselfdualitylocal}
\ast_g\cV = - \cJ \cV \, , 
\een
where $\cJ\in \cC^\infty(M, \GL(2n,\R))$ is the
matrix-valued map defined through:
\be
\cJ=  
\begin{pmatrix} 
\cI^{-1} \cR & \cI^{-1} \\
-\cI  - \cR\cI^{-1}\cR & - \cR \cI^{-1} 
\end{pmatrix}~~.
\ee
We have $\cJ^2 = -1$. Moreover, $F$ with the property above is
uniquely determined by $\cV$.
\end{lemma}

\begin{remark}
Notice that $\cJ$ can be viewed as a complex structure defined
on the trivial real vector bundle of rank $2n$ over $M$. For
classical electrodynamics, the taming map is constant and given by:
\be
\cJ=\begin{pmatrix} \frac{g^2\theta}{8\pi^2} &
\frac{g^2}{4\pi}\\ -\frac{4\pi}{g^2}-\frac{g^2\theta^2}{16\pi^3} &
- \frac{g^2\theta}{8\pi^2}\end{pmatrix}~~.
\ee
In this case, the period matrix map is constant and given by: 
\be
\cN=\frac{4\pi}{g^2}+\i\frac{\theta}{2\pi}~~,
\ee
being traditionally denoted by $\tau$. 
\end{remark}

\begin{proof}
If:
\be
\cV =   
\begin{pmatrix} 
F  \\
G_g(\cN,F)  
\end{pmatrix} \, ,
\ee
then direct computation using the fact that $\ast_g^2 = -1$ on
two-forms shows that $\cV$ satisfies $\ast_g\cV = - \cJ \cV$. On the
other hand, writing $\cV=\begin{pmatrix} F \\ G
\end{pmatrix}$ with $F,G\in \Omega^2(M,\R^n)$ shows that the
equation $\ast_g\cV = - \cJ \cV$ is equivalent to:
\be
\begin{pmatrix} \ast_g F \\ \ast_g G \end{pmatrix}
=\begin{pmatrix} -\cI^{-1}\cR F - \cI^{-1}G \\ -(\cI  + \cR
  \cI^{-1} \cR)F  - \cR \cI^{-1} G\end{pmatrix} \, ,
\ee
which in turn amounts to $G = G_g(\cN,F)$.
\end{proof}

\noindent
Let $\omega_{2n}$ be the standard symplectic form on $\R^{2n}$, which
in our conventions has the following matrix in the canonical basis
$\cE = (e_1 , \hdots , e_{n} , f_1,\hdots , f_{n})$ of the latter:
\ben
{\hat \omega}_{2n} =  
\begin{pmatrix} 
0 & I_n\\
- I_n & 0
\end{pmatrix}~~. 
\een
Here $I_n$ is the identity matrix of size $n$. We have:
\ben
\label{eq:omegastandard}
\omega_{2n} \eqdef \sum_{a=1}^n  e_a^{\ast} \wedge f_a^{\ast}\, ,
\een
where $\cE^{\ast} = (e_1^{\ast} , \hdots , e^{\ast}_{n} ,
f^{\ast}_1,\hdots , f^{\ast}_{n})$ is the basis dual to $\cE = (e_1 ,
\hdots , e_{n} , f_1,\hdots , f_{n})$. The following result gives a
geometric interpretation of the equations of motion
\eqref{eq:twistedselfdualitylocal}. Recall that an almost complex
structure $J$ on $\R^{2n}$ is called a {\em taming} of the standard
symplectic form $\omega_{2n}$ if:
\be
\omega_{2n}(J \xi_1, J \xi_2 ) = \omega_{2n}(\xi_1,\xi_2)\, , \quad \forall\,\, \xi_1 , \xi_2 \in \R^{2n}\, ,
\ee
and:
\be
\omega_{2n}(J\xi,\xi) > 0\, , \quad \forall\,\, \xi\in \R^{2n}\backslash\left\{0\right\}\, .
\ee

\begin{definition}
A {\em taming map} of size $2n$ defined on $M$ is a smooth map $\cJ\in
\cC^\infty(M,\GL(2n,\R))$ such that $\cJ(m)$ is a taming of
$\omega_{2n}$ for every $m\in M$. We denote the set of all such maps
by $\fJ_n(M)$.
\end{definition}

\begin{prop}
\label{prop:cNTaming}
A matrix-valued map $\cJ\in \cC^\infty(M, \GL(2n,\R))$ can be
written as:
\ben
\label{eq:Jlocaltaming}
\cJ= 
\begin{pmatrix} 
 \cI^{-1} \cR & \cI^{-1} \\
- \cI - \cR\cI^{-1}\cR & - \cR \cI^{-1} \, ,
\end{pmatrix}
\een
in terms of a period matrix map $\cN = \cR + \i\cI\in \Per_n(M)$ iff
$\cJ\in\fJ_n(M)$.
\end{prop}

\begin{proof} 
If $\cJ$ is taken as in equation \eqref{eq:Jlocaltaming} for a period
matrix map $\cN$ then direct computation shows that $\cJ(m)$ is a
taming of $\omega_{2n}$ for all $m\in M$. For the converse, assume
that $\cJ(m)\in \GL(2n,\R)$ is a taming of $\omega_{2n}$ for
all $m\in M$ (we omit to indicate the evaluation at $m$ for
ease of notation). Let $\cE = (e_1 , \hdots , e_{n} , f_1,\hdots ,
f_{n})$ the canonical basis of $\R^{2n}$. The vectors $\cE_f = ( f_1,\hdots , f_{n})$
form a basis over $\C$ of the complex vector space
$(\R^{2n},\cJ(m))\simeq \C^n$, hence there exists a unique map
$\tau\in \cC^\infty(M,\Mat(n,\C))$ which satisfies:
\ben
\label{eq:localperiod}
e_{a} = \tau(m)_{ab}\, f_{b}\, , \quad \forall\,\, a = 1, \hdots , n\, ,
\een
where we use Einstein summation over repeated indices. Thus:
\be
\delta_{ab}=\omega_{2n}(e_{a} , f_{b}) = \omega_{2n}(\tau(m)_{ac}\, f_{c} , f_{b}) =
 \Im(\tau(m))_{ac}\,\omega_{2n}(\cJ(m) f_{c} , f_{b})~~,
\ee
which implies that $\Im(\tau(m))$ is symmetric and
positive-definite. Using the previous equation and compatibility of
$\cJ(m)$ with $\omega_{2n}$, we compute:
\be
0 = \omega_{2n}(e_a , e_b) = \mathrm{Re}(\tau(m))_{ba} -
\mathrm{Im}(\tau(m))_{bc}\, \omega_{2n}(\cJ(m)(e_a),f_c) = \mathrm{Re}(\tau(m))_{ba}
- \mathrm{Re}(\tau(m))_{ab}\, ,
\ee
which shows that $\Re(\tau(m))$ is symmetric. Hence the smooth map
$\cN\in\cC^\infty(M, \mathbb{SH}^n)$ defined through $\cN=\cR+\i
\cI$, where:
\be
\cR \eqdef  \mathrm{Re}(\tau)~~, \qquad \cI \eqdef \mathrm{Im}(\tau)\, ,
\ee
is a period matrix map. Equation \eqref{eq:localperiod} gives:
\beqa
&&\cJ(m)(e_a) = \cR(m)_{ab} \cI^{-1}(m)_{bc}\, e_c
-\cR(m)_{ab}\cI(m)^{-1}_{bc}\cR(m)_{cd}\, f_d - \cI(m)_{ad}\, f_d\\
&&\cJ(m)(f_a) = \cI(m)^{-1}_{ab}\, e_b -
\cI(m)^{-1}_{ab}\cR(m)_{bc}\,f_c\, ,
\eeqa
which is equivalent to \eqref{eq:Jlocaltaming}.  
\end{proof}
 
\begin{prop}
\label{prop:1to1electromagnetic}
The map $\Theta\colon \Per_n(M) \to \fJ_n(M)$ defined through:
\be
\Per_n(M)\ni \cN= \cR + \i \cI \mapsto
\Theta(\cN)\eqdef
\begin{pmatrix} 
\cI^{-1} \cR & \cI^{-1} \\
- \cI - \cR\cI^{-1}\cR & - \cR \cI^{-1}
\end{pmatrix}  \in \fJ_n(M)
\ee
is a bijection between $\Per_n(M)$ and $\fJ_n(M)$.
\end{prop}

\begin{proof}
Follows directly from the proof of Proposition
\ref{prop:cNTaming}. The inverse of $\Theta$ takes a taming map $\cJ
\in \fJ_n(M)$ to the period matrix $\Theta^{-1}(\cJ) =
\mathrm{Re}(\tau) + \i\, \mathrm{Im}(\tau)$, where, for all $m\in
M$, $\tau(m)$ is the complex symmetric matrix of size $n$
uniquely determined by the expansion $e_a = \tau_{ab} f_b$ of
$e_a$ over $\C$ when $\R^{2n}$ is endowed with the complex structure
$\cJ(m)$.
\end{proof}

\noindent Since $M$ is contractible, we have
$\Omega^2_\cl(M,\R^{2n})=\Omega^2_\ex(M,\R^{2n})$. By the discussion
above, this implies that local abelian gauge theory can be formulated
equivalently as a theory of closed $\R^{2n}$-valued two-forms $\cV\in
\Omega^2_\cl(M,\R^{2n})$ satisfying the condition:
\be
\ast_g \cV = -\cJ \cV
\ee
with respect to a fixed taming map $\cJ\in\fJ_n(M)$. Consequently, the 
theory is uniquely determined by the choice of taming map. 
The condition $\dd\cV = 0$ is equivalent with $\cV = \dd \cA$, where $\cA\in
\Omega^1(M,\R^{2n})$ is considered modulo gauge transformations
$\cA\mapsto \cA + \dd\alpha$ with $\alpha\in \cC^\infty(M,\R^{2n})$.
The map $[\cA] \mapsto \dd\cA$ gives a well-defined bijection $
\Omega^1(M,\R^{2n})\slash\Omega^1_\cl(M,\R^{2n}) \xrightarrow{\sim}
\Omega^2_\cl(M,\R^{2n})$. Thus we can formulate classical local
abelian gauge theory either in terms of \emph{electromagnetic gauge
potentials} $\cA\in \Omega^1(M,\R^{2n})$ taken modulo
gauge-equivalence or in terms of {\em electromagnetic field strengths}
$\cV\in \Omega^2_\cl(M,\R^{2n})$.

\begin{definition}
Let $\cJ\in \fJ_n(M)$ be a taming map. The space of {\em electromagnetic
gauge configurations} of the $\U(1)^n$ local abelian gauge is $\Omega^1(M,\R^{2n})$.
 
Two gauge configurations are called {\em gauge equivalent} if they differ by an exact
one-form. The theory is defined by the {\em polarized self-duality condition} for
$\cA\in \Omega^1(M,\R^{2n})$:
\ben
\label{eq:abeliangaugelocal}
\ast_g \cV_{\cA} = - \cJ \cV_{\cA}~~,~~\mathrm{where}\quad \cV_{\cA} \eqdef \dd\cA\, .
\een
The {\em space of electromagnetic gauge fields} (or {\em electromagnetic gauge
potentials}) of the theory is the linear subspace of
$\Omega^1(M,\R^{2n})$ consisting of those elements which satisfy
\eqref{eq:abeliangaugelocal}:
\be
\cSol_n(M,g,\cJ) \eqdef \left\{ \cA \in \Omega^1(M,\R^{2n})
\,\, \vert\,\, \ast_g\cV_{\cA} = - \cJ \cV_{\cA} \right\}\, .
\ee
Elements $\cA\in \Omega^1(M,\R^{2n})$ are called (electromagnetic) 
\emph{gauge potentials} or \emph{gauge fields}. The space of 
{\em field strength configurations} is the vector space:
\be
\Conf_n(M)\eqdef \Omega^2_\cl(M,\R^{2n})~~,
\ee
while the space of {\em field strengths} is defined through:
\be
\Sol_n(M,g,\cJ)\eqdef \{\cV\in \Conf_n(M)~\vert~\ast_g\cV=\cJ\cV\}~~.
\ee
\end{definition}

\noindent The map $[\cA] \mapsto \dd\cA$ gives a bijection $
\Omega^1(M,\R^{2n})/\Omega^1_\cl(M,\R^{2n})\xrightarrow{\sim} \Conf_n(M)$,
which restricts to a bijection
$\cSol_n(M,g,\cJ)/\Omega^1_\cl(M,\R^{2n})\xrightarrow{\sim}
\Sol_n(M,g,\cJ)$.


\subsection{Duality groups}
\label{sec:symmetrieslocal}


Let $\Diff(M)$ be the group of orientation-preserving
diffeomorphisms of $M$ and $\cJ\in \fJ_n(M)$ be a taming map of rank
$2n$ defined on $M$. For $(\gamma,f)\in \GL(2n , \R)\times \Diff(M)$,
consider the linear isomorphism:
\ben
\label{eq:Action}
\A_{\gamma,f} \colon \Omega^k(M,\R^{2n})\xrightarrow{\sim}
\Omega^k(M,\R^{2n})\, , \quad \omega \mapsto \gamma (f_{\ast}\omega)\, ,
\een
where $f_{\ast}\colon \Omega^k(M,\R^{2n})\to \Omega^k(M,\R^{2n})$ is
the push-forward through the diffeomorphism $f$. This gives a linear
action of $\GL(2n,\R)\times \Diff(M)$ on
$\Omega^k(M,\R^{2n})$. Since this action commutes with the exterior
derivative, it preserves the space $\Conf_n(M)$ of field strength
configurations.

For any $\gamma\in \Sp(2n,\R)$, the map:
\be
\cJ_{\gamma,f}\eqdef \gamma(\cJ\circ f^{-1}) \gamma^{-1}\, ,
\ee
is a taming map. This gives an action $\mu$ of $\Sp(2n,\R)\times
\Diff(M)$ on $\fJ_n(M)$ defined through:
\be
\mu(\gamma,f)(\cJ)\eqdef \cJ_{\gamma,f}\, , \quad \forall\,\, (\gamma,f)\in \Sp(2n,\R)\times
\Diff(M)\, .
\ee

\begin{prop}
\label{prop:SolAction}
For every $(\gamma,f)\in \Sp(2n,\R) \times \Diff(M)$, the map
$\A_{\gamma,f}$ induces by restriction a linear isomorphism:
\be
\A_{\gamma,f}\colon \Sol_n(M,g,\cJ)\xrightarrow{\sim} \Sol_n(M,f_{\ast}(g),\cJ_{\gamma,f})~~,
\ee
where $f_{\ast}(g)$ denotes the push-forward of $g$ by $f\in
\Diff(M)$.
\end{prop}

\begin{remark}
\label{remark:locallagrangian}
If we consider a pair $(\gamma,f)\in \GL(2n,\R)\times \Diff(M)$ with
$\gamma\not\in \Sp(2n,\R)$, then $\cJ_{\gamma,f}$ is not a taming map, so it
does not define a local abelian gauge theory. From a different point
of view, such a transformation would not preserve the energy momentum
tensor of the theory and its Lagrangian formulation. See
\cite{Olive:1995sw} and references therein for more details about this
point.
\end{remark}

\begin{proof}
For any $\cV\in \Omega^2_\cl(M,\R^{2n})$, we have:
\be
\ast_g \cV=-\cJ\cV \Longleftrightarrow \A_{\gamma,f} (\ast_g
\cV)=-\A_{\gamma,f}(\cJ\cV) \Longleftrightarrow \ast_{f_\ast(g)}
(\A_{\gamma,f}\cV)=-\cJ_{\gamma,f}\A_{\gamma,f}(\cV)~~.
\ee
\end{proof}

\noindent Consider the infinite rank vector bundle with total
space:
\be
\Sol_n(M)\eqdef \prod_{(g,\cJ)\in \Met_{3,1}(M)\times \fJ_n(M)}\Sol_n(M,g,\cJ)~~,
\ee
and infinite-dimensional base $B_n(M)\eqdef \Met_{3,1}(M)\times
\fJ_n(M)$, with the natural projection. Let $\sigma$ be the action of
$\Sp(2n,\R)\times \Diff(M)$ on $B_n(M)$ defined though
$\sigma=f_\ast\times \mu$, i.e.:
\be
\sigma(\gamma,f)(g,\cJ)=(f_\ast(g),\cJ_{\gamma,f})~~.
\ee
Then Proposition \ref{prop:SolAction} shows that the restriction of
$\A$ gives a linearization of $\sigma$ on the vector bundle
$\Sol_n(M)$. In particular, each fiber of $\Sol_n(M)$ carries a
linear representation of the isotropy group of the corresponding point
in the base. Let $\Iso(M,g)$ be the group of
orientation-preserving isometries of $(M,g)$. Then:
\be
\Stab_{\Sp(2n,\R)\times \Diff(M)}(g,\cJ)=\{(\gamma,f)\in \Sp(2n,\R)\times \Iso(M,g)~\vert~\cJ_{\gamma,f}=\cJ\}
\ee
and we have:

\begin{cor}
\label{cor:equivarianceccJlocal}
Let $(\gamma,f)\in \Sp(2n,\R)\times \Iso(M,g)$ such that $\cJ_{\gamma,f}
= \cJ$, i.e. $\cJ\circ f=\gamma\cJ \gamma^{-1}$. Then $\A_{\gamma,f}$ is a linear
automorphism of $\Sol_n(M,g,\cJ)$.
\end{cor}

\begin{definition}
\label{def:dualitygroupslocal}
Let $\cJ\in \fJ_n(M)$ be a taming map and $g$ be a Lorentzian metric on $M$.
\begin{itemize}
\item The group $\Sp(2n,\R)\times \Diff(M)$ is called the {\em
unbased pseudo-duality group}. The linear isomorphism:
\be
\A_{\gamma,f}\colon \Sol_n(M,g,\cJ) \xrightarrow{\sim} \Sol_n(M,f_{\ast}(g), \cJ_{\gamma,f})\, ,
\ee
induced by an element of this group is called a {\em unbased
pseudo-duality transformation}.
\item The group $\Sp(2n,\R)\times \Iso(M,g)$ is called the {\em
unbased duality group}. The linear isomorphism:
\be
\A_{\gamma,f}\colon \Sol_n(M,g,\cJ) \xrightarrow{\sim} \Sol_n(M,g, \cJ_{\gamma,f})\, ,
\ee
induced by an element $(\gamma,f)$ of this group is called a {\em unbased
duality transformation}.
\item The group $\Sp(2n,\R)$ is called the {\em duality group}. The
linear isomorphism:
\be
\A_{\gamma,\id_M}\colon \Sol_n(M,g,\cJ) \xrightarrow{\sim} \Sol_n(M,g, \cJ_\gamma)\, ,
\ee
where $\cJ_\gamma\eqdef \cJ_{\gamma,\id_M}=\gamma \cJ \gamma^{-1}$, is called a {\em classical duality transformation}.
\end{itemize}
\end{definition}

\begin{definition}
Let $\cJ\in \fJ_n(M)$ be a taming map and $g$ be a Lorentzian metric on $M$.
\begin{itemize}
\item The stabilizer:
\ben
\label{eq:unitary}
\cU(M,\cJ) \eqdef \left\{ (\gamma,f)\in \Diff(M,g)\times
  \Sp(2n,\R)\,\, \vert ~\cJ \circ f=\gamma \cJ \gamma^{-1} \right\}\, ,
\een
of $\cJ$ in $\Sp(2n,\R)\times \Diff(M)$ with respect to the
representation $\mu$ is called the {\em unbased unitary pseudo-duality
group}. The linear isomorphism:
\be
\A_{\gamma,f}\colon \Sol_n(M,g,\cJ) \xrightarrow{\sim} \Sol_n(M,f_{\ast}(g), \cJ)\, ,
\ee
induced by an element of this group is called an {\em unbased unitary
pseudo-duality transformation}. 
\item The stabilizer:
\ben
\cU(M,g,\cJ) \eqdef \left\{ (\gamma,f)\in \Sp(2n,\R)\times \Iso(M,g)\,\, \vert ~\cJ \circ f=\gamma \cJ \gamma^{-1} \right\}\, ,
\een
of $\cJ$ in $\Sp(2n,\R)\times \Iso(M)$ with respect to the
representation $\mu$ is called the {\em unbased unitary
duality group}. The linear isomorphism:
\be
\A_{\gamma,f}\colon \Sol_n(M,g,\cJ) \xrightarrow{\sim} \Sol_n(M,g, \cJ)\, ,
\ee
induced by an element of this group is called an {\em unbased unitary
duality transformation}.
\item The stabilizer:
\be
\U_\cJ(n) \eqdef \left\{ \gamma \in \Sp(2n,\R)\,\,
\vert \,\, \gamma \cJ \gamma^{-1} = \cJ \right\}\, .
\ee
of $\cJ$ in $\Sp(2n,\R)$ with respect to the action $\cJ\rightarrow
\gamma\cJ \gamma^{-1}$ is called the {\em unitary duality group}. The linear
automorphism:
\be
\A_{\gamma,f}\colon \Sol_n(M,g,\cJ) \xrightarrow{\sim} \Sol_n(M,g, \cJ)\, 
\ee
of $\Sol_n(M,g, \cJ)$ induced by an element of this group is called a
{\em unitary duality transformation}.
\end{itemize}
\end{definition}

\noindent
We have inclusions:
\be
\U_\cJ(n) \subset \cU(M,g,\cJ) \subset \cU(M,\cJ)
\ee
and short exact sequences:
\beqan
\label{eq:shortUduality}
& 1 \to \U_\cJ(n) \to \cU(M,g,\cJ) \to \Iso_\cJ(M,g)\to 1\, ,\\
& 1 \to \Iso(M,g)_\cJ \to \cU(M,\cJ) \to \Sp_\cJ(2n,\R) \to 1\, ,
\eeqan
where $\Iso_\cJ(M,g)$ is the subgroup of those $f\in
\Iso(M,g)$ for which there exists $\gamma\in \Sp(n,\R)$ such that
$\cJ\circ f=\gamma \cJ \gamma^{-1}$, while $\Sp_\cJ(2n,\R)$ is the subgroup of
those $\gamma\in \Sp(2n,\R)$ for which there exists $f\in \Iso(M,g)$
such that $\cJ\circ f=\gamma\cJ \gamma^{-1}$. Finally, the group:
\be
\Iso(M,g)_\cJ\eqdef \{f\in\Iso(M,g)~\vert~\cJ\circ f=\cJ\}
\ee
is the stabilizer of $\cJ$ in $\Iso(M,g)$. In particular, we have:

\begin{cor}
If $\U_\cJ(n) = 1$ then $\cU(M,g,\cJ) = \Iso_\cJ(M,g)$. If
$\Iso(M,g)_\cJ=1$ then $\cU(M,g,\cJ)=\Sp_\cJ(2n,\R)$.
\end{cor}


\subsection{Gluing local abelian gauge theories}
\label{sec:symplecticgluing}


Let now $M$ be an arbitrary oriented manifold admitting Lorentzian metrics.
Let $\cU=(U_{\alpha})_{\alpha\in I}$ be a good open cover of $M$, where $I$ is an
index set. Denote by $g_{\alpha}$ the restriction of $g$ to
$U_{\alpha}$. Roughly speaking, the definition of abelian gauge theory 
on $(M,g)$ given in Section \ref{sec:classical} is the result of gluing the
local $\U(1)^n$ abelian gauge theories defined on the contractible
Lorentzian four-manifolds $(U_{\alpha}, g_{\alpha})$ using
electromagnetic dualities. In order to implement this idea, we choose
a locally constant $\Sp(2n,\R)$-valued \u{C}ech cocycle for $\cU$:
\be
u_{\alpha \beta} \colon U_{\alpha}\cap U_{\beta} \to \Sp(2n,\R)
\ee
and a family of taming maps $\cJ_{\alpha}\colon U_{\alpha} \to
\R^{2n}$ for $\omega_{2n}$ such that:
\be
\cJ_{\beta} = u_{\alpha \beta}\, \cJ_{\alpha} u^{-1}_{\alpha \beta}
\ee
on double overlaps. The collection:
\be
\left\{(U_{\alpha})_{\alpha\in I}, (g_{\alpha})_{\alpha\in I} ,
(\cJ_{\alpha})_{\alpha\in I}, (u_{\alpha\beta})_{\alpha, \beta\in I}
\right\}\, ,
\ee
is equivalent to a flat symplectic vector bundle $(\cS,\omega,\cD)$
with symplectic form $\omega$ and symplectic flat connection $\cD$,
equipped with an almost complex structure $\cJ$ which is a taming of
$\omega$. A family $(\cV_{\alpha})_{\alpha\in I}$ of solutions of the
local abelian gauge theories defined by $(\cJ_{\alpha})_{\alpha\in I}$
on $(U_{\alpha}, g_{\alpha})$ which satisfies:
\be
\cV_{\beta} = u_{\alpha\beta} \cV_{\alpha}
\ee
corresponds to an $\cS$-valued two-form  $\cV\in \Omega^2(M,\cS)$ which obeys:
\be
\dd_\cD \cV = 0\, ,
\ee
where $\dd_\cD\colon \Omega^\ast(M,\cS)\to \Omega^\ast(M,\cS)$ is the
exterior differential twisted by $\cD$. This construction motivates
the global geometric model introduced in \cite{gesm} and further elaborated
in Section \ref{sec:classical}.


\section{Integral symplectic spaces and integral symplectic tori}
\label{app:symp}


This appendix recalls some notions from the theory of symplectic
lattices and symplectic tori which are used throughout the paper.
We also introduce the notion of integral symplectic torus. 

\begin{definition}
\label{def:integralsymplecticspace} An integral symplectic space is a
triple $(V,\omega,\Lambda)$ such that:
\begin{itemize}
\item $(V,\omega)$ is a finite-dimensional symplectic vector space
over $\R$.
\item $\Lambda\subset V$ is full lattice in $V$, i.e. a lattice in $V$
such that $V = \Lambda\otimes_\Z \R$.
\item $\omega$ is integral with respect to $\Lambda$, i.e. we have
  $\omega(\Lambda,\Lambda)\subset \Z$.
\end{itemize}
\end{definition}

\noindent An isomorphism of integral symplectic spaces $f\colon (V_1 ,
\omega_1 , \Lambda_1) \to (V_2 , \omega_2 , \Lambda_2)$ is a bijective
symplectomorphism from $(V_1 , \omega_1)$ to $(V_2 , \omega_2)$ which
satisfies:
\be
f(\Lambda_1) =\Lambda_2\, .
\ee
Denote by $\Symp_\Z$ the groupoid of integral symplectic spaces and
isomorphisms of such. Let $\Aut(V)$ be the group of linear
automorphisms of the vector space $V$ and $\Sp(V,\omega)\subset
\Aut(V)$ be the subgroup of symplectic transformations. Then the
automorphism group of the integral symplectic space
$(V,\omega,\Lambda)$ is denoted by:
\be
\Sp(V,\omega,\Lambda) = \left\{T\in \Sp(V,\omega)\,\,\vert\,\, T(\Lambda) = \Lambda \right\}\, .
\ee 

\begin{definition}
An \emph{integral symplectic basis} of a $2n$-dimensional integral
symplectic space $(V,\omega,\Lambda)$ is a basis $\cE=
(\xi_1,\hdots , \xi_n, \zeta_1, \hdots ,\zeta_n)$ of $\Lambda$ (as
a free $\Z$-module) such that:
\be
\omega(\xi_i , \xi_j) = \omega(\zeta_i , \zeta_j) = 0\, , \qquad
\omega(\xi_i , \zeta_j) =  t_i \delta_{ij}~~,~~\omega(\zeta_i , \xi_j) = - t_i \delta_{ij}\, , \qquad \forall\,\, i,j
= 1,\hdots , n\, ,
\ee
where $t_1 , \hdots , t_n \in \Z$ are strictly positive
integers satisfying the divisibility conditions:
\be
t_1\vert t_2\vert\hdots \vert t_n\, .
\ee
\end{definition}

\noindent By the \emph{elementary divisor theorem}, see \cite[Chapter
VI]{Debarre}, every integral symplectic space admits an integral
symplectic basis and the positive integers $t_1,\ldots, t_n$ (which
are called the {\em elementary divisors} of $(V,\omega,\Lambda)$) do
not depend on the choice of such a basis. Define:
\be
\Div^n \eqdef \left\{ (t_1, \hdots ,t_n) \in
\Z_{>0}^n\,\, \vert\,\, t_1\vert t_2\vert \hdots \vert t_n
\right\}\, ,
\ee
and:
\be
\delta(n) \eqdef (1,\hdots , 1)\in \Div^n\, .
\ee
Let $\leq$ be the partial order relation on $\Div^n$ defined through:
\be
(t_1, \hdots ,t_n)\leq (t'_1, \hdots ,t'_n)~~\mathrm{iff}~~ t_i\vert t'_i~~ \forall i=1,\ldots, n~~.
\ee
Then $\delta(n)$ is the least element of the ordered set
$(\Div^n,\leq)$. Notice that this ordered set is directed, since any
two elements $t,t'\in \Div^n$ have an upper bound given by
$(t_1t'_1,\ldots, t_nt'_n)$.  In fact, $(\Div^n,\leq)$ is a lattice
with join and meet given by:
\be
t\vee t'=(\lcm(t_1,t'_1),\ldots, \lcm(t_n,t'_n))~~,~~t\wedge t'=(\gcd(t_1,t'_1),\ldots, \gcd(t_n,t'_n))~~.
\ee
This lattice is semi-bounded from below with bottom element given by
$\delta(n)$ and it is complete for meets (i.e., it is a complete meet
semi-lattice).

\begin{definition}
\label{def:type}
The \emph{type} of an integral symplectic space $(V,\omega,\Lambda)$
is the ordered system of elementary divisors of $(V,\omega,\Lambda)$,
which we denote by:
\be
\frt(V,\omega,\Lambda) = (t_1 , \hdots , t_n)\in \Div^n\, .
\ee 
The integral symplectic space $(V,\omega,\Lambda)$ is called \emph{principal} if:
\be
\frt(V,\omega,\Lambda) = \delta(n) \in \Div^n\, .
\ee
\end{definition}

\noindent Let $\omega_{2n}$ denotes the standard symplectic pairing on $\R^{2n}$. 

\begin{prop}
Two integral symplectic spaces have the same type if and only if they are
isomorphic. Moreover, every element of $\Div^n$ is the type of an
integral symplectic space. Hence the type induces a bijection between the
set of isomorphism classes of integral symplectic spaces and the set
$\Div^n$.
\end{prop}

\begin{proof}
The first statement is obvious. For the second statement, fix $\frt
\eqdef (t_1 , \hdots , t_n) \in\Div^n$. Consider the full lattice
$\Lambda_\frt \subseteq \R^{2n}$ defined as follows:
\ben
\label{Lambdafrt}
\Lambda_{\frt} \eqdef \left\{ (l_1 , \hdots ,l_n , t_1 l_{n+1}
, \hdots , t_n l_{n+1})\,\, \vert\,\, l_1 , \hdots , l_n
\in\Z\right\}\, .
\een
Then $(\R^{2n} , \omega_{2n} , \Lambda_{\frt})$ 
is an integral symplectic space of type $\frt$.
\end{proof}

\begin{definition}
The lattice $\Lambda_{\frt}$ defined in \eqref{Lambdafrt} is called
the \emph{standard symplectic lattice} of type $\frt$ and
$(\R^{2n},\omega_{2n},\Lambda_\frt)$ is called the {\em standard
integral symplectic space} of type $\frt$.
\end{definition}

\noindent We have $\Lambda_{\delta(n)} =\Z^{2n}$. Moreover,
$\Lambda_\frt$ is a sub-lattice of $\Z^{2n}$ and we have
$\Z^{2n}/\Lambda_\frt\simeq \Z_{t_1}\times \ldots \times \Z_{t_n}$ for
all $\frt\in \Div(n)$. For $\frt,\frt'\in \Div^n$, we have
$\Lambda_{\frt'}\subset \Lambda_{\frt}$ if and only if $\frt \leq \frt'$.  The
lattice $\Lambda_\frt$ admits the basis:
\beqa
&&\xi_1=e_1=(1,0,\ldots,0),\ldots, \xi_n=e_n=(0,\ldots,0,1,0,\ldots,0)\nn\\
&&\zeta_1=t_1 f_1=(0,\ldots, 0, t_1,0,\ldots, 0),\ldots, \zeta_n=t_n
f_n=(0,\ldots,0,t_n)~~,
\eeqa
in which the standard symplectic form of $\R^{2n}$ has coefficients:
\beqa
&&\omega_{2n}(\xi_i,\xi_j)=\omega_{2n}(\zeta_i,\zeta_j)=0~~\nn\\
&&\omega_{2n}(\xi_i,\zeta_j)= t_i\delta_{ij}~~,~~\omega_{2n}(\zeta_i,\xi_j)=-t_i\delta_{ij}~~.
\eeqa
The isomorphism which takes $\xi_i$ to $e_i$ and $\zeta_j$ to $f_j$
identifies $(\R^{2n},\omega_{2n},\Lambda_\frt)$ with the integral
symplectic space $(\R^{2n},\omega_\frt,\Z^{2n})$, where $\omega_\frt$
is the symplectic pairing defined on $\R^{2n}$ by:
\ben
\label{omegat}
\omega_\frt(e_i,e_j)=\omega_\frt(f_i,f_j)=0\, , \quad \omega_\frt(e_i,f_j)= \delta_{ij}\, , \quad \omega_\frt(f_i,e_j)=- \delta_{ij}\, , \quad \forall \,\, i=1,\ldots, n~~.
\een
Given $\frt = (t_1 , \hdots , t_n)\in \Div^n$, consider the diagonal $n\times n$ matrix:
\be
\mathrm{D}_{\frt} \eqdef \mathrm{diag}(t_1 , \hdots , t_n) \in \Mat(n,\Z)\, ,
\ee
as well as:
\be
\Gamma_{\frt} \eqdef 
\begin{pmatrix} 
\mathrm{I}_n & 0 \\
0 & \mathrm{D}_{\frt} 
\end{pmatrix} \in \Mat(2n,\Z) \, .
\ee

\begin{definition}
\label{def:Siegel}
The {\em modified Siegel modular group} of type $\frt \in \Div^n$ is
the subgroup of $\Aut(\R^{2n},\omega_{2n})\simeq \Sp(2n,\R)$ defined
through:
\be
\Sp_{\frt}(2n,\Z) \eqdef\left\{ T\in \Aut(\R^{2n},\omega_{2n}) \,\,
\vert \,\, T(\Lambda_\frt)=\Lambda_\frt\right\} \simeq \left\{ T\in
\Sp(2n,\R) \,\, \vert \,\, \Gamma_{\frt} T \Gamma_{\frt}^{-1}=T
\right\}\, .
\ee
\end{definition}

\noindent Since $(\R^{2n},\omega_{2n},\Lambda_\frt)\simeq
(\R^{2n},\omega_\frt,\Z^{2n})$, we have $\Sp_\frt(2n,\Z)\simeq
\Aut(\R^{2n},\omega_\frt,\Z^{2n})$. Hence $\Sp_\frt(2n,\Z)$ is a
subgroup of $\GL(2n,\Z)$. The remarks above give:

\begin{prop}{\rm \cite[Proposition F.12]{gesm}}
Let $(V,\omega,\Lambda)$ be an integral symplectic space of dimension
$2n$. Any integral symplectic basis of this space induces an
isomorphism of integral symplectic spaces between $(V,\omega,\Lambda)$
and $(\R^{2n},\omega_{2n},\Lambda_{\frt})$ as well as an isomorphism
of groups between $\Sp(V,\omega,\Lambda)$ and $\Sp_{\frt}(2n,\Z)$.
\end{prop}

\noindent We have $\Sp_{\delta(n)}(2n,\Z) = \Sp(2n,\Z)$ and
$\Sp_{\frt}(2n,\Z)\subseteq \Sp_{\frt'}(2n,\Z)$ when $\frt\leq \frt'$.
Hence $\Sp_\frt(2n,\Z)$ forms a direct system of groups and we have
$\Sp(2n,\Z)\subset \Sp_\frt(2n,\Z)$ for all $\frt\in \Div^n$. The
direct limit $\varinjlim_{\frt\in \Div^n}\Sp_\frt(2n,\Z)$ identifies
with the following subgroup of $\Sp(2n,\R)$:
\be
\Sp_\infty(2n,\Z)\eqdef \{T\in \GL(2n,\R)\,\, \vert\,\, \exists \,\,\frt_T\in \Div^n : T\in \Sp_{\frt_T}(2n,\Z) \}\, ,
\ee
through the isomorphism of groups
$\varphi:\Sp_\infty(2n,\Z)\rightarrow \varinjlim_{\frt\in
\Div^n}\Sp_\frt(2n,\Z)$ which sends $T\in \Sp_\infty(2n,\Z)$ to the
equivalence class $[\alpha(T)]\in \varinjlim_{\frt\in
\Div^n}\Sp_\frt(2n,\Z)$ of the family $\alpha(T)\in \sqcup_{\frt\in
\Div^n} \Sp_\frt(2n,\Z)$ defined through:
\be
\alpha(T)_\frt\eqdef \twopartdef{T}{\frt_T\leq \frt}{1}{\frt_T \nleq \frt}\, .
\ee
Notice that $\Sp(2n,\Z)=\Sp_{\delta(n)}(2n,\Z)$ is a subgroup of $\Sp_\infty(2n,\Z)$.

\begin{definition}
The {\em type} of an element $T\in \Sp_\infty(2n,\Z)$ is defined as the
greatest lower bound $\frt(T)\in \Div^n$ of the finite set $\{\frt\in
\Div^n \vert \frt| \frt_T\}$, where $\frt_T$ is any element of
$\Div^n$ such that $T\in \Sp_\frt(2n,\Z)$.
\end{definition}

\noindent Notice that the type of $T$ is well-defined and that we have 
$T\in \Sp_{\frt(T)}(2n,\Z)$.

\subsection{Integral symplectic tori}

\noindent The following definition distinguishes between a few closely
related notions.

\begin{definition}
A $d$-dimensional {\em torus} is a smooth manifold $\rT$ diffeomorphic
with the {\em standard $d$-torus} $\rT^d\eqdef (\rS^1)^d$.  A
$d$-dimensional {\em torus group} is a compact abelian Lie group $A$
which is isomorphic with the {\em standard $d$-dimensional torus
group} $\U(1)^d$ as a Lie group. A $d$-dimensional {\em affine torus}
is a principal homogeneous space $\fA$ for a $d$-dimensional torus
group. The {\em standard affine $d$-torus} is the $d$-dimensional
affine torus $\fA_d$ defined by the right action of $\U(1)^d$ on
itself.
\end{definition}

\noindent The underlying manifold of the standard $d$-dimensional
torus group is the standard $d$-torus while the underlying manifold of
a $d$-dimensional torus group is a $d$-torus. Moreover, any
$d$-dimensional affine torus group is isomorphic with a standard
affine $d$-torus. The transformations of an affine torus given the
right action of its underlying group will be called {\em
translations}. Choosing a distinguished point in any affine torus
makes it into a torus group having that point as zero
element. Conversely, any torus group defines an affine torus obtained
by `forgetting' its zero element. The singular homology and cohomology
groups of a $d$-torus $\rT$ are the free abelian groups given by:
\beqa
&& H_k(\rT,\Z)\simeq \wedge^k H_1(\rT,\Z)\\
&& H^k(\rT,\Z)\simeq \wedge^k H^1(\rT,\Z) \simeq \wedge^k H_1(\rT,\Z)^\vee
\eeqa
for all $k=0,\ldots, d$, where $H_1(\rT,\Z)\simeq H^1(\rT,\Z)\simeq
\Z^d$. The underlying torus group of any affine torus $\fA$ is
isomorphic with $A\eqdef H_1(\fA,\R)/H_1(\fA,\Z)$. The group
of automorphisms of a $d$-dimensional torus group $A$ is given by:
\be
\Aut(A)=\Aut(H_1(A,\R), H_1(A,\Z))\simeq \GL(d,\Z)\, .
\ee
Note that the group of automorphisms of the $d$-dimensional affine torus
is isomorphic to $ \U(1)^d\rtimes \Aut(\U(1)^d)\simeq \U(1)^d\rtimes \GL(2n,\Z)$,
where $\GL(2n,\Z)$ acts on $\U(1)^d$ through the morphism of groups
$\rho\colon \GL(2n,\Z)\rightarrow \Aut(A)$ given by:
\ben
\label{Taction}
\rho(T)(\exp(2\pi \i x))=\exp(2\pi \i T(x))\, , \quad \forall\,\, T\in
\GL(d,\Z)\, , \quad \forall\,\, x\in \R^d~~.
\een
Here $\exp:\R^d\rightarrow \U(1)^d$ is the exponential map of
$\U(1)^d$, which is given by:
\be
\exp(v)=(e^{v_1},\ldots, e^{v_d})\, , \quad \forall\,\, v=(v_1,\ldots, v_d)\in \R^d\, .
\ee

\begin{definition}
Let $\rT$ be a torus of dimension at least two. A subtorus $\rT'\subset \rT$
is called {\em primitive} if $H_1 (\rT',\Z)$ is a primitive sub-lattice of
$H_1(\rT,\Z)$, i.e. if the abelian group $H_1(\rT,\Z)/H_1(\rT',\Z)$ is
torsion-free.
\end{definition}

\begin{definition}
A {\em symplectic torus} is a pair $\bT=(\rT,\Omega)$, where $\rT$ is
an even-dimensional torus and $\Omega$ is a symplectic form defined on
$\rT$.  A {\em symplectic torus group} is a pair $\bA=(A,\Omega)$,
where $A$ is an even-dimensional torus group and $\Omega$ is a
symplectic form defined on the underlying torus which is
invariant under translations by all elements of $A$. An {\em affine
symplectic torus} is a pair $\bfA=(\fA,\Omega)$, where $\fA$ is an
even-dimensional affine torus and $\Omega$ is a symplectic form on
$\fA$ which is invariant under translations.
\end{definition}

\begin{definition}
A symplectic torus $\bT=(\rT,\Omega)$ is called {\em integral} if the
symplectic area $\int_{T'}\Omega$ of any of its primitive
two-dimensional subtori $T'$ is an integer. 
\end{definition}

\noindent Let $(\rT,\Omega)$ be a symplectic torus. The cohomology
class of $\Omega$ is a non-degenerate element $\omega\in
H^2(\rT,\R)\simeq \wedge^2 H_1(\rT,\R)^\vee$, i.e a symplectic pairing
on the vector space $H_1(\rT,\R)$. The symplectic torus $(T,\Omega)$
is integral if and only if the triplet $(H_1(\rT,\R),H_1(\rT,\Z),\omega)$ is an
integral symplectic space. In this case, $\omega$ descends to a
symplectic form ${\hat \Omega}$ which makes $H_1(\rT,\R)/H_1(\rT,\Z)$
into an integral symplectic torus group. If $\bfA=(\fA,\Omega)$ is an integral
affine symplectic torus, then $\Omega$ is determined by its
cohomology class $\omega$, hence $\bfA$ can also be viewed as a pair
$(\fA,\omega)$ where $\fA$ is an affine torus and $\omega$ is a
symplectic form on $H_1(\fA,\R)$ which is integral for the lattice
$H_1(\fA,\Z)$. In this case, any choice of a point in $\fT$ allows
us to identify $\bfA$ with the integral symplectic torus group
$(H_1(\fA,\R)/H_1(\fA,\Z),{\hat \Omega})$.

Let $(\R^{2n},\omega_{2n},\Lambda_\frt)$ be the standard
integral symplectic space of type $\frt\in \Div^n$
and $\Omega_\frt$ be the translationally invariant symplectic form
induced by $\omega_{2n}$ on the torus group
$\R^{2n}/\Lambda_\frt$. Then the symplectic torus group
$\left(\R^{2n}/\Lambda_\frt,\Omega_\frt\right)$ is integral.

\begin{definition}
The $2n$-dimensional {\em standard integral symplectic torus group} of type
$\frt\in \Div^n$ is:
\be
\bA_\frt\eqdef \left(\R^{2n}/\Lambda_\frt,\Omega_\frt\right)~~.
\ee
The integral affine symplectic torus $\bfA_\frt$ obtained from 
$\bA_\frt$ by forgetting the origin is the {\em standard integral affine 
symplectic torus} of type $\frt$.
\end{definition}

\noindent 
Note that every integral affine symplectic torus $\bfA$ is affinely
symplectomorphic to a standard affine symplectic torus $\bfA_\frt$,
whose type $\frt$ is uniquely-determined and called the {\em type} of
$\bfA$.  Similarly, every integral symplectic torus group $\bA$ is
isomorphic with a standard integral symplectic torus group $\bA_\frt$
whose type $\frt$ is uniquely determined by $\bA$ and called the {\em type}
of $\bA$. The group of automorphisms of $\bA_\frt$ for $\frt \in
\Div^n$ is given by:
\be
\Aut(\bA_\frt)=\Sp_\frt(2n,\Z)~~,
\ee
while the group of automorphisms of an integral symplectic affine torus
$\bfA_\frt =(\fA_\frt ,\Omega)$ of type $\frt\in \Div^n$ is:
\be
\Aut(\bfA_\frt)=A\rtimes \Sp_\frt(2n,\Z)~~,
\ee
where $A=H_1(\fA,\R)/H_1(\fA,\Z)$ is the underlying torus group of $\fA$
and the action of $\Sp_\frt(2n,\Z)\subset \GL(2n,\Z)\simeq \Aut(A)$ on $A$
coincides with that induced from the action on $H_1(\fA,\R)\simeq \R^{2n}$. We 
denote by $\Aff_\frt\eqdef \Aut(\bfA_\frt)$ the automorphism group of 
the integral affine symplectic torus of type $\frt$. We have:
\be
\Aff_\frt=\U(1)^{2n}\rtimes  \Sp_\frt(2n,\Z)~~,
\ee
where $\Sp_\frt(2n,\Z)\subset \GL(2n,\Z)$ acts on $\U(1)^{2n}$ through
the restriction of \eqref{Taction}.

\begin{definition}
Let $\bA=(A,\Omega)$ and $\bA'=(A',\Omega')$ be two integral symplectic
torus groups. A {\em symplectic isogeny} from $\bA$ to $\bA'$ is a
surjective morphism of groups $f:\bA\rightarrow \bA'$ with finite
kernel such that $f^\ast(\Omega')=\Omega$.
\end{definition}

\noindent The following statement is immediate. 

\begin{prop}
\label{prop:isogenies}
Let $\frt,\frt'\in \Div^n$ be such that $\frt\leq \frt'$, namely
$t'_i=q_i t_i$ (where $q_i\in \Z_{>0}$) for all $i=1,\ldots,n$. Then
the map $f:\bA_{\frt'} \rightarrow \bA_\frt$ defined through:
\be
f(x+\Lambda_{\frt'})=x+\Lambda_{\frt}~~\forall x\in \R^{2n}
\ee
is a symplectic isogeny whose kernel is given by:
\be
\ker(f)\simeq \Z_{q_1}\times \ldots \times \Z_{q_n}~~.
\ee
\end{prop}

\noindent In particular, $\bA_\frt$ is isogenous with
$\bA_{\delta(n)}$ for all $\frt\in \Div^n$.

\subsection{Tamings}

\begin{definition}
A \emph{tamed integral symplectic space} is a quadruple
$(V,\omega,\Lambda,J)$, where $(V,\omega,\Lambda)$ is an integral
symplectic space and $J$ is a taming of the symplectic space
$(V,\omega)$. The \emph{type} of a tamed integral symplectic space is
the type of its underlying integral symplectic space.
\end{definition}

\noindent Given a tamed integral symplectic space
$(V,\omega,\Lambda,J)$ of type $\frt\in \Div^n$, the taming $J$ makes
$V$ into a $n$-dimensional complex vector space, which we denote by
$V_J$. The symplectic pairing $\omega$ induces a Kahler form $\Omega$
which makes the complex torus $V_J/\Lambda$ into a (generally
non-principal) polarized abelian variety whose underlying symplectic
torus coincides with $\bA_\frt$.  We refer the reader to \cite[Appendix
F]{gesm} for details on the relation between tamed integral symplectic
spaces and (generally non-principal) abelian varieties.

\end{document}